\documentclass[11pt,english]{extarticle}
\usepackage[]{fontenc}
\usepackage[latin9]{inputenc}
\usepackage[letterpaper]{geometry}
\usepackage{mathrsfs}
\usepackage{amsmath}
\usepackage{amsthm}
\usepackage{amssymb}
\usepackage{stackrel}
\usepackage{graphicx}
\usepackage{esint}

\makeatletter
  \theoremstyle{remark}
  \newtheorem{rem}{\protect\remarkname}
\theoremstyle{plain}
\newtheorem{thm}{\protect\theoremname}
  \theoremstyle{definition}
  \newtheorem{defn}{\protect\definitionname}
  \theoremstyle{plain}
  \newtheorem{lem}{\protect\lemmaname}
 \theoremstyle{definition}
  \newtheorem{example}{\protect\examplename}

\usepackage{dsfont}
\usepackage{relsize}
\usepackage{subdepth}
\allowdisplaybreaks

\providecommand{\definitionname}{Definition}
\providecommand{\examplename}{Example}

\providecommand{\lemmaname}{Lemma}
\providecommand{\remarkname}{Remark}
\providecommand{\theoremname}{Theorem}

\usepackage{url}

\DeclareFontFamily{OT1}{pzc}{}
\DeclareFontShape{OT1}{pzc}{m}{it}{<-> s * [1.200] pzcmi7t}{}
\DeclareMathAlphabet{\mathpzc}{OT1}{pzc}{m}{it}

\usepackage[noadjust]{cite}
\usepackage{filecontents}

\usepackage{stfloats}

\renewcommand\footnotemark{}

\@ifundefined{showcaptionsetup}{}{%
 \PassOptionsToPackage{caption=false}{subfig}}
\usepackage{subfig}
\makeatother

\usepackage[english]{babel}
  \providecommand{\definitionname}{Definition}
  \providecommand{\examplename}{Example}
  \providecommand{\lemmaname}{Lemma}
  \providecommand{\remarkname}{Remark}
\providecommand{\theoremname}{Theorem}

\begin{document}

\title{\textbf{Grid Based Nonlinear Filtering Revisited:}\\
\textbf{Recursive Estimation \& Asymptotic Optimality}}

\author{Dionysios S. Kalogerias\thanks{The Authors are with the Department of Elctrical \& Computer Engineering,
Rutgers, The State University of New Jersey, 94 Brett Rd, Piscataway,
NJ 08854, USA. e-mail: \{d.kalogerias, athinap\}@rutgers.edu.}\thanks{This work is supported by the National Science Foundation (NSF) under
Grants CCF-1526908 \& CNS-1239188.} and Athina P. Petropulu}

\date{April, 2016}

\maketitle
\textbf{\vspace{-30pt}}
\begin{abstract}
We revisit the development of grid based recursive approximate filtering of
general Markov processes in discrete time, partially observed in conditionally Gaussian
noise. The grid based filters considered rely on two types of state
quantization: The \textit{Markovian} type and the \textit{marginal}
type. We propose a set of novel, relaxed sufficient conditions, ensuring
strong and fully characterized pathwise convergence of these filters
to the respective MMSE state estimator. In particular, for marginal
state quantizations, we introduce the notion of \textit{conditional
regularity of stochastic kernels}, which, to the best of our knowledge,
constitutes the most relaxed condition proposed, under which asymptotic
optimality of the respective grid based filters is guaranteed. Further,
we extend our convergence results, including filtering of bounded
and continuous functionals of the state, as well as recursive approximate
state prediction. For both Markovian and marginal quantizations, the
whole development of the respective grid based filters relies more
on linear-algebraic techniques and less on measure theoretic arguments,
making the presentation considerably shorter and technically simpler.
\end{abstract}
\textbf{\textit{$\quad$}}\textbf{Keywords.} Nonlinear Filtering,
Grid Based Filtering, Approximate Filtering, Markov Chains, Markov
Processes, Sequential Estimation, Change of Probability Measures.

\textbf{\vspace{-21pt}
}

\section{Introduction}

It is well known that except for a few special cases \cite{Segall_Point1976,Marcus1979,Elliott1994Exact,Elliott1994_HowToCount,Daum2005},
general nonlinear filters of partially observable Markov processes
(or Hidden Markov Models (HMMs)) do not admit finite dimensional (recursive)
representations \cite{Segall1976,Elliott1994Hidden}. Nonlinear filtering
problems, though, arise naturally in a wide variety of important applications,
including target tracking \cite{Tracking_1_2002,Tracking_2_2004},
localization and robotics \cite{Roumeliotis_2000,Volkov_2015}, mathematical
finance \cite{OXFORD_Crisan_2011} and channel prediction in wireless
sensor networks \cite{KalPetChannelMarkov2014}, just to name a few.
Adopting the Minimum Mean Square Error (MMSE) as the standard optimality
criterion, in most cases, the nonlinear filtering problem results
in a dynamical system in the infinite dimensional space of measures,
making the need for robust approximate solutions imperative.

Approximate nonlinear filtering methods can be primarily categorized
into two major groups \cite{Chen_FILTERS_2003}: \textit{local} and
\textit{global}. Local methods include the celebrated extended Kalman
filter \cite{Elliott2010620}, the unscented Kalman filter \cite{UNSCENTED},
Gaussian approximations \cite{ItoXiong1997}, cubature Kalman filters
\cite{Cubature_2009} and quadrature Kalman filters \cite{Elliott_GAUSS}.
These methods are mainly based on the local ``assumed form of the
conditional density'' approach, which dates back to the 1960's \cite{Kushner1967_Approximations}.
Local methods are characterized by relatively small computational
complexity, making them applicable in relatively higher dimensional
systems. However, they are strictly suboptimal and, thus, they at
most constitute efficient heuristics, but \textit{without explicit
theoretical guarantees}. On the other hand, global methods, which
include grid based approaches (relying on proper quantizations of
the state space of the state process \cite{Pages2005optimal,Kushner2001_BOOK,Kushner2008})
and Monte Carlo approaches (particle filters and related methods \cite{PARTICLE2002tutorial}),
provide approximations to the \textit{whole posterior measure of the
state}. Global methods possess very powerful asymptotic optimality
properties, providing explicit theoretical guarantees and predictable
performance. For that reason, they are very important both in theory
and practice, either as solutions, or as benchmarks for the evaluation
of suboptimal techniques. The main common disadvantage of global methods
is their high computational complexity as the dimensionality of the
underlying model increases. This is true both for grid based and particle
filtering techniques \cite{Bengtsson2008Curse,Quang2010Insight,RebeschiniRamon2013_1,Rebeschini2014Nonlinear,Sellami_2008_Compare}.

In this paper, we focus on \textit{grid based approximate filtering
of Markov processes observed in conditionally Gaussian noise, constructed
by exploiting uniform quantizations of the state}. Two types of state
quantizations are considered: the \textit{Markovian} and the \textit{marginal
}ones (see \cite{Pages2005optimal} and/or Section \ref{sec:Uniform-State-Quantizations}).
Based on existing results \cite{Elliott1994Hidden,Chen_FILTERS_2003,Pages2005optimal},
one can derive grid based, recursive nonlinear filtering schemes,
exploitting the properties of the aforementioned types of state approximations.
The novelty of our work lies in the development of an original convergence
analysis of those schemes, under generic assumptions on the expansiveness
of the observations (see Section \ref{sec:Quant_Filtering}). Our
contributions can be summarized as follows:

\textbf{1)} For marginal state quantizations, we propose the notion
of \textit{conditional regularity of Markov kernels} (Definition 2),
which is an easily verifiable condition for guaranteeing strong asymptotic
consistency of the resulting grid based filter. Conditional regularity
is a simple and relaxed condition, in contrast to more complicated
and potentially stronger conditions found in the literature, such
as the Lipschitz assumption imposed on the stochastic kernel(s) of
the underlying process in \cite{Pages2005optimal}.

\textbf{2)} Under certain conditions, we show that all grid based
filters considered here converge to the true optimal nonlinear filter
in a strong and controllable sense (Theorems \ref{Markovian_Filter}
and \ref{OUR_FILTER2}). In particular, the convergence is compact
in time and uniform in a measurable set occurring with probability
almost $1$; this event is completely characterized in terms of the
filtering horizon and the dimensionality of the observations.

\textbf{3)} We show that all our results can be easily extended in
order to support filters of \textit{functionals} of the state and
recursive, grid based approximate prediction (Theorem 5). More specifically,
we show that grid based filters are asymptotically optimal as long
as the state functional is bounded and continuous; this is a typical
assumption (see also \cite{Elliott1994Hidden,Kushner2008,Crisan2002Survey}).
Of course, this latter assumption is in addition to and independent
from any other condition (e.g., conditional regularity) imposed on
the structure of the partially observable system under consideration.
In a companion paper \cite{KalPetChannelMarkov2014}, this simple
property has been proven particularly useful, in the context of channel
estimation in wireless sensor networks. The assumption of a bounded
and continuous state functional is more relaxed as compared to the
respective bounded and Lipschitz assumption found in \cite{Pages2005optimal}.

Another novel aspect of our contribution is that our original theoretical
development is based more on linear-algebraic arguments and less on
measure theoretic ones, making the presentation shorter, clearer and
easy to follow.

\textbf{\vspace{-20pt}
}

\subsection*{Relation to the Literature}

In this paper, conditional regularity is presented as a relaxed sufficient
condition for asymptotic consistency of discrete time grid based filters,
employing marginal state quantizations. Another set of conditions
ensuring asymptotic convergence of state approximations to optimal
nonlinear filters are the Kushner's local consistency conditions (see,
example, \cite{Kushner2008,Kushner2001_BOOK}). These refer to Markov
chain approximations for continuous time Gaussian diffusion processes
and the related standard nonlinear filtering problem. 

It is important to stress that, as it can be verified in Section IV,
the constraints which conditional regularity imposes on the stochastic
kernel of the hidden Markov process under consideration are \textit{general
and do not require} the assumption of any specific class of hidden
models. In this sense, conditional regularity is a \textit{nonparametric
condition} for ensuring convergence to the optimal nonlinear filter.
For example, hidden Markov processes driven by strictly non-Gaussian
noise are equally supported as their Gaussian counterparts, provided
the same conditions are satisfied, as suggested by conditional regularity
(see Section IV). Consequently, it is clear that conditional regularity
advocated in this paper is different in nature than Kushner's local
consistency conditions \cite{Kushner2008,Kushner2001_BOOK}. In fact,
putting the differences between continuous and discrete time aside,
conditional regularity is more general as well.

Convergence of discrete time approximate nonlinear filters (not necessarily
recursive) is studied in \cite{KalPetNonlinear_2015}. No special
properties of the state are assumed, such as the Markov property;
it is only assumed that the state is almost surely compactly supported.
In this work, the results of \cite{KalPetNonlinear_2015} provide
the tools for showing asymptotic optimality of grid base, \textit{recursive}
approximate estimators. Further, our results have been leveraged in
\cite{KalPetChannelMarkov2014,KalPet_SPAWC_2015}, showing asymptotic
consistency of sequential spatiotemporal estimators/predictors of
the magnitude of the wireless channel over a geographical region,
as well as its variance. The estimation is based on limited channel
observations, obtained by a small number of sensors.

The paper is organized as follows. In Section II, we define the system
model under consideration and formulate the respective filtering approximation
problem. In Section III, we present useful results on the asymptotic
characterization of the Markovian and marginal quantizations of the
state (Lemmata \ref{Marginal_Convergence} and \ref{Markovian_Convergence}).
Exploiting these results, Section IV is devoted to: \textbf{(a)} Showing
convergence of the respective (\textit{not necessarily finite dimensional})
grid based filters (Theorem \ref{CONVERGENCE_THEOREM_2}). \textbf{(b)}
Derivation of the respective recursive, asymptotically optimal filtering
schemes, based on the Markov property and any other conditions imposed
on the state (Theorem \ref{Markovian_Filter} and Lemmata \ref{Doob_Lemma-1}
and \ref{Convergence_FAKE}, leading to Theorem \ref{OUR_FILTER2}).
Extensions to the main results are also presented (Theorem 5), and
recursive filter performance evaluation is also discussed (Theorem
6). Some analytical examples supporting our investigation are discussed
in Section V, along with some numerical simulations. Finally, Section
VI concludes the paper.

\textit{Notation}: In the following, the state vector will be represented
as $X_{t}$, its innovative part as $W_{t}$ (if exists), its approximations
as $X_{t}^{L_{S}}$, and all other matrices and vectors, either random
or not, will be denoted by boldface letters (to be clear by the context).
Real valued random variables will be denoted by uppercase letters.
Calligraphic letters and formal script letters will denote sets and
$\sigma$-algebras, respectively. For any random variable (same for
vector) $Y$, $\sigma\left\{ Y\right\} $ will denote the $\sigma$-algebra
generated by $Y$. The essential supremum (with respect to some measure
- to be clear by the context) of a function $f\left(\cdot\right)$
over a set ${\cal A}$ will be denoted by $\mathrm{ess}\hspace{0.2em}\mathrm{sup}_{x\in{\cal A}}f\left(x\right)$.
The operators $\left(\cdot\right)^{\boldsymbol{T}}$, $\lambda_{min}\left(\cdot\right)$
and $\lambda_{max}\left(\cdot\right)$ will denote transposition,
minimum and maximum eigenvalue, respectively. The $\ell_{p}$-norm
of a vector $\boldsymbol{x}\in\mathbb{R}^{n}$ is $\left\Vert \boldsymbol{x}\right\Vert _{p}\triangleq\left(\sum_{i=1}^{n}\left|x\left(i\right)\right|^{p}\right)^{1/p}$,
for all naturals $p\ge1$. For any Euclidean space $\mathbb{R}^{N}$,
${\bf I}_{N}$ will denote the respective identity operator. For \textit{collections
of sets} $\left\{ {\cal A},{\cal B}\right\} $ and $\left\{ {\cal C},{\cal D}\right\} $,
the usual Cartesian product is overloaded by defining $\left\{ {\cal A},{\cal B}\right\} \times\left\{ {\cal C},{\cal D}\right\} \triangleq\left\{ {\cal A}\times{\cal C},{\cal A}\times{\cal D},{\cal B}\times{\cal C},{\cal B}\times{\cal D}\right\} $.
Additionally, we employ the identifications $\mathbb{R}_{+}\equiv\left[0,\infty\right)$,
$\mathbb{R}_{++}\equiv\left(0,\infty\right)$, $\mathbb{N}^{+}\equiv\left\{ 1,2,\ldots\right\} $,
$\mathbb{N}_{n}^{+}\equiv\left\{ 1,2,\ldots,n\right\} $ and $\mathbb{N}_{n}\equiv\left\{ 0\right\} \cup\mathbb{N}_{n}^{+}$,
for any positive natural $n$.

\section{\label{sec:Quant_Filtering}System Model \& Problem Formulation}

\subsection{System Model \& Technical Assumptions}

All stochastic processes defined below are defined on a common complete
probability space (the base space), defined by a triplet $\left(\Omega,\mathscr{F},{\cal P}\right)$.
Also, for a set ${\cal A}$, $\mathscr{B}\left({\cal A}\right)$ denotes
the respective Borel $\sigma$-algebra.

Let $X_{t}\in\mathbb{R}^{M\times1}$ be Markov with \textit{known}
\textit{dynamics} (stochastic kernel)\footnote{Hereafter, we employ the usual notation ${\cal K}_{t}\left({\cal A}\left|X_{t-1}\equiv\boldsymbol{x}\right.\right)\equiv{\cal K}_{t}\left(\left.{\cal A}\right|\boldsymbol{x}\right)$,
for ${\cal A}$ Borel.}
\begin{equation}
{\cal K}_{t}:\mathscr{B}\left(\mathbb{R}^{M\times1}\right)\times\mathbb{R}^{M\times1}\mapsto\left[0,1\right],\quad t\in\mathbb{N},
\end{equation}
which, together with an initial probability measure ${\cal P}_{X_{-1}}$
on $\left(\mathbb{R}^{M\times1},\mathscr{B}\left(\mathbb{R}^{M\times1}\right)\right)$,
completely describe its stochastic behavior. Generically, the state
is assumed to be compactly supported in $\mathbb{R}^{M\times1}$,
that is, for all $t\in\left\{ -1\right\} \cup\mathbb{N},X_{t}\in{\cal Z}\subset\mathbb{R}^{M\times1}$,
${\cal P}-a.s.$. We may also alternatively assume the existence of
an explicit state transition model describing the temporal evolution
of the state, as
\begin{equation}
X_{t}\triangleq f_{t}\left(X_{t-1},W_{t}\right)\in{\cal Z},\quad\forall t\in\mathbb{N},\label{eq:State_Equation}
\end{equation}
where, for each $t$, $f_{t}:{\cal Z}\times{\cal W}\overset{a.s.}{\mapsto}{\cal Z}$
constitutes a measurable nonlinear state transition mapping with somewhat
``favorable'' analytical behavior (see below) and $W_{t}\equiv W_{t}\left(\omega\right)\in{\cal W}\subseteq\mathbb{R}^{M_{W}\times1}$,
for $t\in\mathbb{N}$, $\omega\in\Omega$, denotes a white noise process
with state space ${\cal W}$. The recursion defined in \eqref{eq:State_Equation}
is initiated by choosing $X_{-1}\sim{\cal P}_{X_{-1}}$, independently
of $W_{t}$.

The state $X_{t}$ is partially observed through the \textit{conditionally
Gaussian} process
\begin{equation}
\mathbb{R}^{N\times1}\ni\left.{\bf y}_{t}\right|X_{t}\overset{i.i.d.}{\sim}{\cal N}\left(\boldsymbol{\mu}_{t}\left(X_{t}\right),\boldsymbol{\Sigma}_{t}\left(X_{t}\right)+\sigma_{\boldsymbol{\Sigma}}^{2}{\bf I}_{N}\right),\label{eq:Observation_Equation}
\end{equation}
$\sigma_{\boldsymbol{\Sigma}}\ge0$, with conditional means and variances
known apriori, for all $t\in\mathbb{N}$. Additionally, we assume
that $\boldsymbol{\Sigma}_{t}\left(X_{t}\right)\succ{\bf 0}$, with
$\boldsymbol{\Sigma}_{t}:{\cal Z}\mapsto{\cal D}_{\boldsymbol{\Sigma}}$,
for all $t\in\mathbb{N}$, where ${\cal D}_{\boldsymbol{\Sigma}}\subset\mathbb{R}^{N\times N}$
is bounded. The observations \eqref{eq:Observation_Equation} can
also be rewritten in the canonical form ${\bf y}_{t}\equiv\boldsymbol{\mu}_{t}\left(X_{t}\right)+\sqrt{{\bf C}_{t}\left(X_{t}\right)}\boldsymbol{u}_{t}$,
for all $t\in\mathbb{N}$, where $\boldsymbol{u}_{t}\equiv\boldsymbol{u}_{t}\left(\omega\right)$
constitutes a standard Gaussian white noise process and, for all $\boldsymbol{x}\in{\cal Z}$,
${\bf C}_{t}\left(\boldsymbol{x}\right)\triangleq\boldsymbol{\Sigma}_{t}\left(\boldsymbol{x}\right)+\sigma_{\boldsymbol{\Sigma}}^{2}{\bf I}_{N}$.
The process $\boldsymbol{u}_{t}$ is assumed to be mutually independent
of $X_{-1}$, and of the innovations $W_{t}$, in case $X_{t}\equiv f_{t}\left(X_{t-1},W_{t}\right)$.

The class of partially observable systems described above is very
wide, containing all (first order) Hidden Markov Models (HMMs) with
compactly supported state processes and conditionally Gaussian measurements.
Hereafter, without loss of generality and in order to facilitate the
presentation, we will assume stationarity of state transitions, dropping
the subscript ``$t$'' in the respective stochastic kernels and/or
transition mappings. However, we should mention that all subsequent
results hold true also for the nonstationary case, if one assumes
that any condition hereafter imposed on the mechanism generating $X_{t}$
holds for all $t\in\mathbb{N}$, that is, for all different ``modes''
of the state process. As in \cite{KalPetNonlinear_2015}, the following
additional technical assumptions are made. \textbf{\medskip{}
}

\noindent \textbf{Assumption 1:} \label{Elementary_Definition_1-1}\textbf{(Boundedness)}
The quantities $\lambda_{max}\left({\bf C}_{t}\left(\boldsymbol{x}\right)\right)$,
$\left\Vert \boldsymbol{\mu}_{t}\left(\boldsymbol{x}\right)\right\Vert _{2}$
are each uniformly upper bounded both with respect to $t\in\mathbb{N}$
and $\boldsymbol{x}\in{\cal X}$, with finite bounds $\lambda_{sup}$
and $\mu_{sup}$, respectively. For technical reasons, it is also
true that $\lambda_{inf}\triangleq\inf_{t\in\mathbb{N}}\inf_{\boldsymbol{x}\in{\cal X}}\lambda_{min}\left({\bf C}_{t}\left(\boldsymbol{x}\right)\right)>1$.
This can always be satisfied by normalization of the observations.
If $\boldsymbol{x}$ is substituted by the $X_{t}\left(\omega\right)$,
then all the above continue to hold almost everywhere.\textbf{\medskip{}
}

\noindent \textbf{Assumption 2:} \textbf{(Continuity \& Expansiveness)}
All members of the family $\left\{ \boldsymbol{\mu}_{t}:{\cal Z}\mapsto\mathbb{R}^{N\times1}\right\} _{t\in\mbox{\ensuremath{\mathbb{N}}}}$
are uniformly Lipschitz continuous on ${\cal Z}$ with respect to
the $\ell_{1}$-norm. Additionally, all members of the family $\left\{ \boldsymbol{\Sigma}_{t}:{\cal Z}\mapsto{\cal D}_{\boldsymbol{\Sigma}}\right\} _{t\in\mathbb{N}}$
are \textit{elementwise} uniformly Lipschitz continuous on ${\cal Z}$
with respect to the $\ell_{1}$-norm. If ${\cal Z}$ is regarded as
the essential state space of $X_{t}\left(\omega\right)$, then all
the above statements are understood essentially.
\begin{rem}
In certain applications, conditional Gaussianity of the observations
given the state may not be a valid modeling assumption. However, such
a structural assumption not only allows for analytical tractability
when it holds, but also provides important insights related to the
performance of the respective approximate filter, even if the conditional
distribution of the observations is not Gaussian, provided it is ``sufficiently
smooth and unimodal''.\hfill{}\ensuremath{\blacksquare}
\end{rem}

\subsection{Prior Results \& Problem Formulation}

Before proceeding and for later reference, let us define the complete
natural filtrations generated by the processes $X_{t}$ and ${\bf y}_{t}$
as $\left\{ \mathscr{X}_{t}\right\} _{t\in\mathbb{N}\cup\left\{ -1\right\} }$
and $\left\{ \mathscr{Y}_{t}\right\} _{t\in\mathbb{N}}$, respectively.

Adopting the MMSE as an optimality criterion for inferring the hidden
process $X_{t}$ on the basis of the observations, one would ideally
like to discover an efficient way for evaluating the conditional expectation
or \textit{filter} of the state, given the available information encoded
in $\mathscr{Y}_{t}$, sequentially in time. Unfortunately, except
for some very special cases, \cite{Segall_Point1976,Marcus1979,Elliott1994Exact,Elliott1994_HowToCount},
it is well known that the optimal nonlinear filter does not admit
an explicit finite dimensional representation \cite{Segall1976,Elliott1994Hidden}.

As a result, one must resort to properly designed approximations to
the general nonlinear filtering problem, leading to well behaved,
finite dimensional, approximate filtering schemes. Such schemes are
typically derived by approximating the desired quantities of interest
either heuristically (see, e.g. \cite{Kushner1967_Approximations,ItoXiong1997}),
or in some more powerful, rigorous sense, (see, e.g., Markov chain
approximations \cite{Kushner2001_BOOK,Kushner2008,Pages2005optimal},
or particle filtering techniques \cite{PARTICLE2002tutorial,Crisan2002Survey}).
In this paper, we follow the latter direction and propose a novel,
rigorous development of grid based approximate filtering, focusing
on the class of partially observable systems described in Section
\ref{sec:Quant_Filtering}.A. For this, we exploit the general asymptotic
results presented in \cite{KalPetNonlinear_2015}.

Our analysis is based on a well known representation of the optimal
filter, employing the simple concept (at least in discrete time) of
\textit{change of probability measures} (see, e.g., \cite{Elliott1994_HowToCount,Elliott1994Exact,Elliott1994Hidden,Elliott2005JUMP}).
Let $\mathbb{E}_{{\cal P}}\left\{ \left.X_{t}\right|\mathscr{Y}_{t}\right\} $
denote the filter of $X_{t}$ given $\mathscr{Y}_{t}$, under the
base measure ${\cal P}$. Then, there exists another (hypothetical)
probability measure $\widetilde{{\cal P}}$ \cite{Elliott1994Hidden,KalPetNonlinear_2015},
such that \textit{
\begin{equation}
\mathbb{E}_{{\cal P}}\left\{ \left.X_{t}\right|\mathscr{Y}_{t}\right\} \equiv\dfrac{\mathbb{E}_{\widetilde{{\cal P}}}\left\{ \left.X_{t}\Lambda_{t}\right|\mathscr{Y}_{t}\right\} }{\mathbb{E}_{\widetilde{{\cal P}}}\left\{ \left.\Lambda_{t}\right|\mathscr{Y}_{t}\right\} },\label{eq:Changed_Expectation}
\end{equation}
}where $\Lambda_{t}\triangleq\prod_{i\in\mathbb{N}_{t}}\mathsf{L}_{i}\left(X_{i},{\bf y}_{i}\right)$
and $\mathsf{L}_{t}\left(X_{t},{\bf y}_{t}\right)\triangleq\left(\sqrt{2\pi}\right)^{N}$
${\cal N}\left({\bf y}_{t};\boldsymbol{\mu}_{t}\left(X_{t}\right),{\bf C}_{t}\left(X_{t}\right)\right)$,
for all $t\in\mathbb{N}$, with ${\cal N}\left(\boldsymbol{x};\boldsymbol{\mu},{\bf C}\right)$
denoting the multivariate Gaussian density as a function of $\boldsymbol{x}$,
with mean $\boldsymbol{\mu}$ and covariance matrix ${\bf C}$. Here,
we also define $\Lambda_{-1}\equiv1$. The most important part is
that, under $\widetilde{{\cal P}}$, the processes $X_{t}$ (including
the initial value $X_{-1}$) and ${\bf y}_{t}$ are mutually statistically
independent, with $X_{t}$ being the same as under the original measure
and ${\bf y}_{t}$ being a Gaussian vector white noise process with
zero mean and covariance matrix the identity. As one might guess,
the measure $\widetilde{{\cal P}}$ is more convenient to work with.
It is worth mentioning that the \textit{Feynman-Kac formula} \eqref{eq:Changed_Expectation}
is true regardless of the nature of the state $X_{t}$, that is, it
holds even if $X_{t}$ is not Markov. In fact, the machinery of change
of measures can be applied to any nonlinear filtering problem and
is not tied to the particular filtering formulations considered in
this paper \cite{Elliott1994Hidden}.

Let us now replace $X_{t}$ in the RHS of \eqref{eq:Changed_Expectation}
with another process $X_{t}^{L_{S}}$, called the \textit{approximation},
with \textit{resolution }or\textit{ approximation parameter} $L_{S}\in\mathbb{N}$
(conventionally), also independent of the observations under $\widetilde{{\cal P}}$,
for which the evaluation of the resulting ``filter'' might be easier.
Then, we can define the \textit{approximate filter} of the state $X_{t}$
\begin{equation}
{\cal E}^{L_{S}}\left(\left.X_{t}\right|\mathscr{Y}_{t}\right)\triangleq\dfrac{\mathbb{E}_{\widetilde{{\cal P}}}\left\{ \left.X_{t}^{L_{S}}\Lambda_{t}^{L_{S}}\right|\mathscr{Y}_{t}\right\} }{\mathbb{E}_{\widetilde{{\cal P}}}\left\{ \left.\Lambda_{t}^{L_{S}}\right|\mathscr{Y}_{t}\right\} },\quad\forall t\in\mathbb{N}.\label{eq:CoM_Approx}
\end{equation}
It was shown in \cite{KalPetNonlinear_2015} that, under certain conditions,
this approximate filter is asymptotically consistent, as follows.

Hereafter, $\mathds{1}_{{\cal A}}:\mathbb{R}\rightarrow\left\{ 0,1\right\} $
denotes the indicator of ${\cal A}$. Given $x\in\mathbb{R}$ and
for any Borel ${\cal A}$, $\mathds{1}_{{\cal A}}\left(x\right)$
constitutes a Dirac (atomic) probability measure. Equivalently, we
write $\mathds{1}_{{\cal A}}\left(x\right)\equiv\delta_{x}\left({\cal A}\right)$.
Also, convergence in probability is meant to be with respect to the
$\ell_{1}$-norm of the random elements involved. Additionally, below
we refer to the concept to ${\cal C}$-weak convergence, which is
nothing but weak convergence \cite{BillingsleyMeasures} of conditional
probability distributions \cite{Berti2006,Grubel2014}. For a sufficient
definition, the reader is referred to \cite{KalPetNonlinear_2015}.
\begin{thm}
\label{CONVERGENCE_THEOREM}\textbf{\textup{(Convergence to the Optimal
Filter \cite{KalPetNonlinear_2015})}} Pick any natural $T<\infty$
and suppose either of the following:
\begin{itemize}
\item \noindent For all $t\in\mathbb{N}_{T}$, the sequence $\left\{ X_{t}^{L_{S}}\right\} _{L_{S}\in\mathbb{N}}$
is marginally ${\cal C}$-weakly convergent to $X_{t}$, given $X_{t}$,
that is,
\begin{equation}
{\cal P}_{\left.X_{t}^{L_{S}}\right|X_{t}}^{L_{S}}\left(\left.\cdot\right|X_{t}\right)\stackrel[L_{S}\rightarrow\infty]{{\cal W}}{\longrightarrow}\delta_{X_{t}}\left(\cdot\right),\quad\forall t\in\mathbb{N}_{T}.
\end{equation}

\item \noindent For all $t\in\mathbb{N}_{T}$, the sequence $\left\{ X_{t}^{L_{S}}\right\} _{L_{S}\in\mathbb{N}}$
is (marginally) convergent to $X_{t}$ in probability, that is,
\begin{equation}
X_{t}^{L_{S}}\stackrel[L_{S}\rightarrow\infty]{{\cal P}}{\longrightarrow}X_{t},\quad\forall t\in\mathbb{N}_{T}.
\end{equation}

\end{itemize}
Then, there exists a measurable subset $\widehat{\Omega}_{T}\subseteq\Omega$
with ${\cal P}$-measure at least $1-\left(T+1\right)^{1-CN}\exp\left(-CN\right)$,
such that
\begin{equation}
\sup_{t\in\mathbb{N}_{T}}\sup_{\omega\in\widehat{\Omega}_{T}}\left\Vert {\cal E}^{L_{S}}\left(\left.X_{t}\right|\mathscr{Y}_{t}\right)\hspace{-2pt}-\hspace{-2pt}\mathbb{E}_{{\cal P}}\left\{ \left.X_{t}\right|\mathscr{Y}_{t}\right\} \right\Vert _{1}\hspace{-2pt}\left(\omega\right)\hspace{-2pt}\underset{L_{S}\rightarrow\infty}{\longrightarrow}\hspace{-2pt}0,\label{eq:UAE}
\end{equation}
for any free, finite constant $C\ge1$. In other words, the convergence
of the respective approximate filtering operators is compact in $t\in\mathbb{N}$
and, with probability at least $1-\left(T+1\right)^{1-CN}\exp\left(-CN\right)$,
uniform in $\omega$.\end{thm}
\begin{rem}
It should be mentioned here that Theorem \ref{CONVERGENCE_THEOREM}
holds for \textit{any process} $X_{t}$, \textit{Markov or not}, as
long as $X_{t}$ is almost surely compactly supported.\hfill{}\ensuremath{\blacksquare}
\end{rem}

\begin{rem}
The mode of filter convergence reported in Theorem \ref{CONVERGENCE_THEOREM}
is particularly strong. It implies that \textit{inside any fixed finite
time interval and among almost all possible paths} of the observations
process, the approximation error between the true and approximate
filters is finitely bounded and converges to zero, as the grid resolution
increases, resulting in a practically appealing asymptotic property.
This mode of convergence constitutes, in a sense, a practically useful,
quantitative justification of Egorov's Theorem \cite{Richardson2009measure},
which abstractly relates almost uniform convergence with almost sure
convergence of measurable functions. Further, it is important to mention
that, for fixed $T$, convergence to the optimal filter tends to be
in the uniformly almost everywhere sense, \textit{at an exponential
rate} with respect to the dimensionality of the observations, $N$.
This shows that, in a sense, the dimensionality of the observations
\textit{stochastically stabilizes} the approximate filtering process.\hfill{}\ensuremath{\blacksquare}
\end{rem}

\begin{rem}
Observe that the adopted approach concerning construction of the approximate
filter of $X_{t}$, the approximation $X_{t}^{L_{S}}$ is naturally
constructed under the base measure $\widetilde{{\cal P}}$, satisfying
the constraint of being independent of the observations, ${\bf y}_{t}$.
However, it is easy to see that if, for each $t$ in the horizon of
interest, $X_{t}^{L_{S}}$\textit{ }is $\left\{ \mathscr{X}_{t}\right\} $-adapted,
then it may be defined under the original base measure ${\cal P}$
without any complication; under $\widetilde{{\cal P}}$, $X_{t}$
(and, thus, $X_{t}^{L_{S}}$) is independent of ${\bf y}_{t}$ by
construction. In greater generality, $X_{t}^{L_{S}}$ may be constructed
under ${\cal P}$, as long as it can be somehow guaranteed to follow
the same distribution \textit{and} be independent of ${\bf y}_{t}$
under $\widetilde{{\cal P}}$. As we shall see below, this is not
always obvious or true; if fact, it is strongly dependent on the information
(encoded in the appropriate $\sigma$-algebra) exploited in order
to define the process $X_{t}^{L_{S}}$, as well as the particular
choice of the alternative measure $\widetilde{{\cal P}}$.\hfill{}\ensuremath{\blacksquare}
\end{rem}

\section{\label{sec:Uniform-State-Quantizations}Uniform State Quantizations}

Although Theorem \ref{CONVERGENCE_THEOREM} presented above provides
the required conditions for convergence of the respective approximate
filter, it does not specify any specific class of processes to be
used as the required approximations. In order to satisfy either of
the conditions of Theorem \ref{CONVERGENCE_THEOREM}, $X_{t}^{L_{S}}$
must be strongly dependent on $X_{t}$. For example, if the approximation
is merely weakly convergent to the original state process (as, for
instance, in particle filtering techniques), the conditions of Theorem
\ref{CONVERGENCE_THEOREM} will not be fulfilled. In this paper, the
state $X_{t}$ is approximated by another closely related process
with discrete state space, constituting a uniformly quantized approximation
of the original one. 

Similarly to \cite{Pages2005optimal}, we will consider two types
of state approximations: \textit{Marginal Quantizations }and \textit{Markovian
Quantizations}. Specifically, in the following, we study pathwise
properties of the aforementioned state approximations. Nevertheless,
and as in every meaningful filtering formulation, neither the state
nor its approximations need to be known or constructed by the user.
Only the (conditional) laws of the approximations need to be known.
To this end, let us state a general definition of a quantizer.
\begin{defn}
\textbf{(Quantizers)} Consider a compact subset ${\cal A}\subset\mathbb{R}^{N}$,
a partition $\Pi\triangleq\left\{ {\cal A}_{i}\right\} _{i\in\mathbb{N}_{L}^{+}}$
of ${\cal A}$ and let ${\cal B}\triangleq\left\{ \left\{ b_{i}\right\} _{i\in\mathbb{N}_{L}^{+}}\right\} $
be a discrete set consisting of \textit{distinct reconstruction points},
with $b_{i}\in\mathbb{R}^{M},\forall i\in\mathbb{N}_{L}^{+}$. Then,
an \textit{$L$-level Euclidean Quantizer} is any bounded and measurable
function ${\cal Q}_{L}:\left({\cal A},\mathscr{B}\left({\cal A}\right)\right)\mapsto\left({\cal B},2^{{\cal B}}\right)$,
defined by assigning all $x\in{\cal A}_{i}\in\Pi,i\in\mathbb{N}_{L}^{+}$
to a unique $b_{j}\in{\cal B},j\in\mathbb{N}_{L}^{+}$, such that
the mapping between the elements of $\Pi$ and ${\cal B}$ is one
to one and onto (a bijection).
\end{defn}

\subsection{Uniformly Quantizing ${\cal Z}$}

For simplicity and without any loss of generality, suppose that ${\cal Z}\equiv\left[a,b\right]^{M}$
(for $a\in\mathbb{R}$ and $b\in\mathbb{R}$ with obviously $a<b$),
representing the compact set of support of the state $X_{t}$. Also,
consider a uniform $L$-set partition of the interval $\left[a,b\right]$,
$\Pi_{L}\triangleq\left\{ {\cal Z}_{l}\right\} _{l\in\mathbb{N}_{L-1}}$
and, additionally, let $\Pi_{L_{S}}\triangleq\underset{M\text{ times}}{\times}\Pi_{L}$
be the overloaded Cartesian product of $M$ copies of the partitions
defined above, with cardinality $L_{S}\triangleq L^{M}$. As usual,
our reconstruction points will be chosen as the center of masses of
the hyperrectangles comprising the hyperpartition $\Pi_{L_{S}}$,
denoted as $\boldsymbol{x}_{L_{S}}^{\left\{ l_{m}\right\} _{m\in\mathbb{N}_{M}^{+}}}\equiv\boldsymbol{x}_{L_{S}}^{\left\{ l_{m}\right\} }$,
where $l_{m}\in\mathbb{N}_{L-1}$. According to some predefined ordering,
we make the identification $\boldsymbol{x}_{L_{S}}^{\left\{ l_{m}\right\} }\equiv\boldsymbol{x}_{L_{S}}^{l}$,
$l\in\mathbb{N}_{L_{S}}^{+}$. Further, let ${\cal X}_{L_{S}}\triangleq\left\{ \boldsymbol{x}_{L_{S}}^{1},\boldsymbol{x}_{L_{S}}^{2},\ldots,\boldsymbol{x}_{L_{S}}^{L_{S}}\right\} $
and define the quantizer ${\cal Q}_{L_{S}}:\left({\cal Z},\mathscr{B}\left({\cal Z}\right)\right)\mapsto\left({\cal X}_{L_{S}},2^{{\cal X}_{L_{S}}}\right)$,
where\renewcommand{\arraystretch}{1.5}
\begin{equation}
\begin{array}{c}
{\cal Q}_{L_{S}}\left(\boldsymbol{x}\right)\triangleq\boldsymbol{x}_{L_{S}}^{\left\{ l_{m}\right\} }\equiv\boldsymbol{x}_{L_{S}}^{l}\in{\cal X}_{L_{S}}\\
\text{iff}\quad\boldsymbol{x}\in\underset{m\in\mathbb{N}_{M}^{+}}{\times}{\cal Z}_{l_{m}}\triangleq{\cal Z}_{L_{S}}^{l}\in\Pi_{L_{S}}
\end{array}.
\end{equation}
\renewcommand{\arraystretch}{1}Given the definitions stated above,
the following simple and basic result is true. The proof, being elementary,
is omitted.
\begin{lem}
\label{EQS_CONVERGENCE_1}\textbf{\textup{(Uniform Convergence of
Quantized Values)}} It is true that
\begin{equation}
\lim_{L_{S}\rightarrow\infty}\sup_{\boldsymbol{x}\in{\cal Z}}\left\Vert {\cal Q}_{L_{S}}\left(\boldsymbol{x}\right)-\boldsymbol{x}\right\Vert _{1}\equiv0,
\end{equation}
that is, ${\cal Q}_{L_{S}}\left(\boldsymbol{x}\right)$ converges
as $L_{S}\rightarrow\infty$, uniformly in $\boldsymbol{x}$.\end{lem}
\begin{rem}
We should mention here that Lemma \ref{EQS_CONVERGENCE_1}, as well
as all the results to be presented below hold equally well when the
support of $X_{t}$ is different in each dimension, or when different
quantization resolutions are chosen in each dimension, just by adding
additional complexity to the respective arguments.\hfill{}\ensuremath{\blacksquare}
\end{rem}

\subsection{Marginal Quantization}

The first class of state process approximations of interest is that
of marginal state quantizations, according to which $X_{t}$ is approximated
by its nearest neighbor
\begin{equation}
X_{t}^{L_{S}}\left(\omega\right)\triangleq{\cal Q}_{L_{S}}\left(X_{t}\left(\omega\right)\right)\in{\cal X}_{L_{S}},\quad\forall t\in\left\{ -1\right\} \cup\mathbb{N},
\end{equation}
${\cal P}-a.s.$, where $L_{S}\in\mathbb{N}$ is identified as the
approximation parameter. Next, we present another simple but important
lemma, concerning the behavior of the quantized stochastic process
$X_{t}^{L_{S}}\left(\omega\right)$, as $L_{S}$ gets large. Again,
the proof is relatively simple, and it is omitted.
\begin{lem}
\label{Marginal_Convergence}\textbf{\textup{(Uniform Convergence
of Marginal State Quantizations)}} For $X_{t}\left(\omega\right)\in{\cal Z}$,
for all $t\in\mathbb{N}$, almost surely, it is true that
\begin{equation}
\lim_{L_{S}\rightarrow\infty}\sup_{t\in\mathbb{N}}\hspace{0.2em}\underset{\omega\in\Omega}{\mathrm{ess}\hspace{0.2em}\mathrm{sup}}\left\Vert X_{t}^{L_{S}}\left(\omega\right)-X_{t}\left(\omega\right)\right\Vert _{1}\equiv0,
\end{equation}
that is, $X_{t}^{L_{S}}\left(\omega\right)$ converges as $L_{S}\rightarrow\infty$,
uniformly in $t$ and uniformly ${\cal P}$-almost everywhere in $\omega$.\end{lem}
\begin{rem}
One drawback of marginal approximations is that they \textit{do not}
possess the Markov property any more. This fact introduces considerable
complications in the development of recursive estimators, as shown
later in Section \ref{sec:CORE_SECTION}. However, marginal approximations
are practically appealing, because they do not require explicit knowledge
of the stochastic kernel describing the transitions of $X_{t}$ \cite{KalPetChannelMarkov2014,KalPet_SPAWC_2015}.\hfill{}\ensuremath{\blacksquare}
\end{rem}

\begin{rem}
\label{rem:Property_EQUIV}Note that the implications of Lemma \ref{Marginal_Convergence}
continue to be true under the base measure $\widetilde{{\cal P}}$.
This is true because $X_{t}^{L_{S}}$ is $\left\{ \mathscr{X}_{t}\right\} $-adapted,
and also due to the fact that the ``local'' probability spaces $\left(\Omega,\mathscr{X}_{\infty},{\cal P}\right)$
and $\left(\Omega,\mathscr{X}_{\infty},\widetilde{{\cal P}}\right)$
are completely identical. Here, $\mathscr{X}_{\infty}\triangleq\sigma\left\{ \bigcup_{t\in\mathbb{N}\cup\left\{ -1\right\} }\mathscr{X}_{t}\right\} $
constitutes the join of the filtration $\left\{ \mathscr{X}_{t}\right\} _{t\in\mathbb{N}\cup\left\{ -1\right\} }$.
In other words, the restrictions of ${\cal P}$ and $\widetilde{{\cal P}}$
on $\mathscr{X}_{\infty}$ -the collection of events ever to be generated
by $X_{t}$- coincide; that is, $\left.{\cal P}\right|_{\mathscr{X}_{\infty}}\equiv\left.\widetilde{{\cal P}}\right|_{\mathscr{X}_{\infty}}$.\hfill{}\ensuremath{\blacksquare}
\end{rem}
\textbf{\vspace{-21pt}
}

\subsection{Markovian Quantization}

The second class of approximations considered is that of Markovian
quantizations of the state. In this case, we assume explicit knowledge
of a transition mapping, modeling the temporal evolution of $X_{t}$.
In particular, we assume a recursion as in \eqref{eq:State_Equation},
where the process $W_{t}$ acts as the driving noise of the state
$X_{t}$ and constitutes an intrinsic characteristic of it. Then,
the Markovian quantization of $X_{t}$ is defined as
\begin{equation}
X_{t}^{L_{S}}\triangleq{\cal Q}_{L_{S}}\left(f\left(X_{t-1}^{L_{S}},W_{t}\right)\right)\in{\cal X}_{L_{S}},\:\forall t\in\mathbb{N},\label{eq:Approx_Process}
\end{equation}
with $X_{-1}^{L_{S}}\hspace{-2pt}\equiv\hspace{-2pt}{\cal Q}_{L_{S}}\left(X_{-1}\right)\hspace{-2pt}\in\hspace{-2pt}{\cal X}_{L_{S}}$,
${\cal P}-a.s.$, and which satisfies the Markov property trivially;
since ${\cal X}_{L_{S}}$ is finite, it constitutes a (time-homogeneous)
\textit{finite state space Markov Chain}. A scheme for generating
$X_{t}^{L_{S}}$ is shown in Fig. \ref{fig:Markovian}.

At this point, it is very important to observe that, whereas $X_{t}$
is guaranteed to be Markov with the same dynamics and independent
of ${\bf y}_{t}$ under $\widetilde{{\cal P}}$, we cannot immediately
say the same for the Markovian approximation $X_{t}^{L_{S}}$. The
reason is that $X_{t}^{L_{S}}$ is measurable with respect to the
filtration generated by the initial condition $X_{-1}$ and the innovations
process $W_{t}$ and not with respect to $\left\{ \mathscr{X}_{t}\right\} _{t\in\mathbb{N}\cup\left\{ -1\right\} }$.
Without any additional considerations, $W_{t}$ may very well be partially
correlated relative to ${\bf y}_{t}$ and/or $X_{-1}$, and/or even
non white itself! Nevertheless, $\widetilde{{\cal P}}$ may be chosen
such that $W_{t}$ indeed satisfies the aforementioned properties
under question, as the following result suggests.
\begin{lem}
\label{P_Equivalence}\textbf{\textup{(Choice of }}$\widetilde{{\cal P}}$\textbf{\textup{)}}
Without any other modification, the base measure $\widetilde{{\cal P}}$
may be chosen such that the initial condition $X_{-1}$ and the innovations
process $W_{t}$ follow the same distributions as under ${\cal P}$
\textbf{and} are all mutually independent relative to the observations,
${\bf y}_{t}$.\end{lem}
\begin{proof}[Proof of Lemma \ref{P_Equivalence}]
See Appendix F.
\end{proof}
Lemma \ref{P_Equivalence} essentially implies that Markovian quantizations
may be constructed and analyzed either under ${\cal P}$ or $\widetilde{{\cal P}}$,
interchangeably. Also adapt Remark \ref{rem:Property_EQUIV} to this
case.

Under the assumption of a transition mapping, every possible path
of $X_{t}\left(\omega\right)$ is completely determined by fixing
$X_{-1}\left(\omega\right)$ and $W_{t}\left(\omega\right)$ at any
particular realization, for each $\omega\in\Omega$. As in the case
of marginal quantizations, the goal of the Markovian quantization
is the \textit{pathwise} approximation of $X_{t}$ by $X_{t}^{L_{S}}$,
for almost all realizations of the white noise process $W_{t}$ and
initial value $X_{-1}$. In practice, however, as noted in the beginning
of this section, knowledge of $W_{t}$ is of course not required by
the user. What is required by the user is the transition matrix of
the Markov chain $X_{t}^{L_{S}}$, which could be obtained via, for
instance, simulation (also see Section IV).

For analytical tractability, we will impose the following reasonable
regularity assumption on the expansiveness of the transition mapping
$f$$\left(\cdot,\cdot\right)$:
\begin{figure}
\centering\includegraphics[scale=1.2]{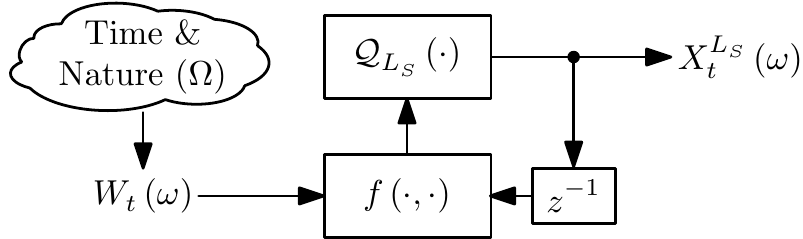} 

\caption{\label{fig:Markovian}Block representation of Markovian quantization.
As noted in the cloud, ``Nature'' here refers to the sample space
$\Omega$ of the base triplet $\left(\Omega,\mathscr{F},{\cal P}\right)$.}
\vspace{-11pt}
\end{figure}
\textbf{\medskip{}
}\\
\textbf{Assumption 3 (Expansiveness of Transition Mappings):} For
all $\boldsymbol{y}\in{\cal W}$, $f:{\cal Z}\times{\cal W}\mapsto{\cal Z}$
is \textit{Lipschitz continuous} in $\boldsymbol{x}\in{\cal Z}$,
that is, possibly dependent on each $\boldsymbol{y}$, there exists
a non-negative, bounded constant $K\left(\boldsymbol{y}\right)$,
where $\sup_{\boldsymbol{y}\in{\cal W}}K\left(\boldsymbol{y}\right)$
exists and is finite, such that
\begin{equation}
\left\Vert f\left(\boldsymbol{x}_{1},\boldsymbol{y}\right)-f\left(\boldsymbol{x}_{2},\boldsymbol{y}\right)\right\Vert _{1}\le K\left(\boldsymbol{y}\right)\left\Vert \boldsymbol{x}_{1}-\boldsymbol{x}_{2}\right\Vert _{1},
\end{equation}
\textbf{$\forall\left(\boldsymbol{x}_{1},\boldsymbol{x}_{2}\right)\in{\cal Z}\times{\cal Z}$}.
If, additionally, $\sup_{\boldsymbol{y}\in{\cal W}}K\left(\boldsymbol{y}\right)<1$,
then $f$$\left(\cdot,\cdot\right)$ will be referred to as \textit{uniformly
contractive}. \textbf{\medskip{}
}

Employing Assumption 3, the next result presented below characterizes
the convergence of the Markovian state approximation $X_{t}^{L_{S}}$
to the true process $X_{t}$, as the quantization of the state space
${\cal Z}$ gets finer and under appropriate conditions.
\begin{lem}
\label{Markovian_Convergence}\textbf{\textup{(Uniform Convergence
of Markovian State Quantizations)}} Suppose that the transition mapping
$f:{\cal Z}\times{\cal W}\mapsto{\cal Z}$ of the Markov process $X_{t}\left(\omega\right)$
is Lipschitz, almost surely and for all $t\in\mathbb{N}$. Also, consider
the approximating Markov process $X_{t}^{L_{S}}\left(\omega\right)$,
as defined in \eqref{eq:Approx_Process}. Then,
\begin{equation}
\lim_{L_{S}\rightarrow\infty}\underset{\omega\in\Omega}{\mathrm{ess}\hspace{0.2em}\mathrm{sup}}\left\Vert X_{t}^{L_{S}}\left(\omega\right)-X_{t}\left(\omega\right)\right\Vert _{1}\equiv0,\quad\forall t\in\mathbb{N},
\end{equation}
that is, $X_{t}^{L_{S}}\left(\omega\right)$ converges as $L_{S}\rightarrow\infty$,
in the pointwise sense in $t$ and uniformly almost everywhere in
$\omega$. If, additionally, $f\left(\cdot,\cdot\right)$ is uniformly
contractive, almost surely and for all $t\in\mathbb{N}$, then it
is true that
\begin{equation}
\lim_{L_{S}\rightarrow\infty}\sup_{t\in\mathbb{N}}\hspace{0.2em}\underset{\omega\in\Omega}{\mathrm{ess}\hspace{0.2em}\mathrm{sup}}\left\Vert X_{t}^{L_{S}}\left(\omega\right)-X_{t}\left(\omega\right)\right\Vert _{1}\equiv0,
\end{equation}
that is, the convergence is additionally uniform in $t$.\end{lem}
\begin{proof}[Proof of Lemma \ref{Markovian_Convergence}]
See Appendix A.
\end{proof}
Especially concerning temporally uniform convergence of the quantization
schemes under consideration, and to highlight its great practical
importance, it would be useful to illustrate the implications of Lemmata
\ref{Marginal_Convergence} and \ref{Markovian_Convergence} by means
of the following simple numerical example.
\begin{figure*}
\centering\subfloat[]{\includegraphics[scale=0.57]{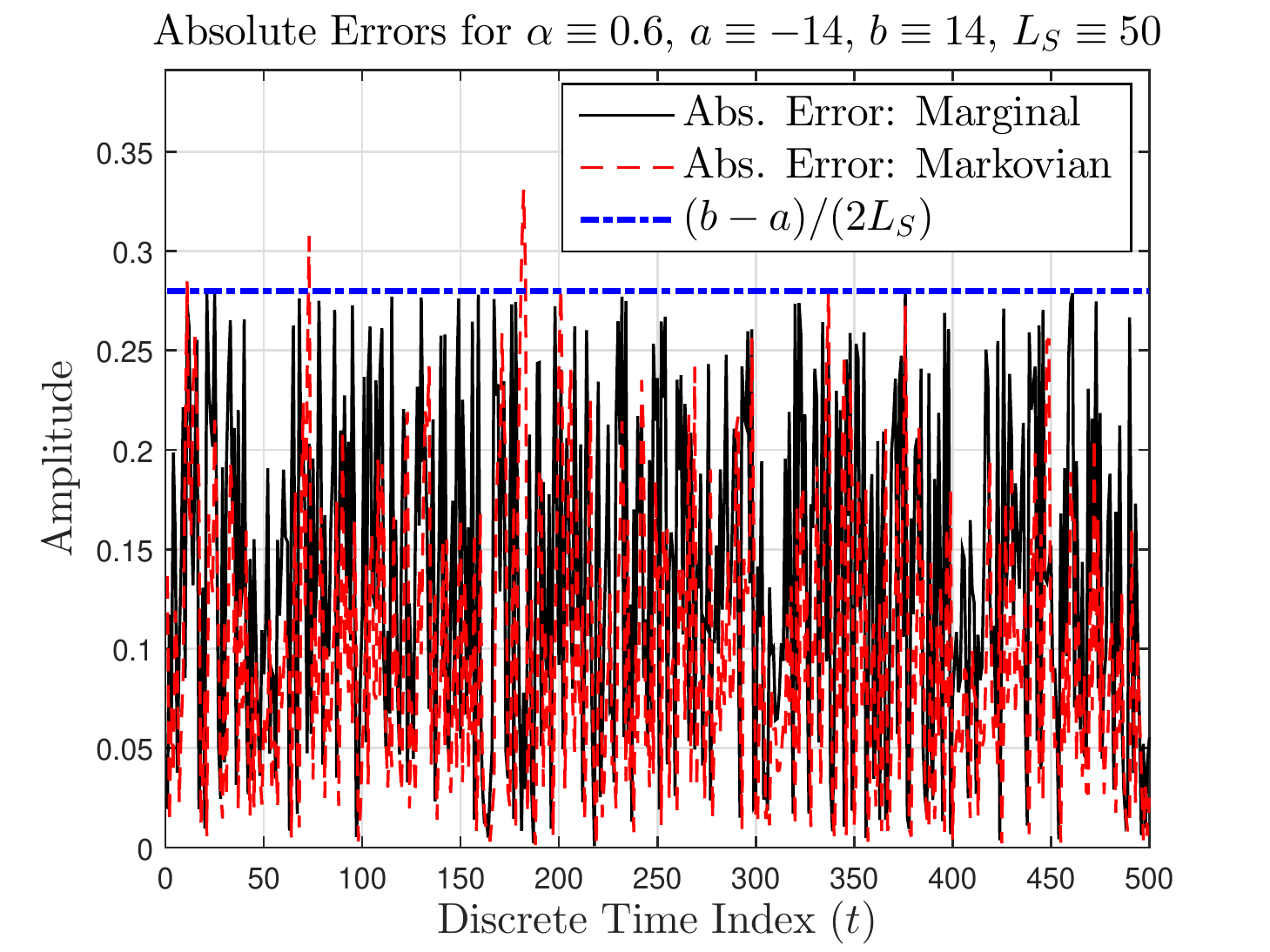}

}\negthickspace{}\subfloat[]{\includegraphics[scale=0.57]{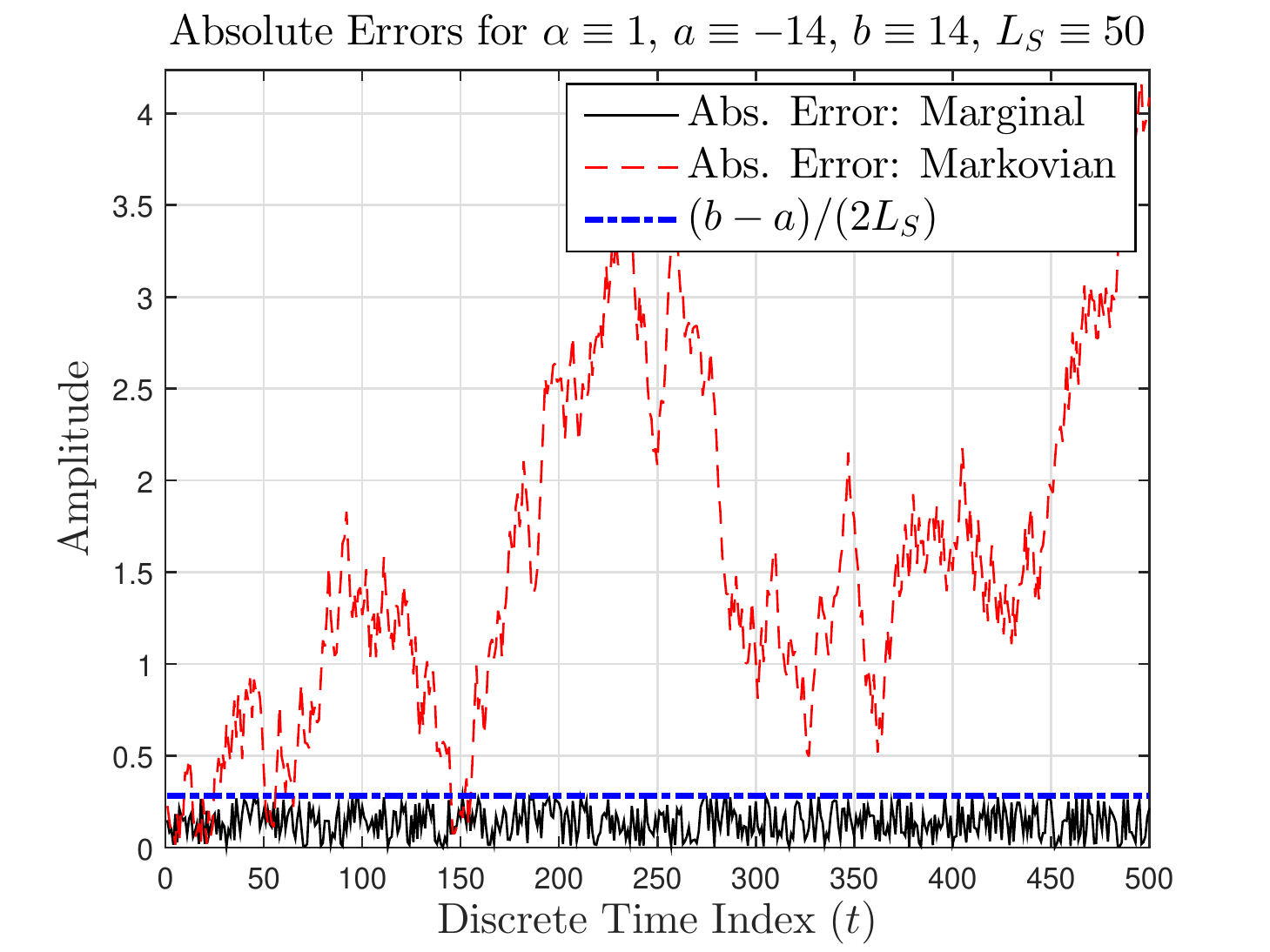}

}

\caption{\label{fig:Unstable_AR}Absolute errors between each of the quantized
versions of the $AR\left(1\right)$ process of our example, and the
true process itself, respectively, for (a) $\alpha\equiv0.6$ (stable
process) and (b) $\alpha\equiv1$ (a random walk).}
\vspace{-11pt}
\end{figure*}

\begin{example}
Let $X_{t}$ be a scalar, first order autoregressive process ($AR\left(1\right)$),
defined via the linear stochastic difference equation
\begin{equation}
X_{t}\triangleq\alpha X_{t-1}+W_{t},\quad\forall t\in\mathbb{N},\label{eq:AR_1}
\end{equation}
where $W_{t}\overset{i.i.d}{\sim}{\cal N}\left(0,1\right),\forall t\in\mathbb{N}$.
In our example, the parameter $\alpha\in\left[-1,1\right]$ is known
apriori and controls the stability of the process, with the case where
$\alpha\equiv1$ corresponding to a Gaussian random walk. Of course,
it is true that the state space of the process defined by \eqref{eq:AR_1}
is the whole $\mathbb{R}$, which means that, strictly speaking, there
are no finite $a$ and $b$ such that $X_{t}\in\left[a,b\right]\equiv{\cal Z},\forall t\in\mathbb{N}$,
with probability $1$. However, it is true that for sufficiently large
but finite $a$ and $b$, there exists a ``large'' measurable set
of possible outcomes for which $X_{t}$, being a Gaussian process,
indeed belongs to ${\cal Z}$ with very high probability. Whenever
this happens, we should be able to verify Lemmata \ref{Marginal_Convergence}
and \ref{Markovian_Convergence} directly.

Additionally, it is trivial to verify that the linear transition function
in \eqref{eq:AR_1} is always a contraction, with Lipschitz constant
$K\equiv\left|\alpha\right|$, whenever the $AR\left(1\right)$ process
of interest is stable, that is, whenever $\left|\alpha\right|<1$.

Fig. \ref{fig:Unstable_AR}(a) and \ref{fig:Unstable_AR}(b) show
the absolute errors between two $AR\left(1\right)$ processes and
their quantized versions according to Lemmata \ref{Marginal_Convergence}
and \ref{Markovian_Convergence}, for $\alpha\equiv0.6$ and $\alpha\equiv1$,
respectively. From the figure, one can readily observe that the marginal
quantization of $X_{t}$ always converges to $X_{t}$ uniformly in
time, regardless of the particular value of $\alpha$, experimentally
validating Lemma \ref{Marginal_Convergence}. On the other hand, it
is obvious that when the transition function of our system is not
a contraction (Lemma \ref{Markovian_Convergence}), uniform convergence
of the respective Markovian quantization to the true state $X_{t}$
cannot be guaranteed. Of course, we have not proved any additional
necessity regarding our sufficiency assumption related to the contractiveness
of the transition mapping of the process of interest, meaning that
there might exist processes which do not fulfill this requirement
and still converge uniformly. However, for uniform contractions, the
convergence will always be uniform \textit{whenever the process $X_{t}$
is bounded in ${\cal Z}$}.\hfill{}\ensuremath{\blacksquare}
\end{example}

\section{\label{sec:CORE_SECTION}Grid Based Approximate Filtering:\protect \\
Recursive Estimation \& Asymptotic Optimality}

It is indeed easy to show that when used as candidate state approximations
for defining approximate filtering operators in the fashion of Section
\ref{sec:Quant_Filtering}.B, both the marginal and Markovian quantization
schemes presented in Sections \ref{sec:Uniform-State-Quantizations}.B
and \ref{sec:Uniform-State-Quantizations}.C, respectively, converge
to the optimal nonlinear filter of the state $X_{t}$. Convergence
is in the sense of Theorem \ref{CONVERGENCE_THEOREM} presented in
Section \ref{sec:Quant_Filtering}.B, \textit{corroborating asymptotic
optimality under a unified convergence criterion}.

Specifically, under the respective (and usual) assumptions, Lemmata
\ref{Marginal_Convergence} and \ref{Markovian_Convergence} presented
above imply that both the marginal and Markovian approximations converge
to the true state $X_{t}$ at least in the almost sure sense, for
all $t\in\mathbb{N}$. Therefore, both will also converge to the true
state in probability, satisfying the second sufficient condition of
Theorem \ref{CONVERGENCE_THEOREM}. The following result is true.
Its proof, being apparent, is omitted.
\begin{thm}
\label{CONVERGENCE_THEOREM_2}\textbf{\textup{(Convergence of Approximate
Filters)}} Pick any natural $T<\infty$ and let the process $X_{t}^{L_{S}}$
represent either the marginal or the Markovian approximation of the
state $X_{t}$. Then, under the respective assumptions implied by
Lemmata \ref{Marginal_Convergence} and \ref{Markovian_Convergence},
the approximate filter ${\cal E}^{L_{S}}\left(\left.X_{t}\right|\mathscr{Y}_{t}\right)$
converges to the true nonlinear filter $\mathbb{E}_{{\cal P}}\left\{ \left.X_{t}\right|\mathscr{Y}_{t}\right\} $,
in the sense of Theorem \ref{CONVERGENCE_THEOREM}.
\end{thm}
Although Theorem \ref{CONVERGENCE_THEOREM_2} shows asymptotic consistency
of the marginal and Markovian approximate filters in a strong sense,
it does not imply the existence of any finite dimensional scheme for
actually realizing these estimators. This is the purpose of the next
subsections. In particular, we develop recursive representations for
the asymptotically optimal (as $L_{S}\rightarrow\infty$) filter ${\cal E}^{L_{S}}\left(\left.X_{t}\right|\mathscr{Y}_{t}\right)$,
as defined previously in \eqref{eq:CoM_Approx}.

For later reference, let us define the \textit{bijective} mapping
(a trivial quantizer) ${\cal Q}_{L_{S}}^{e}:\left({\cal X}_{L_{S}},2^{{\cal X}_{L_{S}}}\right)\mapsto\left({\cal V}_{L_{S}},2^{{\cal V}_{L_{S}}}\right)$,
where the set ${\cal V}_{L_{S}}\triangleq\left\{ {\bf e}_{1}^{L_{S}},\ldots,{\bf e}_{L_{S}}^{L_{S}}\right\} $
contains the complete standard basis in $\mathbb{R}^{L_{S}\times1}$.
Since $\boldsymbol{x}_{L_{S}}^{l}$ is bijectively mapped to ${\bf e}_{l}^{L_{S}}$
for all $l\in\mathbb{N}_{L_{S}}^{+}$, we can write $\boldsymbol{x}_{L_{S}}^{l}\equiv{\bf X}{\bf e}_{l}^{L_{S}}$,
where ${\bf X}\triangleq\left[\boldsymbol{x}_{L_{S}}^{1}\,\boldsymbol{x}_{L_{S}}^{2}\,\ldots\,\boldsymbol{x}_{L_{S}}^{L_{S}}\right]\in\mathbb{R}^{M\times L_{S}}$
constitutes the respective \textit{reconstruction} matrix. From this
discussion, it is obvious that
\begin{equation}
\mathbb{E}_{\widetilde{{\cal P}}}\left\{ \left.X_{t}^{L_{S}}\Lambda_{t}^{L_{S}}\right|\mathscr{Y}_{t}\right\} \hspace{-2pt}\equiv\hspace{-2pt}{\bf X}\mathbb{E}_{\widetilde{{\cal P}}}\left\{ \left.{\cal Q}_{L_{S}}^{e}\left(X_{t}^{L_{S}}\right)\Lambda_{t}^{L_{S}}\right|\mathscr{Y}_{t}\right\} ,
\end{equation}
leading to the expression
\begin{equation}
{\cal E}^{L_{S}}\left(\left.X_{t}\right|\mathscr{Y}_{t}\right)\equiv\dfrac{{\bf X}\mathbb{E}_{\widetilde{{\cal P}}}\left\{ \left.{\cal Q}_{L_{S}}^{e}\left(X_{t}^{L_{S}}\right)\Lambda_{t}^{L_{S}}\right|\mathscr{Y}_{t}\right\} }{\mathbb{E}_{\widetilde{{\cal P}}}\left\{ \left.\Lambda_{t}^{L_{S}}\right|\mathscr{Y}_{t}\right\} },\label{eq:main_approx}
\end{equation}
for all $t\in\mathbb{N}$, regardless of the type of state quantization
employed. We additionally define the \textit{likelihood} matrix
\begin{equation}
\boldsymbol{\Lambda}_{t}\hspace{-2pt}\triangleq\hspace{-2pt}\mathrm{diag}\left(\mathsf{L}_{t}\left(\boldsymbol{x}_{L_{S}}^{1},{\bf y}_{t}\right)\,\ldots\,\mathsf{L}_{t}\left(\boldsymbol{x}_{L_{S}}^{L_{S}},{\bf y}_{t}\right)\hspace{-2pt}\right)\hspace{-2pt}\in\hspace{-2pt}\mathbb{R}^{L_{S}\times L_{S}}.
\end{equation}
Also to be subsequently used, given the quantization type, define
the column stochastic matrix $\boldsymbol{P}\in\left[0,1\right]^{L_{S}\times L_{S}}$
as 
\begin{equation}
\boldsymbol{P}\left(i,j\right)\triangleq{\cal P}\left(\left.X_{t}^{L_{S}}\equiv\boldsymbol{x}_{L_{S}}^{i}\right|X_{t-1}^{L_{S}}\equiv\boldsymbol{x}_{L_{S}}^{j}\right),\label{eq:Trans}
\end{equation}
for all $\left(i,j\right)\in\mathbb{N}_{L_{S}}^{+}\times\mathbb{N}_{L_{S}}^{+}$.

At this point, it will be important to note that the transition matrix
$\boldsymbol{P}$ defined in \eqref{eq:Trans} is implicitly assumed
to be time invariant, regardless of the state approximation employed.
Under the system model established in Section \ref{sec:Quant_Filtering}.A
(assuming temporal homogeneity for the original Markov process $X_{t}$),
this is unconditionally true when one considers Markovian state quantizations,
simply because the resulting approximating process $X_{t}^{L_{S}}$
constitutes a Markov chain with finite state space, as stated earlier
in Section \ref{sec:Uniform-State-Quantizations}.C. On the other
hand, the situation is quite different when one considers marginal
quantizations of the state. In that case, the conditional probabilities
\begin{align}
\hspace{-2pt}\hspace{-2pt}{\cal P}\hspace{-2pt}\left(\hspace{-2pt}\hspace{-1pt}\left.X_{t}^{L_{S}}\hspace{-2pt}\equiv\hspace{-2pt}\boldsymbol{x}_{L_{S}}^{i}\right|\hspace{-2pt}X_{t-1}^{L_{S}}\hspace{-2pt}\equiv\hspace{-2pt}\boldsymbol{x}_{L_{S}}^{j}\hspace{-1pt}\right)\hspace{-2pt}\hspace{-2pt} & \equiv\hspace{-2pt}{\cal P}\hspace{-2pt}\left(\hspace{-2pt}\hspace{-1pt}\left.X_{t}\hspace{-2pt}\in\hspace{-2pt}{\cal Z}_{L_{S}}^{i}\right|\hspace{-2pt}X_{t-1}\hspace{-2pt}\in\hspace{-2pt}{\cal Z}_{L_{S}}^{j}\hspace{-1pt}\right)\hspace{-2pt},\hspace{-2pt}\hspace{-2pt}\label{eq:marginal_MATRIX}
\end{align}
which would correspond to the $\left(i,j\right)$-th element of the
resulting transition matrix, are, in general, \textit{not time invariant
any more, even if the original Markov process is time homogeneous}.
Nevertheless, assuming the existence of at least one \textit{invariant
measure} (a stationary distribution) for the Markov process $X_{t}$,
\textit{also chosen as its initial distribution}, the aforementioned
probabilities are indeed time invariant. This is a very common and
reasonable assumption employed in practice, especially when tracking
stationary signals. For notational and intuitional simplicity, and
in order to present a unified treatment of all the approximate filters
considered in this paper, the aforementioned assumption will also
be adopted in the analysis that follows.

\subsection{Markovian Quantization}

We start with the case of Markovian quantizations, since it is easier
and more straightforward. Here, the development of the respective
approximate filter is based on the fact that $X_{t}^{L_{S}}$ constitutes
a Markov chain. Actually, this fact is the only requirement for the
existence of a recursive realization of the filter, with Lemma 3 providing
a sufficient condition, ensuring asymptotic optimality. The resulting
recursive scheme is summarized in the following result. The proof
is omitted, since it involves standard arguments in nonlinear filtering,
similar to the ones employed in the derivation of the filtering recursions
for a partially observed Markov chain with finite state space \cite{Elliott1994Exact,Elliott1994Hidden,Cappe_BOOK2005},
as previously mentioned.
\begin{thm}
\label{Markovian_Filter}\textbf{\textup{(The Markovian Filter)}}
Consider the Markovian state approximation $X_{t}^{L_{S}}$ and define
$\boldsymbol{E}_{t}\triangleq\mathbb{E}_{\widetilde{{\cal P}}}\left\{ \left.{\cal Q}_{L_{S}}^{e}\left(X_{t}^{L_{S}}\right)\Lambda_{t}^{L_{S}}\right|\mathscr{Y}_{t}\right\} \in\mathbb{R}^{L_{S}\times1}$,
for all $t\in\mathbb{N}$. Then, under the appropriate assumptions
(Lipschitz property of Lemma \ref{Markovian_Convergence}), the asymptotically
optimal in $L_{S}$ approximate grid based filter ${\cal E}^{L_{S}}\left(\left.X_{t}\right|\mathscr{Y}_{t}\right)$
can be expressed as
\begin{equation}
{\cal E}^{L_{S}}\left(\left.X_{t}\right|\mathscr{Y}_{t}\right)\equiv\dfrac{{\bf X}\boldsymbol{E}_{t}}{\left\Vert \boldsymbol{E}_{t}\right\Vert _{1}},\quad\forall t\in\mathbb{N},\label{eq:FILTER_1}
\end{equation}
where the process $\boldsymbol{E}_{t}$ satisfies the linear recursion
\begin{equation}
\boldsymbol{E}_{t}\equiv\boldsymbol{\Lambda}_{t}\boldsymbol{P}\boldsymbol{E}_{t-1},\quad\forall t\in\mathbb{N}.
\end{equation}
The filter is initialized setting $\boldsymbol{E}_{-1}\triangleq\mathbb{E}_{{\cal P}}\left\{ {\cal Q}_{L_{S}}^{e}\left(X_{-1}^{L_{S}}\right)\right\} $.\end{thm}
\begin{rem}
It is worth mentioning that, although \textit{formally similar} to,
the approximate filter introduced in Theorem \ref{Markovian_Filter}
\textit{does not} refer to a Markov chain with finite state space,
because the observations process utilized in the filtering iterations
corresponds to that of the \textit{real} partially observable system
under consideration. The quantity ${\cal E}^{L_{S}}\left(\left.X_{t}\right|\mathscr{Y}_{t}\right)$
\textit{does not constitute a conditional expectation} of the Markov
chain associated with $\boldsymbol{P}$, because the latter process
does not follow the probability law of the true state process $X_{t}$.\hfill{}\ensuremath{\blacksquare}
\end{rem}

\begin{rem}
In fact, $\boldsymbol{E}_{t}$ may be interpreted as a vector encoding
an \textit{unnormalized} point mass function, which, \textit{roughly
speaking}, expresses \textit{the belief of the quantized state}, given
the observations up to and including time $t$. Normalization by $\left\Vert \boldsymbol{E}_{t}\right\Vert _{1}$
corresponds precisely to a point mass function.\hfill{}\ensuremath{\blacksquare}
\end{rem}

\begin{rem}
For the benefit of the reader, we should mention that the Markovian
filter considered above essentially coincides with the approximate
grid based filter reported in (\cite{PARTICLE2002tutorial}, Section
IV.B), although the construction of the two filters is different:
the former is constructed via a Markovian quantization of the state,
whereas the latter \cite{PARTICLE2002tutorial} is based on a ``quasi-marginal''
approach (compare with \eqref{eq:marginal_MATRIX}). Nevertheless,
given our assumptions on the HMM under consideration, both formulations
result in exactly the same transition matrix. Therefore, the optimality
properties of the Markovian filter are indeed inherited by the grid
based filter described in \cite{PARTICLE2002tutorial}. \hfill{}\ensuremath{\blacksquare}
\end{rem}

\subsection{Marginal Quantization}

We now move on to the case of marginal quantizations. In order to
be able to come up with a simple, Markov chain based, recursive filtering
scheme, as in the case of Markovian quantizations previously treated,
it turns out that a further assumption is required, this time concerning
the stochastic kernel of the Markov process $X_{t}$. But before embarking
on the relevant analysis, let us present some essential definitions. 

First, for any process $X_{t}$, we will say that a sequence of functions
$\left\{ f_{n}\left(\cdot\right)\right\} _{n}$ is ${\cal P}_{X_{t}}\hspace{-2pt}\hspace{-2pt}-\hspace{-2pt}UI$,
if $\left\{ f_{n}\left(\cdot\right)\right\} _{n}$ is Uniformly Integrable
with respect to the pushforward measure induced by $X_{t}$, ${\cal P}_{X_{t}}$,
where $t\in\mathbb{N}\cup\left\{ -1\right\} $, i.e.,
\begin{equation}
\lim_{K\rightarrow\infty}\sup_{n}\int_{\left\{ \left|f_{n}\left(\boldsymbol{x}\right)\right|>K\right\} }\left|f_{n}\left(\boldsymbol{x}\right)\right|{\cal P}_{X_{t}}\left(\text{d}\boldsymbol{x}\right)\equiv0.
\end{equation}
Second, given $L_{S}$, recall from Section \ref{sec:Uniform-State-Quantizations}.A
that the set $\Pi_{L_{S}}$ contains as members all quantization regions
of ${\cal Z}$, ${\cal Z}_{L_{S}}^{j}$, $j\in\mathbb{N}_{L_{S}}^{+}$.
Then, given the stochastic kernel ${\cal K}\left(\left.\cdot\right|\cdot\right)$
associated with the time invariant transitions of $X_{t}$ and for
each $L_{S}\in\mathbb{N}^{+}$, we define the \textit{cumulative kernel}
\begin{flalign}
\hspace{-2pt}\hspace{-2pt}{\cal K}\hspace{-2pt}\left(\left.{\cal A}\right|\hspace{-2pt}\in\hspace{-2pt}{\cal Z}_{L_{S}}\left(\boldsymbol{x}\right)\right) & \hspace{-2pt}\triangleq\hspace{-2pt}\dfrac{{\displaystyle \int_{{\cal Z}_{L_{S}}\left(\boldsymbol{x}\right)}{\cal K}\left(\left.{\cal A}\right|\boldsymbol{\theta}\right){\cal P}_{X_{t-1}}\left(\text{d}\boldsymbol{\theta}\right)}}{{\cal P}\left(X_{t-1}\in{\cal Z}_{L_{S}}\left(\boldsymbol{x}\right)\right)}\nonumber \\
 & \hspace{-2pt}\equiv\hspace{-2pt}\dfrac{\mathbb{E}\left\{ {\cal K}\left(\left.{\cal A}\right|X_{t-1}\right)\mathds{1}_{\left\{ X_{t-1}\in{\cal Z}_{L_{S}}\left(\boldsymbol{x}\right)\right\} }\right\} }{\mathbb{E}\left\{ \mathds{1}_{\left\{ X_{t-1}\in{\cal Z}_{L_{S}}\left(\boldsymbol{x}\right)\right\} }\right\} }\nonumber \\
 & \hspace{-2pt}\equiv\mathbb{E}\left\{ \hspace{-2pt}\left.{\cal K}\left(\left.{\cal A}\right|X_{t-1}\right)\right|X_{t-1}\hspace{-2pt}\in\hspace{-2pt}{\cal Z}_{L_{S}}\left(\boldsymbol{x}\right)\right\} \hspace{-2pt},\label{eq:CUM_KERNEL}
\end{flalign}
for all Borel ${\cal A}\in\mathscr{B}\hspace{-2pt}\left(\mathbb{R}^{M\times1}\right)$
and all $\boldsymbol{x}\in{\cal Z}$, where ${\cal Z}_{L_{S}}\hspace{-2pt}\left(\boldsymbol{x}\right)\in\Pi_{L_{S}}$
denotes the unique quantization region, which includes $\boldsymbol{x}$.
Note that if $\boldsymbol{x}$ is substituted by $X_{t-1}\left(\omega\right)$,
the resulting quantity ${\cal Z}_{L_{S}}\left(X_{t-1}\left(\omega\right)\right)$
constitutes an $\mathscr{X}_{t}$-predictable \textit{set-valued random
element}. Now, if, for any $\boldsymbol{x}\in{\cal Z}$, ${\cal K}\left(\left.\cdot\right|\boldsymbol{x}\right)$
admits a stochastic kernel density $\kappa:\mathbb{R}^{M\times1}\times\mathbb{R}^{M\times1}\mapsto\mathbb{R}_{+},$
suggestively denoted as $\kappa\left(\left.\boldsymbol{y}\right|\boldsymbol{x}\right)$,
we define, in exactly the same fashion as above, the \textit{cumulative
kernel density}
\begin{equation}
\kappa\left(\left.\boldsymbol{y}\right|\hspace{-2pt}\in\hspace{-2pt}{\cal Z}_{L_{S}}\left(\boldsymbol{x}\right)\right)\triangleq\mathbb{E}\left\{ \hspace{-2pt}\left.\kappa\left(\left.\boldsymbol{y}\right|X_{t-1}\right)\right|X_{t-1}\hspace{-2pt}\in\hspace{-2pt}{\cal Z}_{L_{S}}\left(\boldsymbol{x}\right)\right\} ,\label{eq:CUM_DENSITY}
\end{equation}
for all $\boldsymbol{y}\in\mathbb{R}^{M\times1}$. The fact that $\kappa\left(\left.\cdot\right|\hspace{-2pt}\in\hspace{-2pt}{\cal Z}_{L_{S}}\left(\boldsymbol{x}\right)\right)$
is indeed a Radon-Nikodym derivative of ${\cal K}\hspace{-2pt}\left(\left.\cdot\right|\hspace{-2pt}\in\hspace{-2pt}{\cal Z}_{L_{S}}\left(\boldsymbol{x}\right)\right)$
readily follows by definition of the latter and Fubini's Theorem.
\begin{rem}
Observe that, although integration is with respect to ${\cal P}_{X_{t-1}}$
on the RHS of \eqref{eq:CUM_KERNEL}, ${\cal K}\hspace{-2pt}\left(\left.\cdot\right|\hspace{-2pt}\in\hspace{-2pt}{\cal Z}_{L_{S}}\left(\cdot\right)\right)$
is time invariant. This is due to stationarity of $X_{t}$, as assumed
in the beginning of Section \ref{sec:CORE_SECTION}, implying time
invariance of the marginal measure ${\cal P}_{X_{t}}$, for all $t\in\mathbb{N}\cup\left\{ -1\right\} $.
Additionally, for each $\boldsymbol{x}\in{\cal Z}$, when ${\cal A}$
is restricted to $\Pi_{L_{S}}$, ${\cal K}\hspace{-2pt}\left(\left.{\cal A}\right|\hspace{-2pt}\in\hspace{-2pt}{\cal Z}_{L_{S}}\left(\boldsymbol{x}\right)\right)$
corresponds to an entry of the (time invariant) matrix $\boldsymbol{P}$,
also defined earlier. In the general case, where the aforementioned
cumulative kernel is time varying, all subsequent analysis continues
to be valid, just by adding additional notational complexity.\hfill{}\ensuremath{\blacksquare}
\end{rem}
In respect to the relevant assumption required on ${\cal K}\left(\left.\cdot\right|\cdot\right)$,
as asserted above, let us now present the following definition.
\begin{defn}
\label{Cond_Reg}\textbf{(Cumulative Conditional Regularity of Markov
Kernels)} Consider the kernel ${\cal K}\left(\left.\cdot\right|\cdot\right)$,
associated with $X_{t}$, for all $t\in\mathbb{N}$. We say that ${\cal K}\left(\left.\cdot\right|\cdot\right)$
is \textit{Conditionally Regular of Type I (CRT I)}, if, for ${\cal P}_{X_{t}}\hspace{-2pt}\equiv\hspace{-2pt}{\cal P}_{X_{-1}}$-almost
all $\boldsymbol{x}$, there exists a ${\cal P}_{X_{-1}}\hspace{-2pt}\hspace{-2pt}-UI$
sequence $\left\{ \delta_{n}^{I}\left(\cdot\right)\ge0\right\} _{n\in\mathbb{N}^{+}}$
with $\delta_{n}^{I}\left(\cdot\right)\stackrel[n\rightarrow\infty]{a.e.}{\longrightarrow}0$,
such that
\begin{equation}
{\displaystyle \sup_{{\cal A}\in\Pi_{L_{S}}}}\left|{\cal K}\left(\left.{\cal A}\right|\boldsymbol{x}\right)-{\cal K}\left(\left.{\cal A}\right|\hspace{-2pt}\in\hspace{-2pt}{\cal Z}_{L_{S}}\left(\boldsymbol{x}\right)\right)\right|\le\dfrac{\delta_{L_{S}}^{I}\left(\boldsymbol{x}\right)}{L_{S}}.
\end{equation}
\renewcommand{\arraystretch}{1}If, further, for ${\cal P}_{X_{-1}}$-almost
all $\boldsymbol{x}$, the measure ${\cal K}\left(\left.\cdot\right|\boldsymbol{x}\right)$
admits a density $\kappa\left(\left.\cdot\right|\boldsymbol{x}\right),$
and if there exists another ${\cal P}_{X_{-1}}\hspace{-2pt}\hspace{-2pt}-UI$
sequence $\left\{ \delta_{n}^{II}\left(\cdot\right)\ge0\right\} _{n\in\mathbb{N}^{+}}$
with $\delta_{n}^{II}\left(\cdot\right)\stackrel[n\rightarrow\infty]{a.e.}{\longrightarrow}0$,
such that
\begin{equation}
\underset{\boldsymbol{y}\in\mathbb{R}^{M\times1}}{\mathrm{ess}\hspace{0.2em}\mathrm{sup}}\left|\kappa\left(\left.\boldsymbol{y}\right|\boldsymbol{x}\right)-\kappa\left(\left.\boldsymbol{y}\right|\hspace{-2pt}\in\hspace{-2pt}{\cal Z}_{L_{S}}\left(\boldsymbol{x}\right)\right)\right|\le\delta_{L_{S}}^{II}\left(\boldsymbol{x}\right),
\end{equation}
${\cal K}\left(\left.\cdot\right|\cdot\right)$ is called \textit{Conditionally
Regular of Type II (CRT II)}. In any case, $X_{t}$ will also be called
conditionally regular.
\end{defn}
\vspace{-0.6cm}

\noindent \begin{center}
\rule[0.5ex]{0.5\columnwidth}{0.5pt}
\par\end{center}

\vspace{-0.2cm}

A consequence of conditional regularity is the following Martingale
Difference (MD) \cite{Segall1976,Elliott1994Hidden} type representation
of the marginally quantized process ${\cal Q}_{L_{S}}^{e}\left(X_{t}^{L_{S}}\right)$.
\begin{lem}
\label{Doob_Lemma-1}\textbf{\textup{(Semirecursive MD-type Representation
of Marginal Quantizations)}} Assume that the state process $X_{t}$
is conditionally regular. Then, the quantized process ${\cal Q}_{L_{S}}^{e}\left(X_{t}^{L_{S}}\right)$
admits the representation
\begin{flalign}
{\cal Q}_{L_{S}}^{e}\left(X_{t}^{L_{S}}\right) & \equiv\boldsymbol{P}{\cal Q}_{L_{S}}^{e}\left(X_{t-1}^{L_{S}}\right)+\boldsymbol{{\cal M}}_{t}^{e}+\boldsymbol{\varepsilon}_{t}^{L_{S}},
\end{flalign}
where, under the base measure $\widetilde{{\cal P}}$,$\boldsymbol{{\cal M}}_{t}^{e}\in\mathbb{R}^{L_{S}\times1}$
constitutes an $\mathscr{X}_{t}$-MD process and $\boldsymbol{\varepsilon}_{t}^{L_{S}}\in\mathbb{R}^{L_{S}\times1}$
constitutes a $\left\{ \mathscr{X}_{t}\right\} $-predictable process,
such that 
\begin{itemize}
\item if $X_{t}$ is CRT I, then
\begin{equation}
\left\Vert \boldsymbol{\varepsilon}_{t}^{L_{S}}\right\Vert _{1}\le\delta_{L_{S}}^{I}\left(X_{t-1}\right)\underset{L_{S}\rightarrow\infty}{\longrightarrow}0,\;\widetilde{{\cal P}}-a.s.,
\end{equation}

\item whereas, if $X_{t}$ is CRT II, then
\begin{equation}
\hspace{-2pt}\hspace{-2pt}\hspace{-2pt}\left\Vert \boldsymbol{\varepsilon}_{t}^{L_{S}}\right\Vert _{1}\le\left|b-a\right|^{M}\delta_{L_{S}}^{II}\left(X_{t-1}\right)\underset{L_{S}\rightarrow\infty}{\longrightarrow}0,\;\widetilde{{\cal P}}-a.s.,
\end{equation}

\end{itemize}
everywhere in time.\end{lem}
\begin{proof}[Proof of Lemma \ref{Doob_Lemma-1}]
See Appendix B.
\end{proof}
Now, consider an auxiliary Markov chain $Z_{t}^{L_{S}}\in{\cal V}_{L_{S}}$,
with $\boldsymbol{P}$ (defined as in \eqref{eq:Trans}) as its transition
matrix and with initial distribution to be specified. Of course, $Z_{t}^{L_{S}}$
can be represented as $Z_{t}^{L_{S}}\equiv\boldsymbol{P}Z_{t-1}^{L_{S}}+\widetilde{\boldsymbol{{\cal M}}}_{t}^{e}$,
where $\widetilde{\boldsymbol{{\cal M}}}_{t}^{e}\in\mathbb{R}^{L_{S}\times1}$
constitutes a $\mathscr{Z}_{t}$-MD process, with $\left\{ \mathscr{Z}_{t}\right\} _{t\in\mathbb{N}}$
being the complete natural filtration generated by $Z_{t}^{L_{S}}$. 

Due to the existence of the ``bias'' process $\boldsymbol{\varepsilon}_{t}^{L_{S}}$
in the martingale difference representation of ${\cal Q}_{L_{S}}^{e}\left(X_{t}^{L_{S}}\right)$
(see Lemma \ref{Doob_Lemma-1}), the direct derivation of a filtering
recursion for this process is difficult. However, it turns out that
the approximate filter involving the marginal state quantization $X_{t}^{L_{S}}$,
${\cal E}^{L_{S}}\left(\left.X_{t}\right|\mathscr{Y}_{t}\right)$,
can be further approximated by the also approximate filter
\begin{equation}
\widetilde{{\cal E}}^{L_{S}}\left(\left.X_{t}\right|\mathscr{Y}_{t}\right)\triangleq\dfrac{{\bf X}\mathbb{E}_{\widetilde{{\cal P}}}\left\{ \left.Z_{t}^{L_{S}}\Lambda_{t}^{Z,L_{S}}\right|\mathscr{Y}_{t}\right\} }{\mathbb{E}_{\widetilde{{\cal P}}}\left\{ \left.\Lambda_{t}^{Z,L_{S}}\right|\mathscr{Y}_{t}\right\} },
\end{equation}
for all $t\in\mathbb{N}$, where the functional $\Lambda_{t}^{Z,L_{S}}$
is defined exactly like $\Lambda_{t}^{L_{S}},$ but replacing $X_{t}^{L_{S}}$
with $Z_{t}^{L_{S}}$. This latter filter indeed admits the recursive
representation proposed in Theorem \ref{Markovian_Filter} (with $\boldsymbol{P}$
defined as in \eqref{eq:Trans}, reflecting the choice of a marginal
state approximation).

Consequently, if we are interested in the asymptotic behavior of the
approximation error between $\widetilde{{\cal E}}^{L_{S}}\left(\left.X_{t}\right|\mathscr{Y}_{t}\right)$
and the original nonlinear filter $\mathbb{E}_{{\cal P}}\left\{ \left.X_{t}\right|\mathscr{Y}_{t}\right\} $,
we can write
\begin{multline}
\left\Vert \mathbb{E}_{{\cal P}}\left\{ \left.X_{t}\right|\mathscr{Y}_{t}\right\} -\widetilde{{\cal E}}^{L_{S}}\left(\left.X_{t}\right|\mathscr{Y}_{t}\right)\right\Vert _{1}\\
\le\left\Vert \mathbb{E}_{{\cal P}}\left\{ \left.X_{t}\right|\mathscr{Y}_{t}\right\} -{\cal E}^{L_{S}}\left(\left.X_{t}\right|\mathscr{Y}_{t}\right)\right\Vert _{1}+\left\Vert {\cal E}^{L_{S}}\left(\left.X_{t}\right|\mathscr{Y}_{t}\right)-\widetilde{{\cal E}}^{L_{S}}\left(\left.X_{t}\right|\mathscr{Y}_{t}\right)\right\Vert _{1}.\label{eq:filters_proof_7}
\end{multline}
However, from Theorem \ref{CONVERGENCE_THEOREM_2}, we know that,
under the respective conditions,
\begin{equation}
\sup_{t\in\mathbb{N}_{T}}\sup_{\omega\in\widehat{\Omega}_{T}}\left\Vert {\cal E}^{L_{S}}\left(\left.X_{t}\right|\mathscr{Y}_{t}\right)-\mathbb{E}_{{\cal P}}\left\{ \left.X_{t}\right|\mathscr{Y}_{t}\right\} \right\Vert _{1}\underset{L_{S}\rightarrow\infty}{\longrightarrow}0.
\end{equation}
Therefore, if we show that error between ${\cal E}^{L_{S}}\left(\left.X_{t}\right|\mathscr{Y}_{t}\right)$
and $\widetilde{{\cal E}}^{L_{S}}\left(\left.X_{t}\right|\mathscr{Y}_{t}\right)$
vanishes in the above sense, then, $\widetilde{{\cal E}}^{L_{S}}\left(\left.X_{t}\right|\mathscr{Y}_{t}\right)$
will converge to $\mathbb{E}_{{\cal P}}\left\{ \left.X_{t}\right|\mathscr{Y}_{t}\right\} $,
also in the same sense. It turns out that if $X_{t}$ is conditionally
regular, the aforementioned desired statement always holds, as follows.
\begin{lem}
\label{Convergence_FAKE}\textbf{\textup{(Convergence of Approximate
Filters)}} For any natural $T<\infty$, suppose that the state process
$X_{t}$ is conditionally regular and that the initial measure of
the chain $Z_{t}^{L_{S}}$ is chosen such that 
\begin{equation}
\mathbb{E}_{\widetilde{{\cal P}}}\left\{ Z_{-1}^{L_{S}}\right\} \equiv\mathbb{E}_{{\cal P}}\left\{ {\cal Q}_{L_{S}}^{e}\left(X_{-1}^{L_{S}}\right)\right\} .
\end{equation}
Then, for the same measurable subset $\widehat{\Omega}_{T}\subseteq\Omega$
of Theorem \ref{CONVERGENCE_THEOREM}, it is true that
\begin{equation}
\sup_{t\in\mathbb{N}_{T}}\sup_{\omega\in\widehat{\Omega}_{T}}\left\Vert {\cal E}^{L_{S}}\left(\left.X_{t}\right|\mathscr{Y}_{t}\right)-\widetilde{{\cal E}}^{L_{S}}\left(\left.X_{t}\right|\mathscr{Y}_{t}\right)\right\Vert _{1}\underset{L_{S}\rightarrow\infty}{\longrightarrow}0.
\end{equation}
Additionally, under the same setting, it follows that
\begin{equation}
\sup_{t\in\mathbb{N}_{T}}\sup_{\omega\in\widehat{\Omega}_{T}}\left\Vert \widetilde{{\cal E}}^{L_{S}}\left(\left.X_{t}\right|\mathscr{Y}_{t}\right)-\mathbb{E}_{{\cal P}}\left\{ \left.X_{t}\right|\mathscr{Y}_{t}\right\} \right\Vert _{1}\underset{L_{S}\rightarrow\infty}{\longrightarrow}0.
\end{equation}
\end{lem}
\begin{proof}[Proof of Lemma \ref{Convergence_FAKE}]
See Appendix C.
\end{proof}
Finally, the next theorem establishes precisely the form of the recursive
grid based filter, employing the marginal quantization of the state.
\begin{thm}
\label{OUR_FILTER2}\textbf{\textup{(The Marginal Filter)}} Consider
the marginal state approximation $X_{t}^{L_{S}}$ and suppose that
the state process $X_{t}$ is conditionally regular. Then, for each
$t\in\mathbb{N}$, the asymptotically optimal in $L_{S}$ approximate
filtering operator $\widetilde{{\cal E}}^{L_{S}}\left(\left.X_{t}\right|\mathscr{Y}_{t}\right)$
can be recursively expressed exactly as in Theorem \ref{Markovian_Filter},
with initial conditions as in Lemma \ref{Convergence_FAKE} and transition
matrix $\boldsymbol{P}$ defined as in \eqref{eq:Trans}.\end{thm}
\begin{rem}
\textbf{(Weak Conditional Regularity)} All the derivations presented
above are still valid if, in the definition of conditional regularity
(Definition 2), one replaces almost everywhere convergence of the
sequences $\left\{ \delta_{n}^{I}\left(\cdot\right)\right\} _{n}$
and $\left\{ \delta_{n}^{II}\left(\cdot\right)\right\} _{n}$ with
\textit{convergence in probability}. This is due to the fact that
uniform integrability plus convergence in measure are necessary and
sufficient conditions for showing convergence in ${\cal L}_{1}$ (for
finite measure spaces). Consequently, if we focus on, for instance,
CRT I (CRT II is similar), it is easy to see that in order to ensure
asymptotic consistency of the marginal approximate filter in the sense
of Theorem \ref{OUR_FILTER2}, it suffices that, for ${\cal A}\in\Pi_{L_{S}}$
and for any $\epsilon>0$,
\begin{equation}
{\cal P}_{X_{t-1}}\hspace{-2pt}\left(\hspace{-1.5pt}\sup_{{\cal A}}\hspace{-2pt}\left|{\cal K}\hspace{-2pt}\left(\left.{\cal A}\right|\hspace{-2pt}\boldsymbol{x}\right)\hspace{-2pt}-\hspace{-2pt}{\cal K}\hspace{-2pt}\left(\left.{\cal A}\right|\hspace{-2pt}\in\hspace{-2pt}{\cal Z}_{L_{S}}\hspace{-2pt}\left(\boldsymbol{x}\right)\right)\right|\hspace{-2pt}>\hspace{-2pt}\dfrac{\epsilon}{L_{S}}\hspace{-1.5pt}\right)\hspace{-2pt}\hspace{-2pt}\underset{L_{S}\rightarrow\infty}{\longrightarrow}\hspace{-2pt}\hspace{-2pt}0,\hspace{-2pt}
\end{equation}
for all $t\in\mathbb{N}$ (in general), given the stochastic kernel
${\cal K}\left(\left.\cdot\right|\cdot\right)$ and for the desired
choice of the quantizer ${\cal Q}_{L_{S}}\left(\cdot\right)$. Here,
the ${\cal P}_{X_{t}}\hspace{-2pt}-\hspace{-2pt}UI$ sequence $\left\{ \delta_{n}^{I}\left(\cdot\right)\right\} _{n}$
is identified as
\begin{equation}
\delta_{L_{S}}^{I}\left(\boldsymbol{x}\right)\equiv\sup_{{\cal A}\in\Pi_{L_{S}}}L_{S}\left|{\cal K}\left(\left.{\cal A}\right|\boldsymbol{x}\right)\hspace{-2pt}-\hspace{-2pt}{\cal K}\left(\left.{\cal A}\right|\hspace{-2pt}\in\hspace{-2pt}{\cal Z}_{L_{S}}\left(\boldsymbol{x}\right)\right)\right|,
\end{equation}
for all $L_{S}\in\mathbb{N}^{+}$ and for almost all $\boldsymbol{x}\in\mathbb{R}^{M\times1}$.
In other words, it is required that, for any $\epsilon>0$,
\begin{equation}
\sup_{{\cal A}\in\Pi_{L_{S}}}\left|{\cal K}\left(\left.{\cal A}\right|X_{t-1}\right)\hspace{-2pt}-\hspace{-2pt}{\cal K}\left(\left.{\cal A}\right|\hspace{-2pt}\in\hspace{-2pt}{\cal Z}_{L_{S}}\left(X_{t-1}\right)\right)\right|\hspace{-2pt}\le\hspace{-2pt}\dfrac{\epsilon}{L_{S}},
\end{equation}
with probability at least $1\hspace{-2pt}-\hspace{-2pt}\pi_{t-1}\hspace{-2pt}\left(\epsilon,L_{S}\right)$,
for all $t\hspace{-2pt}\in\hspace{-2pt}\mathbb{N}$ (in general),
where, for each $t,$ $\left\{ \pi_{t-1}\hspace{-2pt}\left(\epsilon,n\right)\right\} _{n\in\mathbb{N}^{+}}$
constitutes a sequence vanishing at infinity. This is a considerably
weaker form of conditional regularity, as stated in Definition 2.\hfill{}\ensuremath{\blacksquare}
\end{rem}

\subsection{Extensions: State Functionals \& Approximate Prediction}

All the results presented so far can be extended as follows. First,
if $\left\{ \boldsymbol{\phi}_{t}:\mathbb{R}^{M\times1}\mapsto\mathbb{R}^{M_{\phi_{t}}\times1}\right\} _{t\in\mathbb{N}}$
is a family of bounded and continuous functions, it is easy to show
that every relevant theorem presented so far is still true if one
replaces $X_{t}$ by $\boldsymbol{\phi}_{t}\left(X_{t}\right)$ in
the respective formulations of the approximate filters discussed.
This is made possible by observing that \eqref{eq:Changed_Expectation}
still holds if we replace $X_{t}$ by $\boldsymbol{\phi}_{t}\left(X_{t}\right)$,
by invoking the Continuous Mapping Theorem and using the boundedness
of $\boldsymbol{\phi}_{t}\left(X_{t}\right)$, instead of the boundedness
of $X_{t}$, whenever required.

Second, exploiting very similar arguments as in the previous sections,
it is possible to derive asymptotically optimal $\rho$-step state
predictors, where $\rho>0$ denotes the desired (and finite) prediction
horizon. In particular, under the usual assumptions \cite{Elliott1994Hidden},
it is easy to show that, as in the filtering case, the optimal nonlinear
temporal predictor $\mathbb{E}_{{\cal P}}\left\{ \left.X_{t+\rho}\right|\mathscr{Y}_{t}\right\} $
can be expressed through the Feynman-Kac type of formula
\begin{equation}
\mathbb{E}_{{\cal P}}\left\{ \left.X_{t+\rho}\right|\mathscr{Y}_{t}\right\} \equiv\dfrac{\mathbb{E}_{\widetilde{{\cal P}}}\left\{ \left.X_{t+\rho}\Lambda_{t}\right|\mathscr{Y}_{t}\right\} }{\mathbb{E}_{\widetilde{{\cal P}}}\left\{ \left.\Lambda_{t}\right|\mathscr{Y}_{t}\right\} },\quad\forall t\in\mathbb{N}.
\end{equation}
Therefore, in analogy to \eqref{eq:main_approx}, it is reasonable
to consider grid based approximations of the form
\begin{equation}
{\cal E}^{L_{S}}\left(\left.X_{t+\rho}\right|\mathscr{Y}_{t}\right)\equiv\dfrac{{\bf X}\mathbb{E}_{\widetilde{{\cal P}}}\left\{ \left.{\cal Q}_{L_{S}}^{e}\left(X_{t+\rho}^{L_{S}}\right)\Lambda_{t}^{L_{S}}\right|\mathscr{Y}_{t}\right\} }{\mathbb{E}_{\widetilde{{\cal P}}}\left\{ \left.\Lambda_{t}^{L_{S}}\right|\mathscr{Y}_{t}\right\} },\label{eq:Pred_1}
\end{equation}
for all $t\in\mathbb{N}$. Focusing on marginal state quantizations
(the Markovian case is similar, albeit easier), then, exploiting Lemma
\ref{Doob_Lemma-1} and using induction, it is easy to show that
\begin{align}
{\cal Q}_{L_{S}}^{e}\left(X_{t+\rho}^{L_{S}}\right) & \equiv\boldsymbol{P}^{\rho}{\cal Q}_{L_{S}}^{e}\left(X_{t}^{L_{S}}\right)+\sum_{i=1}^{\rho}\boldsymbol{P}^{\rho-i}\boldsymbol{{\cal M}}_{t+i}^{e}+\sum_{i=1}^{\rho}\boldsymbol{P}^{\rho-i}\boldsymbol{\varepsilon}_{t+i}^{L_{S}},\quad\forall t\in\mathbb{N}.
\end{align}
Thus, using simple properties of MD sequences, it follows that the
numerator of the fraction on the RHS of \eqref{eq:Pred_1} can be
decomposed as
\begin{align}
{\cal E}^{L_{S}}\left(\left.X_{t+\rho}\right|\mathscr{Y}_{t}\right) & \equiv\dfrac{{\bf X}\boldsymbol{P}^{\rho}\mathbb{E}_{\widetilde{{\cal P}}}\left\{ \left.{\cal Q}_{L_{S}}^{e}\left(X_{t+\rho}^{L_{S}}\right)\Lambda_{t}^{L_{S}}\right|\mathscr{Y}_{t}\right\} }{\mathbb{E}_{\widetilde{{\cal P}}}\left\{ \left.\Lambda_{t}^{L_{S}}\right|\mathscr{Y}_{t}\right\} }+\dfrac{{\displaystyle {\bf X}\sum_{i=1}^{\rho}\boldsymbol{P}^{\rho-i}\mathbb{E}_{\widetilde{{\cal P}}}\left\{ \left.\boldsymbol{\varepsilon}_{t+i}^{L_{S}}\Lambda_{t}^{L_{S}}\right|\mathscr{Y}_{t}\right\} }}{\mathbb{E}_{\widetilde{{\cal P}}}\left\{ \left.\Lambda_{t}^{L_{S}}\right|\mathscr{Y}_{t}\right\} }.\label{eq:Pred_2}
\end{align}
The first term on the RHS of \eqref{eq:Pred_2} is analyzed exactly
as in the proof of Lemma \ref{Convergence_FAKE}. For the second term,
it is true that
\begin{align}
\left\Vert {\displaystyle {\bf X}\sum_{i=1}^{\rho}\boldsymbol{P}^{\rho-i}\mathbb{E}_{\widetilde{{\cal P}}}\left\{ \left.\boldsymbol{\varepsilon}_{t+i}^{L_{S}}\Lambda_{t}^{L_{S}}\right|\mathscr{Y}_{t}\right\} }\right\Vert _{1} & \le M\gamma\sqrt{\lambda_{inf}^{-N\left(t+1\right)}}\sum_{i=1}^{\rho}\mathbb{E}_{\widetilde{{\cal P}}}\left\{ \left\Vert \boldsymbol{\varepsilon}_{t+i}^{L_{S}}\right\Vert _{1}\right\} ,
\end{align}
which can be treated as an extra error term, also in the fashion of
Lemma \ref{Convergence_FAKE}.

Putting it altogether (state functionals plus prediction), the following
general theorem holds, covering every aspect of the investigation
presented in this paper.
\begin{thm}
\label{Functionals}\textbf{\textup{(Grid Based Filtering/Prediction
\& Functionals of the State)}} For any deterministic functional family
$\left\{ \boldsymbol{\phi}_{t}:\mathbb{R}^{M\times1}\mapsto\mathbb{R}^{M_{\phi_{t}}\times1}\right\} _{t\in\mathbb{N}}$
with bounded and continuous members and any finite prediction horizon
$\rho\ge0$, the strictly optimal filter and $\rho$-step predictor
of the \textbf{transformed} \textbf{process} $\boldsymbol{\phi}_{t}\left(X_{t}\right)$
can be approximated as
\begin{equation}
{\cal E}^{L_{S}}\left(\left.\boldsymbol{\phi}_{t+\rho}\left(X_{t+\rho}\right)\right|\mathscr{Y}_{t}\right)\triangleq\boldsymbol{\Phi}_{t+\rho}\dfrac{\boldsymbol{P}^{\rho}\boldsymbol{E}_{t}}{\left\Vert \boldsymbol{E}_{t}\right\Vert _{1}}\in\mathbb{R}^{M_{\phi_{t}}\times1},\label{eq:Spatial_2}
\end{equation}
for all $t\in\mathbb{N}$, where the process $E_{t}\in\mathbb{R}^{L_{S}\times1}$
can be recursively evaluated as in Theorem \ref{Markovian_Filter},
$\boldsymbol{P}$ is defined according to the chosen state quantization
and 
\begin{equation}
\boldsymbol{\Phi}_{t+\rho}\triangleq\left[\boldsymbol{\phi}_{t+\rho}\left(\boldsymbol{x}_{L_{S}}^{1}\right)\,\ldots\,\boldsymbol{\phi}_{t+\rho}\left(\boldsymbol{x}_{L_{S}}^{L_{S}}\right)\right]\in\mathbb{R}^{M_{\phi_{t}}\times L_{S}}.
\end{equation}
Additionally, under the appropriate assumptions (see Lemma \ref{Markovian_Convergence}
and Lemma \ref{Convergence_FAKE}, respectively) the approximate filter
is asymptotically optimal. in the sense of Theorem \ref{CONVERGENCE_THEOREM}.
\end{thm}

\begin{rem}
\label{rem:UNIFORMITY}As Theorem \ref{Functionals} clearly states,
for each choice of state functionals and any finite prediction horizon,
convergence of the respective approximate grid based filters is in
the sense of Theorem \ref{CONVERGENCE_THEOREM}. This implies the
existence of an exceptional measurable set of measure almost unity,
inside of which convergence is in the uniform sense. It is important
to emphasize that this exceptional event, $\widehat{\Omega}_{T}$,
as well as its measure, are \textit{independent of the particular
choice of both the bounded family $\left\{ \boldsymbol{\phi}_{t}\right\} _{t}$
and the prediction horizon $\rho$}. This fact can be easily verified
by a quick detour of the proof of Theorem \ref{CONVERGENCE_THEOREM}
in \cite{KalPetNonlinear_2015}. In particular, for any fixed choice
of $T$, $\widehat{\Omega}_{T}$ characterizes exclusively the growth
of the observations ${\bf y}_{t}$, which are the same regardless
of filtering, prediction, or any functional imposed on the state.
Therefore, \textit{stochastically uniform (in $\widehat{\Omega}_{T}$)
convergence of one estimator implies stochastically uniform convergence
of any other estimator}, within any class of estimators, constructed
employing any uniformly bounded and continuous class of functionals
of the state and finite prediction horizons.\hfill{}\ensuremath{\blacksquare}
\end{rem}

\subsection{Filter Performance}

The uncertainty of a filtering estimator can be quantified via its
\textit{posterior} quadratic deviation from the true state, at each
time $t$. This information is encoded into the posterior covariance
matrix
\begin{align}
\hspace{-2pt}\hspace{-2pt}\hspace{-2pt}\hspace{-2pt}\hspace{-2pt}\hspace{-2pt}\mathbb{V}\hspace{-2pt}\left\{ \hspace{-2pt}\left.X_{t}\right|\mathscr{Y}_{t}\right\} \equiv\mathbb{E}\left\{ \hspace{-2pt}\left.X_{t}X_{t}^{\boldsymbol{T}}\right|\mathscr{Y}_{t}\right\} \hspace{-2pt}-\hspace{-2pt}\mathbb{E}\left\{ \hspace{-2pt}\left.X_{t}\right|\mathscr{Y}_{t}\right\} \hspace{-2pt}\left(\mathbb{E}\left\{ \hspace{-2pt}\left.X_{t}\right|\mathscr{Y}_{t}\right\} \right)^{\boldsymbol{T}}\hspace{-2pt}\hspace{-1pt},\hspace{-2pt}\hspace{-2pt}\hspace{-2pt}\hspace{-2pt}\hspace{-2pt}\label{eq:FIRST}
\end{align}
for all $t\in\mathbb{N}$. Next, in a general setting, we consider
asymptotically consistent approximations of $\mathbb{V}\left\{ \hspace{-2pt}\left.\boldsymbol{\phi}_{t+\rho}\left(X_{t+\rho}\right)\right|\mathscr{Y}_{t}\right\} $,
which, at the same time, admit finite dimensional representations.
In the following, $\left\Vert \cdot\right\Vert _{1}^{\mathsf{E}}$
denotes the \textit{entrywise} $\ell_{1}$-norm for matrices, which
upper bounds both the $\ell_{1}$-operator-induced and the Frobenius
norms.
\begin{thm}
\label{Covariance}\textbf{\textup{(Posterior Covariance Recursions)}}
Under the same setting as in Theorem \ref{Functionals}, the posterior
covariance matrix of the optimal filter of the transformed process
$\boldsymbol{\phi}_{t+\rho}\left(X_{t+\rho}\right)$ can be approximated
as
\begin{flalign}
{\cal V}^{L_{S}}\left(\hspace{-2pt}\left.\boldsymbol{\phi}_{t+\rho}\left(X_{t+\rho}\right)\right|\mathscr{Y}_{t}\right) & \triangleq\boldsymbol{\Phi}_{t+\rho}\mbox{\ensuremath{\left[\mathrm{diag}\hspace{-2pt}\left(\dfrac{\boldsymbol{P}^{\rho}\boldsymbol{E}_{t}}{\left\Vert \boldsymbol{E}_{t}\right\Vert _{1}}\right)\hspace{-2pt}-\hspace{-2pt}\dfrac{\boldsymbol{P}^{\rho}\boldsymbol{E}_{t}}{\left\Vert \boldsymbol{E}_{t}\right\Vert _{1}}\hspace{-2pt}\left(\dfrac{\boldsymbol{P}^{\rho}\boldsymbol{E}_{t}}{\left\Vert \boldsymbol{E}_{t}\right\Vert _{1}}\right)^{\boldsymbol{T}}\right]}}\boldsymbol{\Phi}_{t+\rho}^{\boldsymbol{T}},\label{eq:Spatial_2-1}
\end{flalign}
for all $t\hspace{-2pt}\in\hspace{-2pt}\mathbb{N}$. Under the appropriate
assumptions (Lemma \ref{Markovian_Convergence}/\ref{Convergence_FAKE}),
the approximate estimator is asymptotically optimal in the sense of
Theorem \ref{CONVERGENCE_THEOREM}.\end{thm}
\begin{proof}[Proof of Theorem \ref{Covariance}]
See Appendix D.
\end{proof}

\section{Analytical Examples \& Some Simulations}

This section is centered around a discussion about the practical applicability
of the grid based filters under consideration, mainly in regard to
filter implementation, as well as the sufficient conditions for asymptotic
optimality presented and analyzed in Section \ref{sec:CORE_SECTION}.
In what follows, we consider a class of $1$-dimensional (for simplicity),
common and rather practically important \textit{additive} \textit{Nonlinear
AutoRegressions} \textit{(NARs)}, where $X_{t}$ evolves according
to the stochastic difference equation
\begin{equation}
X_{t}\equiv h\left(X_{t-1}\right)+W_{t},\quad\forall t\in\mathbb{N},\label{eq:NAR}
\end{equation}
$X_{-1}\sim{\cal P}_{X_{-1}}$, where $h\left(\cdot\right)$ constitutes
a uniformly bounded and at least continuous nonlinear functional and
$W_{t}$ is a white noise process with known measure. To ensure that
the state is bounded, we will assume that the white noise $W_{t}$
follows, for each $t\in\mathbb{N}$, a zero location (and mean), \textit{truncated
Gaussian distribution} in $\left[-\alpha,\alpha\right]$, with scale
$\sigma$ and with density
\begin{equation}
f_{W}\left(x\right)\triangleq\dfrac{\varphi\left(x/\sigma\right)}{2\sigma\varPhi\left(\alpha/\sigma\right)-\sigma}\mathds{1}_{\left[-\alpha,\alpha\right]}\left(x\right),\;\forall x\in\mathbb{R},
\end{equation}
where $\varphi\left(\cdot\right)$ and $\varPhi\left(\cdot\right)$
denote the standard Gaussian density and cumulative distribution functions,
respectively. Under these considerations, if $\sup_{x\in\mathbb{R}}\left|h\left(x\right)\right|\equiv B$,
then $\left|X_{t}\right|\le B+\alpha$ and, thus, ${\cal Z}$ is identified
as the set $\left[a,b\equiv-a\right]$, with $b\triangleq B+\alpha$.

\subsection{Markovian Filter}

In this case, the respective approximation of the state process is
given by the quantized stochastic difference equation
\begin{equation}
X_{t}^{L_{S}}\triangleq{\cal Q}_{L_{S}}\left(h\left(X_{t-1}^{L_{S}}\right)+W_{t}\right),\quad\forall t\in\mathbb{N},
\end{equation}
initialized as $X_{-1}^{L_{S}}\equiv{\cal Q}_{L_{S}}\left(X_{-1}\right)$,
with probability $1$. In order to guarantee asymptotic optimality
of the respective approximate filter described in Theorem \ref{Markovian_Filter},
the original process $X_{t}$ is required to at least satisfy the
basic Lipschitz condition of Assumption 3. Indeed, \textit{if we merely
assume that $h\left(\cdot\right)$ is additionally Lipschitz} with
constant $L_{h}>0$ (that is, regardless of the stochastic character
of $W_{t}$, in general), then the function
\begin{equation}
f\left(x,y\right)\triangleq h\left(x\right)+y,\quad\left(x,y\right)\in\left[-B,B\right]\times\left[-\alpha,\alpha\right]
\end{equation}
is also Lipschitz with respect to $x$ (for all $y$), with constant
$L_{h}$ as well. Therefore, under the mild Lipschitz assumption for
$h\left(\cdot\right)$, we have shown that the resulting Markovian
filter will indeed be asymptotically consistent. In practice, we expect
that a smaller constant $L_{h}$ would result in better performance
of the approximate filter, with best results if $h\left(\cdot\right)$
constitutes a contraction, which makes $f\left(\cdot,y\right)$ uniformly
contractive in $y$. The above is indeed true, since filtering is
essentially implemented via a stochastic difference equation itself,
and, in general, any discretized approximation to this difference
equation is subject to error accumulation.

Of course, in order for the Markovian filter to be realizable, both
the transition matrix $\boldsymbol{P}$ and the initial value $\boldsymbol{E}_{-1}$
have to be determined. In all cases, under our assumptions, $\boldsymbol{P}$
(and obviously $\boldsymbol{E}_{-1}$) may be determined during an
\textit{offline training phase}, and stored in memory. A brute force
way for estimating $\boldsymbol{P}$ is to simulate $X_{t}$ (recall
that the stochastic description of the transitions of $X_{t}$ is
known apriori). Then, $\boldsymbol{P}$ can be empirically estimated
using the Strong Law of Large Numbers (SLLN). The aforementioned procedure
results in excellent performance in practice \cite{KalPetChannelMarkov2014}.
Exactly the same idea may be employed in order to estimate $\boldsymbol{E}_{-1}$,
given the initial measure of $X_{t}$. Note that the above described
empirical method for the estimation of $\boldsymbol{P}$ and $\boldsymbol{E}_{-1}$
does not assume a specific model describing the temporal evolution
of $X_{t}$, or any particular choice of state quantization. Thus,
it is generally applicable.

However, for the specific (though general) class of systems discussed
above, we may also present an analytical construction for $\boldsymbol{P}$
(and $\boldsymbol{E}_{-1}$, assuming ${\cal P}_{X_{-1}}$ is known),
resulting in compact, closed form expressions. Indeed, by definition
of $X_{t}^{L_{S}}$, $\boldsymbol{P}\left(i,j\right)$ and each ${\cal Z}_{L_{S}}^{i}$,
whose center is $\boldsymbol{x}_{L_{S}}^{i}$, we get
\begin{flalign}
\boldsymbol{P}\left(i,j\right) & \equiv{\cal P}\left(\left.h\left(X_{t-1}^{L_{S}}\right)+W_{t}\in{\cal Z}_{L_{S}}^{i}\right|X_{t-1}^{L_{S}}\equiv\boldsymbol{x}_{L_{S}}^{j}\right)\nonumber \\
 & =\int_{{\cal Z}_{L_{S}}^{i}}\hspace{-2pt}f_{W}\hspace{-2pt}\left(x\hspace{-2pt}-\hspace{-2pt}h\left(\boldsymbol{x}_{L_{S}}^{j}\right)\hspace{-2pt}\right)\hspace{-2pt}\mathrm{d}x,
\end{flalign}
which, based on \eqref{eq:NAR}, can be written in closed form as
\begin{equation}
\boldsymbol{P}\hspace{-2pt}\left(i,j\right)\hspace{-2pt}\equiv\hspace{-2pt}\dfrac{\varPhi\hspace{-2pt}\left(\hspace{-2pt}\dfrac{p_{L_{S}}^{ij}\hspace{-2pt}\left(\alpha,B\right)}{\sigma}\hspace{-2pt}\right)\hspace{-2pt}\hspace{-2pt}-\hspace{-2pt}\varPhi\hspace{-2pt}\left(\hspace{-2pt}\dfrac{q_{L_{S}}^{ij}\hspace{-2pt}\left(\alpha,B\right)}{\sigma}\hspace{-2pt}\right)}{2\sigma\varPhi\left(\alpha/\sigma\right)\hspace{-2pt}-\hspace{-2pt}\sigma}\mathds{1}_{\left(-\infty,p\right)}\hspace{-2pt}\left(q\right),\hspace{-2pt}\hspace{-2pt}\label{eq:Markovian_CLOSED}
\end{equation}
for all $\left(i,j\right)\in\mathbb{N}_{L_{S}}^{+}\times\mathbb{N}_{L_{S}}^{+}$,
where
\begin{flalign}
p_{L_{S}}^{ij}\left(\alpha,B\right) & \triangleq\min\left\{ \alpha,\boldsymbol{x}_{L_{S}}^{i}\hspace{-2pt}-\hspace{-2pt}h\left(\boldsymbol{x}_{L_{S}}^{j}\right)\hspace{-2pt}+\hspace{-2pt}\dfrac{B+\alpha}{L_{S}}\right\} \;\text{and}\\
q_{L_{S}}^{ij}\left(\alpha,B\right) & \triangleq\max\left\{ -\alpha,\boldsymbol{x}_{L_{S}}^{i}\hspace{-2pt}-\hspace{-2pt}h\left(\boldsymbol{x}_{L_{S}}^{j}\right)\hspace{-2pt}-\hspace{-2pt}\dfrac{B+\alpha}{L_{S}}\right\} .
\end{flalign}
Consequently, via \eqref{eq:Markovian_CLOSED}, one may obtain the
whole matrix $\boldsymbol{P}$ for any set of parameters $\sigma,\alpha,B$
and for any resolution $L_{S}$. As far as the initial value $\boldsymbol{E}_{-1}$
is concerned, assuming that the initial measure of $X_{t}$, ${\cal P}_{X_{-1}}$,
is known and recalling that the mapping ${\cal Q}_{L_{S}}^{e}\left(\cdot\right)$
is bijective, it will be true that
\begin{align}
\boldsymbol{E}_{-1} & \equiv\sum_{j\in\mathbb{N}_{L_{S}}^{+}}{\bf e}_{j}^{L_{S}}{\cal P}\left(X_{-1}^{L_{S}}\equiv\boldsymbol{x}_{L_{S}}^{j}\right),\label{eq:E_-1}
\end{align}
where ${\cal P}\left(X_{-1}^{L_{S}}\equiv\boldsymbol{x}_{L_{S}}^{j}\right)\equiv\int_{{\cal Z}_{L_{S}}^{j}}{\cal P}_{X_{-1}}\hspace{-2pt}\left(\mathrm{d}x\right),$
for all $j\in\mathbb{N}_{L_{S}}^{+}$. Thus, $\boldsymbol{E}_{-1}$
can be evaluated in closed form, as long as the aforementioned integrals
can be analytically computed.

\textbf{\vspace{-20pt}
}

\subsection{Marginal Filter}

Marginal filters are, in general, slightly more complicated. However,
at least in theory, they are provably more powerful than Markovian
filters, as the following result suggests.
\begin{thm}
\label{NAR_CR}\textbf{\textup{(Additive NARs are }}\textbf{Almost}\textbf{\textup{
CRT II)}} Let $X_{t}\in\mathbb{R}$ evolve as in \eqref{eq:NAR},
with $X_{-1}\sim{\cal P}_{X_{-1}}$, and where
\begin{itemize}
\item $h\left(\cdot\right)$ is continuous and uniformly bounded by $B>0$.
\item $W_{t}$ follows the truncated Gaussian law in $\left[-\alpha,\alpha\right]$,
$\alpha>0$, with scale zero and location $\sigma>0$.
\end{itemize}
Then, for any quantizer ${\cal Q}_{L_{S}}\left(\cdot\right)$ and
any initial measure ${\cal P}_{X_{-1}}$, $X_{t}$ is \textbf{almost}
conditionally regular, in the sense that
\begin{equation}
\underset{y\in\mathbb{R}}{\mathrm{ess}\hspace{0.2em}\mathrm{sup}}\left|\kappa\hspace{-2pt}\left(\hspace{-1pt}\left.y\right|x\right)\hspace{-2pt}-\hspace{-2pt}\kappa_{t}\hspace{-2pt}\left(\hspace{-1pt}\left.y\right|\hspace{-2pt}\in\hspace{-2pt}{\cal Z}_{L_{S}}\hspace{-2pt}\left(x\right)\right)\right|\hspace{-2pt}\le\hspace{-2pt}\delta_{L_{S}}^{II}\left(x\right)\hspace{-2pt}+\hspace{-2pt}f_{W}\left(\alpha\right),
\end{equation}
for some uniformly bounded, time invariant, nonnegative sequence $\left\{ \delta_{n}^{II}\left(\cdot\right)\right\} _{n\in\mathbb{N}^{+}}$,
converging to zero ${\cal P}_{X_{t}}$-almost everywhere, for all
$t\in\left\{ -1\right\} \cup\mathbb{N}$.\end{thm}
\begin{proof}[Proof of Theorem \ref{NAR_CR}]
See Appendix E.
\end{proof}
As Theorem \ref{NAR_CR} suggests, regardless of the respective initial
measures and without any additional assumptions on the nature of $h\left(\cdot\right)$,
except for continuity, the truncated Gaussian NARs under consideration
are \textit{almost} conditionally regular of type II, in the sense
that the relevant condition on the respective stochastic kernel is
modified by adding \textit{the drift} $f_{W}\left(\alpha\right)$.
In general, this drift parameter might cause error accumulation during
the implementation of the marginal filter. On the other hand though,
it is true that for any fixed scale parameter $\sigma$, $f_{W}\left(\alpha\right)\equiv{\cal O}\hspace{-2pt}\left(\hspace{-2pt}\exp\left(-\alpha^{2}\right)\hspace{-2pt}\right)$.
Thus, for sufficiently large $\alpha$, $f_{W}\left(\alpha\right)$
will not essentially affect filter performance.

Nevertheless, technically, this drift error can vanish, if one considers
a white noise $W_{t}$ following a distribution admitting a finitely
supported and essentially Lipschitz in $\left[-\alpha,\alpha\right]$
density, \textit{taking zero values at} $\pm\alpha$. This is possible
by observing that the proof to Theorem \ref{NAR_CR} in fact works
for such densities, without significant modifications. Then, $f_{W}\left(\pm\alpha\right)\equiv0$
and, hence, the resulting NAR will be CRT II. Such densities exist
and are, in fact, popular; examples are the Logit-Normal and the Raised
Cosine densities, which constitute nice truncated approximations to
the Gaussian density, or more interesting choices, such as the Beta
and Kumarasawmy densities.

Regarding the implementation of the marginal filter, unlike the Markovian
case, closed forms for the elements of $\boldsymbol{P}_{\left(t\right)}$
are very difficult to obtain, because they explicitly depend on the
marginal measures of $X_{t}$, for each $t$, as \eqref{eq:CUM_KERNEL}
suggests. Even if ${\cal P}_{X_{-1}}$ is an invariant measure, implying
that the transition matrix is time invariant, the closed form determination
of $\boldsymbol{P}\left(i,j\right)$ requires proper choice of ${\cal P}_{X_{-1}}$,
which, in most cases, cannot be made by the user. Therefore, in most
cases, $\boldsymbol{P}_{\left(t\right)}$ has to be computed via,
for instance, simulation, and employing the SLLN. As restated above,
this simple technique gives excellent empirical results. Also, assuming
knowledge of the initial measure ${\cal P}_{X_{-1}}$, $\boldsymbol{E}_{-1}$
is again given by \eqref{eq:E_-1}.

In order to demonstrate the applicability of the marginal filter,
as well as empirically evaluate the training-by-simulation technique
advocated above, below we present some additional experimental results
(note that the following also holds for the Markovian filter, under
the appropriate assumptions). As we shall see, these results will
also confirm some aspects of the particular mode of convergence advocated
in Theorem \ref{CONVERGENCE_THEOREM}. Specifically, consider an additive
NAR of the form discussed above, where $h\left(x\right)\equiv\tanh\left(1.3x\right)\in\left(-1,1\right)$,
that is, $B\equiv1$, and where $\alpha\equiv1$ and $\sigma\equiv0.3$.
Additionally, the resulting state process $X_{t}$ is observed via
the nonlinear functional ${\bf y}_{t}\equiv\left[X_{t}\right]^{3}\boldsymbol{1}_{N}+\boldsymbol{w}_{t}$
($\boldsymbol{1}_{n}$ being the $n$-by-$1$ all-ones vector), where
$\boldsymbol{w}_{t}\overset{i.i.d}{\sim}{\cal N}\left(0,\sigma_{\boldsymbol{w}}^{2}{\bf I}_{N}\right)$,
$\sigma_{\boldsymbol{w}}^{2}\equiv2$, for all $t\in\mathbb{N}$.
In order to stress test the marginal approximation approach, we set
${\cal P}_{X_{-1}}\hspace{-2pt}\equiv{\cal U}\left[-2,2\right]$ and
we arbitrarily assume stationarity of $X_{t}$, regardless of ${\cal P}_{X_{-1}}$
being an invariant measure or not. This is a common tactic in practice.
Under this setting, $\boldsymbol{E}_{-1}\equiv L_{S}^{-1}{\bf 1}_{L_{S}}$,
whereas a \textit{single} $\boldsymbol{P}$ is estimated \textit{offline}
from $3\cdot10^{5}$ samples of a \textit{single} simulated version
of $X_{t}$.

As Theorem \ref{CONVERGENCE_THEOREM} suggests, one should be interested
in the approximation error between the approximate and exact filters
of $X_{t}$. However, the exact nonlinear filter of $X_{t}$ is impossible
to compute in a reasonable manner; besides, this is the motive for
developing approximate filters. For that reason, we will further approximate
the approximation error by replacing the optimal filter of $X_{t}$
by a\textit{ particle filter} (an also approximate \textit{global}
method), but employing a very high number of particles. The resampling
step of the particle filter is implemented using \textit{systematic
resampling}, known to minimize Monte Carlo (MC) variation \cite{PARTICLE2002tutorial}.
In our simulations, $5000$ particles are employed in each filtering
iteration.

In the above fashion, Fig. \ref{fig:ERRORS} shows, for each filtering
resolution $L_{S}$, ranging from $2$ to $50$, the \textit{worst}
absolute approximation error, chosen amongst $10$ realizations (MC
trials) of the (approximate) filtering process, where the filtering
horizon was chosen as $T\equiv150$ time steps. The error process
depicted in Fig. \ref{fig:ERRORS} provides a good approximation to
the exact uniform approximation error of \eqref{eq:UAE} in Theorem
\ref{CONVERGENCE_THEOREM}.
\begin{figure}
\centering\includegraphics[scale=0.61]{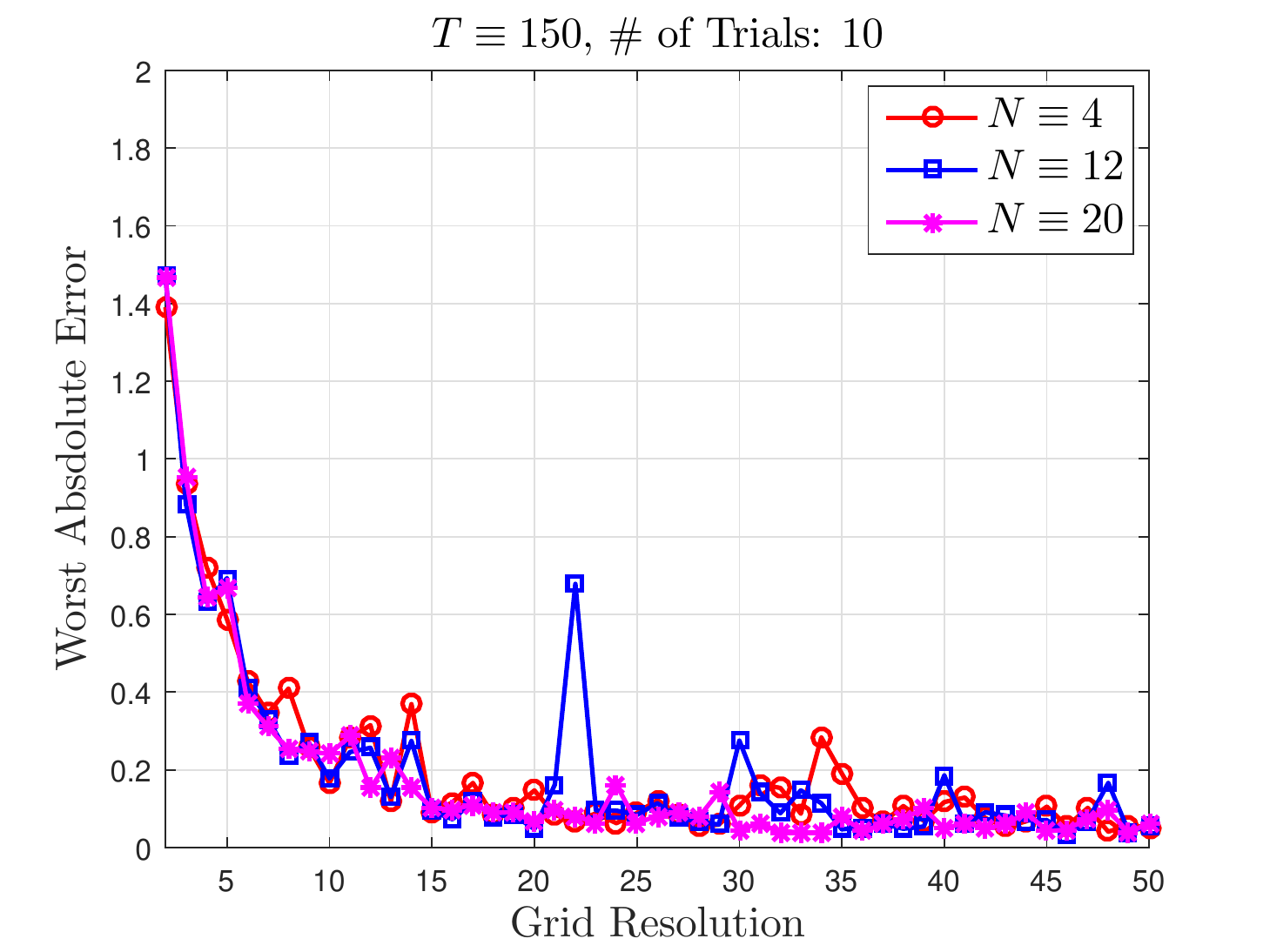}\caption{\label{fig:ERRORS}Marginal filter: worst error with respect to filter
resolution ($L_{S}$), over $10$ trials and for different values
of $N$.}
\vspace{-11pt}
\end{figure}

From the figure, we observe that convergence of the worst approximation
error is confirmed; for all values of $N$, a clear \textit{strictly
decreasing error trend} is identified, as $L_{S}$ increases. This
roughly justifies Theorem \ref{CONVERGENCE_THEOREM}. What is more,
at least for the $10$ realizations collected for each combination
of $N$ and $L_{S}$, the decay of the approximation error is \textit{superstable,
for all values of} $N$. This indicates that, in practice, the realizations
of the approximate filtering process, which will ever be observed
by the user, will be such that convergence to the optimal filter is
indeed uniform, and \textit{almost monotonic} (the ``outliers''
present at $L_{S}\equiv22,30,34$ are most probably due to the use
of a particle filter -a randomized estimator- for emulating the true
filter of $X_{t}$). In the language of Theorem \ref{CONVERGENCE_THEOREM},
it will ``always'' be the case that $\omega\in\widehat{\Omega}_{T}$
(an event occurring with high probability). This in turn implies that,
although general, Theorem \ref{CONVERGENCE_THEOREM} might be somewhat
looser than reality for ``good'' hidden model setups. Finally, another
practically significant detail, which is revealed via Fig. \ref{fig:ERRORS},
and seems to be a common feature of grid based methods, is that the
uniform error bound of the approximate filters \textit{does not increase
as a function of} $N$. Note that this fact cannot be verified via
Theorem \ref{CONVERGENCE_THEOREM}.

In addition to the above, the reader is referred to \cite{KalPetChannelMarkov2014,KalPet_SPAWC_2015},
where complementary simulation results are presented, in the context
of channel estimation in wireless sensor networks.

\section{Conclusion}

We have presented a comprehensive treatment of grid based approximate
nonlinear filtering of discrete time Markov processes observed in
conditionally Gaussian noise, relying on Markovian and marginal approximations
of the state. For the Markovian case, it has been shown that the resulting
approximate filter is strongly asymptotically optimal as long as the
transition mapping of the state is Lipschitz. For the marginal case,
the novel concept of conditional regularity was proposed as a sufficient
condition for ensuring asymptotic optimality. Conditional regularity
is proven to be potentially more relaxed, compared to the state of
the art in grid based filtering, revealing the potential strength
of the grid based approach, and also justifying its good performance
in applications. For both state approximation cases, convergence to
the optimal filter has been proven to be in a strong sense, i.e.,
compact in time and uniform in a fully characterized event occurring
almost certainly. Additionally, typical but important extensions of
our results were discussed and justified. The whole theoretical development
was based on a novel methodological scheme, especially for marginal
state approximations. This focused more on the use of linear-algebraic
techniques and less on measure theoretic arguments, making the presentation
more tangible and easier to grasp. In a companion paper \cite{KalPetChannelMarkov2014},
the results presented herein have been successfully exploited, providing
theoretical guarantees in the context of channel estimation in mobile
wireless sensor networks.

\section*{Appendix A: Proof of Lemma \ref{Markovian_Convergence}}

Consider the event ${\cal E}\triangleq\left\{ \omega\in\Omega\left|X_{t}\left(\omega\right)\in{\cal Z},\ \forall t\in\mathbb{N}\right.\right\} $
of unity probability measure, that is, with ${\cal P}\left({\cal E}\right)\equiv1$.
Of course, by our assumptions so far, ${\cal P}\left({\cal E}^{c}\right)\equiv0$,
with 
\begin{equation}
{\cal E}^{c}\triangleq\left\{ \omega\in\Omega\left|X_{t}\left(\omega\right)\notin{\cal Z},\ \text{for some }t\in\mathbb{N}\right.\right\} 
\end{equation}
being an ``impossible'' measurable set. Then, for $\omega\in{\cal E}$,
we have $X_{t}\left(\omega\right)\in{\cal Z}$ for all $t\in\mathbb{N}$
and we may rewrite \eqref{eq:Approx_Process} as
\begin{equation}
X_{t}^{L_{S}}\left(\omega\right)=f\left(X_{t-1}^{L_{S}}\left(\omega\right),W_{t}\left(\omega\right)\right)+\varepsilon_{t}^{L_{S}}\left(\omega\right),
\end{equation}
for some bounded process $\varepsilon_{t}^{L_{S}}\left(\omega\right)$.
By Assumption 3,
\begin{flalign}
\left\Vert X_{t}^{L_{S}}\left(\omega\right)-X_{t}\left(\omega\right)\right\Vert _{1}\le & K\hspace{-2pt}\left(W_{t}\left(\omega\right)\right)\hspace{-2pt}\left\Vert X_{t-1}^{L_{S}}\hspace{-2pt}\left(\omega\right)\hspace{-2pt}-\hspace{-2pt}X_{t-1}\hspace{-2pt}\left(\omega\right)\right\Vert _{1}\hspace{-2pt}+\hspace{-2pt}\left\Vert \varepsilon_{t}^{L_{S}}\left(\omega\right)\right\Vert _{1}\hspace{-2pt},\label{eq:Big_Inequality}
\end{flalign}
for all $t\in\mathbb{N}$. By construction of the quantizer ${\cal Q}_{L_{S}}\left(\cdot\right)$,
it is easy to show that, for all $\omega\in{\cal E}$, $\left\Vert \varepsilon_{t}^{L_{S}}\left(\omega\right)\right\Vert _{1}\le M\left|b-a\right|/2L_{S}$,
for all $t\in\mathbb{N}$. Then, iterating the right hand side of
\eqref{eq:Big_Inequality} and using induction, it can be easily shown
that
\begin{flalign}
 & \hspace{-2pt}\hspace{-2pt}\hspace{-2pt}\hspace{-2pt}\hspace{-2pt}\hspace{-2pt}\hspace{-2pt}\hspace{-2pt}\hspace{-2pt}\left\Vert X_{t}^{L_{S}}\left(\omega\right)-X_{t}\left(\omega\right)\right\Vert _{1}\nonumber \\
 & \le\left(\prod_{i=0}^{t}K\left(W_{i}\left(\omega\right)\right)\right)\left\Vert X_{-1}^{L_{S}}\left(\omega\right)-X_{-1}\left(\omega\right)\right\Vert _{1}+\dfrac{M\left|b-a\right|}{2L_{S}}\left(1+\sum_{j=1}^{t}\prod_{i=j}^{t}K\left(W_{i}\left(\omega\right)\right)\right),\label{eq:TWO_TERMS}
\end{flalign}
where $X_{-1}^{L_{S}}\left(\omega\right)$ and $X_{-1}\left(\omega\right)$
constitute the initial values of the processes $X_{t}^{L_{S}}\left(\omega\right)$
and $X_{t}\left(\omega\right)$, respectively. Let us focus on the
second term on the RHS of \eqref{eq:TWO_TERMS}. Since, by assumption,
the respective Lipschitz constants are bounded with respect to the
supremum norm in ${\cal E}$ and for all $t\in\mathbb{N}$, it holds
that
\begin{align}
\sum_{j=1}^{t}\prod_{i=j}^{t}K\left(W_{i}\left(\omega\right)\right) & \le\sup_{\omega\in{\cal E}}\sum_{j=1}^{t}\prod_{i=j}^{t}K\left(W_{i}\left(\omega\right)\right)\triangleq\sum_{j=1}^{t}\prod_{i=j}^{t}K\left(W_{i}^{*}\right).\label{eq:SUP_in_omega-1}
\end{align}
Note, however, that the supremum of \eqref{eq:SUP_in_omega-1} in
$t\in\mathbb{N}$ indeed might not be finite. Likewise, regarding
the first term on the RHS of \eqref{eq:TWO_TERMS}, we have
\begin{equation}
\prod_{i=0}^{t}K\hspace{-2pt}\left(W_{i}\left(\omega\right)\right)\le\sup_{\omega\in{\cal E}}\prod_{i=0}^{t}K\hspace{-2pt}\left(W_{i}\left(\omega\right)\right)\triangleq\prod_{i=0}^{t}K\hspace{-2pt}\left(W_{i}^{\star}\right)\hspace{-2pt}.
\end{equation}
As a result, assuming only Lipschitz continuity of $f\left(\cdot,\cdot\right)$
and recalling that $X_{-1}^{L_{S}}\equiv{\cal Q}_{L_{S}}\left(X_{-1}\right)$,
taking the supremum on both sides \eqref{eq:TWO_TERMS} yields
\begin{flalign}
\sup_{\omega\in{\cal E}}\left\Vert X_{t}^{L_{S}}\left(\omega\right)-\hspace{-2pt}X_{t}\left(\omega\right)\right\Vert _{1}\hspace{-2pt} & \equiv\underset{\omega\in\Omega}{\mathrm{ess}\hspace{0.2em}\mathrm{sup}}\left\Vert X_{t}^{L_{S}}\left(\omega\right)-\hspace{-2pt}X_{t}\left(\omega\right)\right\Vert _{1}\nonumber \\
 & \le\hspace{-2pt}\dfrac{M\left|b-a\right|}{2L_{S}}\hspace{-2pt}\left(\hspace{-2pt}1\hspace{-2pt}+\hspace{-2pt}\sum_{j=1}^{t}\prod_{i=j}^{t}\hspace{-2pt}K\hspace{-2pt}\left(W_{i}^{*}\right)\hspace{-2pt}+\hspace{-2pt}\prod_{i=0}^{t}\hspace{-2pt}K\hspace{-2pt}\left(W_{i}^{\star}\right)\hspace{-2pt}\right)\hspace{-2pt}\hspace{-2pt}\underset{L_{S}\rightarrow\infty}{\longrightarrow}\hspace{-2pt}\hspace{-2pt}0,\hspace{-2pt}
\end{flalign}
where the convergence rate may depend on each finite $t$, therefore
only guaranteeing convergence of $X_{t}^{L_{S}}\left(\omega\right)$
in the pointwise sense in $t$ and uniformly almost everywhere in
$\omega$. Now, if $f$$\left(\cdot,\cdot\right)$ is uniformly contractive
for all $\omega\in{\cal E}$, and for all $t\in\mathbb{N}$, then
it will be true that $K\left(W_{t}\left(\omega\right)\right)\in\left[0,1\right)$,
surely in ${\cal E}$ and everywhere in time as well. Consequently,
focusing on the second term on the RHS of \eqref{eq:TWO_TERMS}, it
should be true that 
\begin{flalign}
1+\sum_{j=1}^{t}\prod_{i=j}^{t}K\left(W_{i}\left(\omega\right)\right) & \le1+\sum_{j=1}^{t}\prod_{i=j}^{t}\hspace{0.2em}\sup_{l\in\mathbb{N}}\hspace{0.2em}\sup_{\omega\in{\cal E}}K\left(W_{l}\left(\omega\right)\right)\nonumber \\
 & \triangleq\hspace{-2pt}1+\sum_{j=1}^{t}\prod_{i=j}^{t}K_{*}\hspace{-2pt}\equiv\hspace{-2pt}\sum_{j=0}^{t}K_{*}^{j}\hspace{-2pt}=\hspace{-2pt}\dfrac{1-K_{*}^{t+1}}{1-K_{*}}\hspace{-2pt}\le\hspace{-2pt}\dfrac{1}{1-K_{*}},
\end{flalign}
where $K_{*}\in\left[0,1\right)$ constitutes a global ``Lipschitz
constant'' for $f\left(\cdot,\cdot\right)$ in ${\cal E}$ and for
all $t\in\mathbb{N}$. The situation is of course similar for the
simpler first term on the RHS of \eqref{eq:TWO_TERMS}. As a result,
we readily get that
\begin{align}
\sup_{t\in\mathbb{N}}\hspace{0.2em}\underset{\omega\in\Omega}{\mathrm{ess}\hspace{0.2em}\mathrm{sup}}\left\Vert X_{t}^{L_{S}}\hspace{-2pt}\left(\omega\right)\hspace{-2pt}-\hspace{-2pt}X_{t}\hspace{-2pt}\left(\omega\right)\right\Vert _{1}\hspace{-2pt} & \le\hspace{-2pt}\dfrac{M\left|b\hspace{-2pt}-\hspace{-2pt}a\right|}{2L_{S}}\dfrac{2\hspace{-2pt}-\hspace{-2pt}K_{*}}{1\hspace{-2pt}-\hspace{-2pt}K_{*}}\hspace{-2pt},
\end{align}
where the RHS vanishes as $L_{S}\rightarrow\infty$, thus proving
the second part of the lemma.\hfill{}\ensuremath{\blacksquare}

\section*{Appendix B: Proof of Lemma \ref{Doob_Lemma-1}}

Since the mapping ${\cal Q}_{L_{S}}^{e}\left(\cdot\right)$ is bijective
and using the Markov property of $X_{t}$, it is true that
\begin{flalign}
\mathbb{E}_{\widetilde{{\cal P}}}\left\{ \left.{\cal Q}_{L_{S}}^{e}\left(X_{t}^{L_{S}}\right)\right|\mathscr{X}_{t-1}\right\}  & \equiv\mathbb{E}_{\widetilde{{\cal P}}}\left\{ \left.{\cal Q}_{L_{S}}^{e}\left(X_{t}^{L_{S}}\right)\right|X_{t-1}\right\} \nonumber \\
 & \equiv\sum_{j\in\mathbb{N}_{L_{S}}^{+}}{\bf e}_{j}^{L_{S}}\widetilde{{\cal P}}\left(\left.X_{t}\in{\cal Z}_{L_{S}}^{j}\right|X_{t-1}\right).\label{eq:filters_proof_4}
\end{flalign}

First, let us consider the case where $X_{t}$ is CRT II. Then, assuming
the existence of a stochastic kernel density, it follows that there
is a nonnegative sequence $\left\{ \delta_{L_{S}}^{II}\left(\cdot\right)\right\} _{L_{S}\in\mathbb{N}^{+}}$
converging almost everywhere to $0$ as $L_{S}\rightarrow\infty$,
such that, for all $\boldsymbol{y}\in\mathbb{R}^{M\times1}$, $\kappa\left(\left.\boldsymbol{y}\right|\boldsymbol{x}\right)\le\delta_{L_{S}}^{II}\left(\boldsymbol{x}\right)+\kappa\left(\left.\boldsymbol{y}\right|\hspace{-2pt}\in\hspace{-2pt}{\cal Z}_{L_{S}}\left(\boldsymbol{x}\right)\right)$.
Thus, for each particular choice of $\left(\boldsymbol{y},\boldsymbol{x}\right)$,
there exists a process $\varepsilon_{L_{S}}\left(\boldsymbol{y},\boldsymbol{x}\right)\in\left[-\delta_{L_{S}}^{II}\left(\boldsymbol{x}\right),\delta_{L_{S}}^{II}\left(\boldsymbol{x}\right)\right]$,
such that
\begin{equation}
\kappa\left(\left.\boldsymbol{y}\right|\boldsymbol{x}\right)\equiv\varepsilon_{L_{S}}\left(\boldsymbol{y},\boldsymbol{x}\right)+\kappa\left(\left.\boldsymbol{y}\right|\hspace{-2pt}\in\hspace{-2pt}{\cal Z}_{L_{S}}\left(\boldsymbol{x}\right)\right).
\end{equation}
Consequently, \eqref{eq:filters_proof_4} can be expressed as
\begin{flalign}
\mathbb{E}_{\widetilde{{\cal P}}}\left\{ \hspace{-2pt}\left.{\cal Q}_{L_{S}}^{e}\left(X_{t}^{L_{S}}\right)\right|\mathscr{X}_{t-1}\right\}  & =\sum_{j\in\mathbb{N}_{L_{S}}^{+}}\hspace{-2pt}\hspace{-2pt}{\bf e}_{j}^{L_{S}}\hspace{-2pt}\int_{{\cal Z}_{L_{S}}^{j}}\hspace{-2pt}\hspace{-2pt}\kappa\left(\left.\boldsymbol{x}_{t}\right|X_{t-1}\left(\omega\right)\right)\mathrm{d}\boldsymbol{x}_{t}\nonumber \\
 & =\sum_{j\in\mathbb{N}_{L_{S}}^{+}}{\bf e}_{j}^{L_{S}}\int_{{\cal Z}_{L_{S}}^{j}}\kappa\left(\left.\boldsymbol{x}_{t}\right|\hspace{-2pt}\in\hspace{-2pt}{\cal Z}_{L_{S}}\left(X_{t-1}\left(\omega\right)\right)\right)\mathrm{d}\boldsymbol{x}_{t}+\boldsymbol{\varepsilon}_{t}^{L_{S}}\nonumber \\
 & \equiv\sum_{j\in\mathbb{N}_{L_{S}}^{+}}{\bf e}_{j}^{L_{S}}\widetilde{{\cal P}}\left(\left.X_{t}^{L_{S}}\equiv\boldsymbol{x}_{L_{S}}^{j}\right|X_{t-1}^{L_{S}}\right)+\boldsymbol{\varepsilon}_{t}^{L_{S}},\label{eq:filters_proof_5}
\end{flalign}
where the $\left\{ \mathscr{X}_{t}\right\} $-predictable error process
$\boldsymbol{\varepsilon}_{t}^{L_{S}}\in\mathbb{R}^{L_{S}\times1}$
is defined as\renewcommand{\arraystretch}{1.3}
\begin{equation}
\boldsymbol{\varepsilon}_{t}^{L_{S}}\triangleq\left[\left\{ {\displaystyle \int_{{\cal Z}_{L_{S}}^{j}}\varepsilon_{L_{S}}\left(\boldsymbol{x}_{t},X_{t-1}\right)\mathrm{d}\boldsymbol{x}_{t}}\right\} _{j\in\mathbb{N}_{L_{S}}^{+}}\right]^{\boldsymbol{T}}.
\end{equation}
\renewcommand{\arraystretch}{1}Then, since the state space of ${\cal Q}_{L_{S}}^{e}\left(X_{t}^{L_{S}}\right)$
is finite with cardinality $L_{S}$, we can write
\begin{equation}
\mathbb{E}_{\widetilde{{\cal P}}}\left\{ \left.{\cal Q}_{L_{S}}^{e}\left(X_{t}^{L_{S}}\right)\right|\mathscr{X}_{t-1}\right\} \equiv\boldsymbol{P}{\cal Q}_{L_{S}}^{e}\left(X_{t-1}^{L_{S}}\right)+\boldsymbol{\varepsilon}_{t}^{L_{S}},
\end{equation}
or, equivalently,
\begin{align}
\mathbb{E}_{\widetilde{{\cal P}}}\left\{ \left.{\cal Q}_{L_{S}}^{e}\left(X_{t}^{L_{S}}\right)-\boldsymbol{P}{\cal Q}_{L_{S}}^{e}\left(X_{t-1}^{L_{S}}\right)-\boldsymbol{\varepsilon}_{t}^{L_{S}}\right|\mathscr{X}_{t-1}\right\}  & \triangleq\mathbb{E}_{\widetilde{{\cal P}}}\left\{ \left.\boldsymbol{{\cal M}}_{t}^{e}\right|\mathscr{X}_{t-1}\right\} \equiv0.
\end{align}
As far as the quantity $\left\Vert \boldsymbol{\varepsilon}_{t}^{L_{S}}\right\Vert _{1}$
is concerned, it is true that
\begin{align}
\left\Vert \boldsymbol{\varepsilon}_{t}^{L_{S}}\right\Vert _{1} & \le\sum_{j\in\mathbb{N}_{L_{S}}^{+}}\int_{{\cal Z}_{L_{S}}^{j}}\left|{\displaystyle \varepsilon_{L_{S}}\left(\boldsymbol{x}_{t},X_{t-1}\right)}\right|\mathrm{d}\boldsymbol{x}_{t}\nonumber \\
 & \le\sum_{j\in\mathbb{N}_{L_{S}}^{+}}\int_{{\cal Z}_{L_{S}}^{j}}\delta_{L_{S}}^{II}\left(X_{t-1}\right)\mathrm{d}\boldsymbol{x}_{t}\nonumber \\
 & =\left|b-a\right|^{M}\delta_{L_{S}}^{II}\left(X_{t-1}\right)\underset{L_{S}\rightarrow\infty}{\longrightarrow}0,\quad\widetilde{{\cal P}}-a.s.
\end{align}
and for all $t\in\mathbb{N}$.

For the case where $X_{t}$ constitutes a CRT I process, the situation
is similar. Specifically, \eqref{eq:filters_proof_4} can be expressed
as
\begin{align}
\mathbb{E}_{\widetilde{{\cal P}}}\left\{ \left.{\cal Q}_{L_{S}}^{e}\left(X_{t}^{L_{S}}\right)\right|\mathscr{X}_{t-1}\right\}  & \equiv\sum_{j\in\mathbb{N}_{L_{S}}^{+}}{\bf e}_{j}^{L_{S}}\widetilde{{\cal P}}\left(\left.X_{t}^{L_{S}}\equiv\boldsymbol{x}_{L_{S}}^{j}\right|X_{t-1}^{L_{S}}\right)+\boldsymbol{\varepsilon}_{t}^{L_{S}},\label{eq:filters_proof_6}
\end{align}
where the process $\boldsymbol{\varepsilon}_{t}^{L_{S}}\in\mathbb{R}^{L_{S}\times1}$
is defined similarly to the previous case as\renewcommand{\arraystretch}{1.3}
\begin{equation}
\boldsymbol{\varepsilon}_{t}^{L_{S}}\triangleq\left[\varepsilon_{L_{S}}\left({\cal Z}_{L_{S}}^{1},X_{t-1}\right)\,\ldots\,{\displaystyle \varepsilon_{L_{S}}\left({\cal Z}_{L_{S}}^{L_{S}},X_{t-1}\right)}\right]^{\boldsymbol{T}},
\end{equation}
\renewcommand{\arraystretch}{1}with
\begin{align}
\left\Vert \boldsymbol{\varepsilon}_{t}^{L_{S}}\right\Vert _{1} & \equiv\sum_{j\in\mathbb{N}_{L_{S}}^{+}}\left|\varepsilon_{L_{S}}\left({\cal Z}_{L_{S}}^{j},X_{t-1}\right)\right|\nonumber \\
 & \le\sum_{j\in\mathbb{N}_{L_{S}}^{+}}\dfrac{\delta_{L_{S}}^{I}\left(X_{t-1}\right)}{L_{S}}\equiv\delta_{L_{S}}^{I}\left(X_{t-1}\right)\underset{L_{S}\rightarrow\infty}{\longrightarrow}0,
\end{align}
$\widetilde{{\cal P}}-a.s.$ and for all $t\in\mathbb{N}$. The proof
is complete.\hfill{}\ensuremath{\blacksquare}

\section*{Appendix C: Proof of Lemma \ref{Convergence_FAKE}}

Let us first recall some identifications. First, it can be easily
shown that
\begin{flalign}
\mathbb{E}_{\widetilde{{\cal P}}}\left\{ \left.\Lambda_{t}^{L_{S}}\right|\mathscr{Y}_{t}\right\}  & \equiv\left\Vert \mathbb{E}_{\widetilde{{\cal P}}}\left\{ \left.{\cal Q}_{L_{S}}^{e}\left(X_{t}^{L_{S}}\right)\Lambda_{t}^{L_{S}}\right|\mathscr{Y}_{t}\right\} \right\Vert _{1}\nonumber \\
 & \triangleq\left\Vert \boldsymbol{E}_{t}^{X}\right\Vert _{1},\quad\text{and}\\
\mathbb{E}_{\widetilde{{\cal P}}}\left\{ \left.\Lambda_{t}^{Z,L_{S}}\right|\mathscr{Y}_{t}\right\}  & \equiv\left\Vert \mathbb{E}_{\widetilde{{\cal P}}}\left\{ \left.Z_{t}^{L_{S}}\Lambda_{t}^{Z,L_{S}}\right|\mathscr{Y}_{t}\right\} \right\Vert _{1}\nonumber \\
 & \triangleq\left\Vert \boldsymbol{E}_{t}^{Z}\right\Vert _{1}.
\end{flalign}
Then, we can write\footnote{Here, $\left\Vert {\bf A}\right\Vert _{1}$ denotes the operator norm
induced by the $\ell_{1}$ vector norm.}
\begin{align}
\left\Vert {\cal E}^{L_{S}}\left(\left.X_{t}\right|\mathscr{Y}_{t}\right)-\widetilde{{\cal E}}^{L_{S}}\left(\left.X_{t}\right|\mathscr{Y}_{t}\right)\right\Vert _{1} & \le\left\Vert {\bf X}\right\Vert _{1}\dfrac{\left|\left\Vert \boldsymbol{E}_{t}^{Z}\right\Vert _{1}-\left\Vert \boldsymbol{E}_{t}^{X}\right\Vert _{1}\right|+\left\Vert \boldsymbol{E}_{t}^{X}-\boldsymbol{E}_{t}^{Z}\right\Vert _{1}}{\left\Vert \boldsymbol{E}_{t}^{X}\right\Vert _{1}}.
\end{align}
Since $\left\Vert {\bf X}\right\Vert _{1}\equiv M\max\left\{ \left|a\right|,\left|b\right|\right\} \triangleq M\gamma$
and using the reverse triangle inequality, we get
\begin{equation}
\left\Vert {\cal E}^{L_{S}}\left(\left.X_{t}\right|\mathscr{Y}_{t}\right)\hspace{-2pt}-\hspace{-2pt}\widetilde{{\cal E}}^{L_{S}}\left(\left.X_{t}\right|\mathscr{Y}_{t}\right)\right\Vert _{1}\hspace{-2pt}\le\hspace{-2pt}2M\gamma\dfrac{\left\Vert \boldsymbol{E}_{t}^{X}\hspace{-2pt}-\hspace{-2pt}\boldsymbol{E}_{t}^{Z}\right\Vert _{1}}{\left\Vert \boldsymbol{E}_{t}^{X}\right\Vert _{1}}.\label{eq:filters_proof_8}
\end{equation}

Since $Z_{t}^{L_{S}}$ is a Markov chain, it can be readily shown
that $\boldsymbol{E}_{t}^{Z}$ satisfies the linear recursion $\boldsymbol{E}_{t}^{Z}=\boldsymbol{\Lambda}_{t}\boldsymbol{P}\boldsymbol{E}_{t-1}^{Z}$,
for all $t\in\mathbb{N}$ (also see Theorem \ref{Markovian_Filter}).
Similarly, using the martingale difference type representation given
in Lemma \ref{Doob_Lemma-1}, it easy to show that $\boldsymbol{E}_{t}^{X}$
satisfies another recursion of the form
\begin{equation}
\boldsymbol{E}_{t}^{X}=\boldsymbol{\Lambda}_{t}\boldsymbol{P}\boldsymbol{E}_{t-1}^{X}+\boldsymbol{\Lambda}_{t}\mathbb{E}_{\widetilde{{\cal P}}}\left\{ \left.\boldsymbol{\varepsilon}_{t}^{L_{S}}\Lambda_{t-1}^{L_{S}}\right|\mathscr{Y}_{t-1}\right\} ,\:\forall t\in\mathbb{N}.
\end{equation}
Then, by induction, the error process $\boldsymbol{E}_{t}^{Z}-\boldsymbol{E}_{t}^{X}$
satisfies 
\begin{multline}
\boldsymbol{E}_{t}^{Z}-\boldsymbol{E}_{t}^{X}\\
=\left(\prod_{i\in\mathbb{N}_{t}}\left(\boldsymbol{\Lambda}_{t-i}\boldsymbol{P}\right)\right)\left(\boldsymbol{E}_{-1}^{Z}-\boldsymbol{E}_{-1}^{X}\right)-\sum_{j\in\mathbb{N}_{t}}\left(\prod_{i=0}^{j-1}\left(\boldsymbol{\Lambda}_{t-i}\boldsymbol{P}\right)\right)\boldsymbol{\Lambda}_{t-j}\mathbb{E}_{\widetilde{{\cal P}}}\left\{ \left.\boldsymbol{\varepsilon}_{t-j}^{L_{S}}\Lambda_{t-j-1}^{L_{S}}\right|\mathscr{Y}_{t-j-1}\right\} ,
\end{multline}
for all $t\in\mathbb{N}$. Setting
\begin{equation}
\boldsymbol{E}_{-1}^{Z}\equiv\mathbb{E}_{\widetilde{{\cal P}}}\left\{ Z_{-1}^{L_{S}}\right\} \equiv\mathbb{E}_{{\cal P}}\left\{ {\cal Q}_{L_{S}}^{e}\left(X_{-1}^{L_{S}}\right)\right\} \equiv\boldsymbol{E}_{-1}^{X}
\end{equation}
and taking the $\ell_{1}$-norm of $\boldsymbol{E}_{t}^{Z}-\boldsymbol{E}_{t}^{X}$,
it is true that
\begin{flalign}
\left\Vert \boldsymbol{E}_{t}^{Z}\hspace{-2pt}-\hspace{-2pt}\boldsymbol{E}_{t}^{X}\right\Vert _{1} & \le\hspace{-2pt}\sum_{\tau\in\mathbb{N}_{t}}\hspace{-2pt}\left(\prod_{i=0}^{t-\tau-1}\left\Vert \widehat{\boldsymbol{\Lambda}}_{t-i}\right\Vert _{1}\right)\hspace{-2pt}\left\Vert \widehat{\boldsymbol{\Lambda}}_{\tau}\right\Vert _{1}\mathbb{E}_{\widetilde{{\cal P}}}\hspace{-2pt}\left\{ \left.\widehat{\Lambda}_{\tau-1}^{L_{S}}\left\Vert \boldsymbol{\varepsilon}_{\tau}^{L_{S}}\right\Vert _{1}\right|\mathscr{Y}_{\tau-1}\right\} \nonumber \\
 & \le\hspace{-2pt}\sum_{\tau\in\mathbb{N}_{t}}\hspace{-2pt}\sqrt{\lambda_{inf}^{-N\left(t-\tau+1\right)}}\sqrt{\lambda_{inf}^{-N\tau}}\mathbb{E}_{\widetilde{{\cal P}}}\hspace{-2pt}\left\{ \left.\left\Vert \boldsymbol{\varepsilon}_{\tau}^{L_{S}}\right\Vert _{1}\right|\mathscr{Y}_{\tau-1}\right\} \nonumber \\
 & \equiv\hspace{-2pt}\sqrt{\lambda_{inf}^{-N\left(t+1\right)}}{\displaystyle \sum_{\tau\in\mathbb{N}_{t}}\hspace{-2pt}\mathbb{E}_{\widetilde{{\cal P}}}\hspace{-2pt}\left\{ \left\Vert \boldsymbol{\varepsilon}_{\tau}^{L_{S}}\right\Vert _{1}\right\} },
\end{flalign}
since the process $\boldsymbol{\varepsilon}_{t}^{L_{S}}$ is $\left\{ \mathscr{X}_{t}\right\} $-predictable
and, under $\widetilde{{\cal P}}$, the processes $X_{t}$ and ${\bf y}_{t}$
are statistically independent. 

Now, assuming, for example, that $X_{t}$ is CRT II (the case where
$X_{t}$ is CRT I is similar), we get
\begin{flalign}
\left\Vert \boldsymbol{E}_{t}^{Z}\hspace{-2pt}-\hspace{-2pt}\boldsymbol{E}_{t}^{X}\right\Vert _{1} & \hspace{-2pt}\le\hspace{-2pt}\sqrt{\lambda_{inf}^{-N\left(t+1\right)}}\left|b-a\right|^{M}\hspace{-2pt}\sum_{\tau\in\mathbb{N}_{t}}\hspace{-2pt}\mathbb{E}_{\widetilde{{\cal P}}}\hspace{-2pt}\left\{ {\displaystyle \delta_{L_{S}}^{II}}\hspace{-2pt}\left(X_{\tau-1}\right)\hspace{-2pt}\right\} \nonumber \\
 & \hspace{-2pt}\le\hspace{-2pt}{\displaystyle \dfrac{\left|b-a\right|^{M}}{N\log\left(\lambda_{inf}\right)}\sup_{\tau\in\mathbb{N}_{t}}\mathbb{E}_{\widetilde{{\cal P}}}\left\{ {\displaystyle \delta_{L_{S}}^{II}}\hspace{-2pt}\left(X_{\tau-1}\right)\right\} },
\end{flalign}
$\widetilde{{\cal P}}-a.s.$ and for all $t\in\mathbb{N}$. Regarding
the denominator on the RHS of \eqref{eq:filters_proof_8}, in (\cite{KalPetNonlinear_2015},
last part of the proof of Theorem 3 - Theorem \ref{CONVERGENCE_THEOREM}
in this paper), the authors have shown that, in general, for any fixed
$T<\infty$,
\begin{equation}
\inf_{t\in\mathbb{N}_{T}}\inf_{\omega\in\widehat{\Omega}_{T}}{\displaystyle \inf_{L_{S}\in\mathbb{N}}\mathbb{E}_{\widetilde{{\cal P}}}\left\{ \left.\Lambda_{t}^{L_{S}}\right|\mathscr{Y}_{t}\right\} \left(\omega\right)}>0,
\end{equation}
where $\widehat{\Omega}_{T}\hspace{-2pt}\subseteq\hspace{-2pt}\Omega$
constitutes exactly the same measurable set of Theorem \ref{CONVERGENCE_THEOREM},
occurring with ${\cal P}$-probability at least $1-$ $\left(T\hspace{-2pt}+\hspace{-2pt}1\right)^{1-CN}\hspace{-2pt}\exp\left(-CN\right)$,
for $C\ge1$. Thus, \eqref{eq:filters_proof_8} becomes
\begin{align}
\left\Vert {\cal E}^{L_{S}}\left(\left.X_{t}\right|\mathscr{Y}_{t}\right)-\widetilde{{\cal E}}^{L_{S}}\left(\left.X_{t}\right|\mathscr{Y}_{t}\right)\right\Vert _{1} & \le\dfrac{2M\gamma{\displaystyle \left|b-a\right|^{M}\sup_{\tau\in\mathbb{N}_{t}}\mathbb{E}_{\widetilde{{\cal P}}}\left\{ {\displaystyle \delta_{L_{S}}^{II}}\left(X_{\tau-1}\right)\right\} }}{N\log\left(\lambda_{inf}\right){\displaystyle \inf_{L_{S}\in\mathbb{N}}\left\Vert \boldsymbol{E}_{t}^{X}\right\Vert _{1}}}
\end{align}
and taking the supremum both with respect to $\omega\in\widehat{\Omega}_{T}$
and $t\in\mathbb{N}_{T}$ on both sides, we get
\begin{align}
{\displaystyle \sup_{t\in\mathbb{N}_{T}}}{\displaystyle \sup_{\omega\in\widehat{\Omega}_{T}}}\left\Vert {\cal E}^{L_{S}}\left(\left.X_{t}\right|\mathscr{Y}_{t}\right)-\widetilde{{\cal E}}^{L_{S}}\left(\left.X_{t}\right|\mathscr{Y}_{t}\right)\right\Vert _{1}\left(\omega\right) & \le\dfrac{{\displaystyle 2M\gamma\left|b-a\right|^{M}\sup_{\tau\in\mathbb{N}_{T}}\mathbb{E}_{\widetilde{{\cal P}}}\left\{ {\displaystyle \delta_{L_{S}}^{II}}\left(X_{\tau-1}\right)\right\} }}{N\log\left(\lambda_{inf}\right){\displaystyle \inf_{t\in\mathbb{N}_{T}}\inf_{\omega\in\widehat{\Omega}_{T}}\inf_{L_{S}\in\mathbb{N}}\left\Vert \boldsymbol{E}_{t}^{X}\right\Vert _{1}\left(\omega\right)}}.\label{eq:FINAL_INEQ}
\end{align}
Since the sequence $\left\{ \delta_{L_{S}}^{II}\left(\cdot\right)\right\} _{L_{S}}$
is ${\cal P}_{X_{-1}}\hspace{-2pt}\equiv\hspace{-2pt}{\cal P}_{X_{t}}\hspace{-2pt}-UI$,
it is trivial that the sequence $\left\{ \delta_{L_{S}}^{II}\left(X_{t-1}\left(\cdot\right)\right)\right\} _{L_{S}}$
is uniformly integrable, for all $t\in\mathbb{N}_{T}$. Then, because
$\delta_{L_{S}}^{II}\left(X_{t-1}\left(\cdot\right)\right)\stackrel[L_{S}\rightarrow\infty]{a.e.}{\longrightarrow}0$
(with respect to $\widetilde{{\cal P}}$), Vitali's Convergence Theorem
implies that $\mathbb{E}_{\widetilde{{\cal P}}}\left\{ {\displaystyle \delta_{L_{S}}^{II}}\left(X_{t-1}\right)\right\} \underset{L_{S}\rightarrow\infty}{\longrightarrow}0$,
for all $t\in\mathbb{N}_{T}$, which in turn implies that $\sup_{\tau\in\mathbb{N}_{T}}\mathbb{E}_{\widetilde{{\cal P}}}\left\{ {\displaystyle \delta_{L_{S}}^{II}}\left(X_{\tau-1}\right)\right\} \underset{L_{S}\rightarrow\infty}{\longrightarrow}0$.
Thus, the RHS of \eqref{eq:FINAL_INEQ} converges, and so does its
LHS as well.\hfill{}\ensuremath{\blacksquare}

\section*{Appendix D: Proof of Theorem \ref{Covariance}}

For simplicity and clarity in the exposition, we consider the standard
case where $\boldsymbol{\phi}_{t}\left(X_{t}\right)\equiv X_{t}$,
for all $t\in\mathbb{N}$ and $\rho\equiv1$. Starting with the definitions,
since $\mathbb{V}\hspace{-2pt}\left\{ \hspace{-2pt}\left.X_{t}\right|\mathscr{Y}_{t}\right\} $
is given by \eqref{eq:FIRST}, for all $t\in\mathbb{N}$, it is reasonable
to define the grid based ``filter''
\begin{flalign}
{\cal V}^{L_{S}}\left(\hspace{-2pt}\left.X_{t}\right|\mathscr{Y}_{t}\right) & \triangleq{\cal E}^{L_{S}}\hspace{-2pt}\left(\hspace{-2pt}\left.X_{t}X_{t}^{\boldsymbol{T}}\right|\mathscr{Y}_{t}\right)\hspace{-2pt}-{\cal E}^{L_{S}}\hspace{-2pt}\left(\hspace{-2pt}\left.X_{t}\right|\mathscr{Y}_{t}\right)\hspace{-2pt}\left({\cal E}^{L_{S}}\hspace{-2pt}\left(\hspace{-2pt}\left.X_{t}\right|\mathscr{Y}_{t}\right)\right)^{\boldsymbol{T}}\hspace{-2pt},\label{eq:SECOND}
\end{flalign}
for all $t\in\mathbb{N}$, where ${\cal E}^{L_{S}}\hspace{-2pt}\left(\hspace{-2pt}\left.X_{t}X_{t}^{\boldsymbol{T}}\right|\mathscr{Y}_{t}\right)$
constitutes an \textit{entrywise} operator on the matrix $X_{t}X_{t}^{\boldsymbol{T}}\in\mathbb{R}^{M\times M}$,
defined as 
\begin{flalign}
{\cal E}^{L_{S}}\hspace{-2pt}\left(\hspace{-2pt}\left.X_{t}X_{t}^{\boldsymbol{T}}\right|\mathscr{Y}_{t}\right)\left(i,j\right) & \triangleq\dfrac{1}{\left\Vert \boldsymbol{E}_{t}\right\Vert _{1}}\sum_{l\in\mathbb{N}_{L_{S}}^{+}}\boldsymbol{x}_{L_{S}}^{l}\left(i\right)\boldsymbol{x}_{L_{S}}^{l}\left(j\right)\boldsymbol{E}_{t}\left(l\right)\label{eq:IMPLEMENT}\\
 & \triangleq\dfrac{1}{\left\Vert \boldsymbol{E}_{t}\right\Vert _{1}}\sum_{l\in\mathbb{N}_{L_{S}}^{+}}\boldsymbol{\phi}^{ij}\left(\boldsymbol{x}_{L_{S}}^{l}\right)\boldsymbol{E}_{t}\left(l\right)\triangleq\boldsymbol{\Phi}^{ij}\dfrac{\boldsymbol{E}_{t}}{\left\Vert \boldsymbol{E}_{t}\right\Vert _{1}},
\end{flalign}
for all $\left(i,j\right)\in\mathbb{N}_{M}^{+}\times\mathbb{N}_{M}^{+}$.
In the above, the function(al) $\boldsymbol{\phi}^{ij}:\mathbb{R}^{L_{S}\times1}\rightarrow\mathbb{R}$
is obviously bounded as continuous. Then, making use of the triangle
inequality, the entrywise $\ell_{1}$-norm of ${\cal V}^{L_{S}}\left\{ \hspace{-2pt}\left.X_{t}\right|\mathscr{Y}_{t}\right\} -\mathbb{V}\left\{ \hspace{-2pt}\left.X_{t}\right|\mathscr{Y}_{t}\right\} $
may be bounded from above by the sum of the entrywise $\ell_{1}$-norms
of the differences between the first (Difference 1) and the second
(Difference 2) terms on the RHSs of \eqref{eq:FIRST} and \eqref{eq:SECOND},
respectively. For Difference 1,
\begin{flalign}
\left\Vert {\cal E}^{L_{S}}\hspace{-2pt}\left(\hspace{-2pt}\left.X_{t}X_{t}^{\boldsymbol{T}}\right|\mathscr{Y}_{t}\right)\hspace{-2pt}-\hspace{-2pt}\mathbb{E}\hspace{-2pt}\left\{ \hspace{-2pt}\left.X_{t}X_{t}^{\boldsymbol{T}}\right|\mathscr{Y}_{t}\right\} \right\Vert _{1}^{\mathsf{E}} & \le M^{2}\hspace{-2pt}\sup_{\left(i,j\right)\in\mathbb{N}_{M}^{+}\times\mathbb{N}_{M}^{+}}\left|\mathbb{E}\hspace{-2pt}\left\{ \hspace{-2pt}\left.\boldsymbol{\phi}^{ij}\left(X_{t}\right)\right|\mathscr{Y}_{t}\right\} \hspace{-2pt}-\hspace{-2pt}\boldsymbol{\Phi}^{ij}\dfrac{\boldsymbol{E}_{t}}{\left\Vert \boldsymbol{E}_{t}\right\Vert _{1}}\right|,
\end{flalign}
for all $t\in\mathbb{N}$, where we have exploited the definitions
above and which means that Difference 1 converges to zero as $L_{S}\rightarrow\infty$,
in the sense of Theorem \ref{Functionals}, for any fixed natural
$T<\infty$ and for the same measurable set $\widehat{\Omega}_{T}$
of Theorem \ref{Functionals} (also see Remark \ref{rem:UNIFORMITY}).
For Difference 2, it is easy to show that
\begin{multline}
\left\Vert {\cal E}^{L_{S}}\hspace{-2pt}\left(\hspace{-2pt}\left.X_{t}\right|\mathscr{Y}_{t}\right)\hspace{-2pt}\left({\cal E}^{L_{S}}\hspace{-2pt}\left(\hspace{-2pt}\left.X_{t}\right|\mathscr{Y}_{t}\right)\right)^{\boldsymbol{T}}\hspace{-2pt}\hspace{-2pt}\hspace{-2pt}-\hspace{-2pt}\mathbb{E}\hspace{-2pt}\left\{ \hspace{-2pt}\left.X_{t}\right|\mathscr{Y}_{t}\right\} \hspace{-2pt}\left(\mathbb{E}\hspace{-2pt}\left\{ \hspace{-2pt}\left.X_{t}\right|\mathscr{Y}_{t}\right\} \right)^{\boldsymbol{T}}\right\Vert _{1}^{\mathsf{E}}\\
\le\left(\left\Vert {\cal E}^{L_{S}}\hspace{-2pt}\left(\hspace{-2pt}\left.X_{t}\right|\mathscr{Y}_{t}\right)\right\Vert _{1}+\left\Vert \mathbb{E}\hspace{-2pt}\left\{ \hspace{-2pt}\left.X_{t}\right|\mathscr{Y}_{t}\right\} \right\Vert _{1}\right)\left\Vert {\cal E}^{L_{S}}\hspace{-2pt}\left(\hspace{-2pt}\left.X_{t}\right|\mathscr{Y}_{t}\right)\hspace{-2pt}-\hspace{-2pt}\mathbb{E}\hspace{-2pt}\left\{ \hspace{-2pt}\left.X_{t}\right|\mathscr{Y}_{t}\right\} \right\Vert _{1}\\
\le2M\gamma\left\Vert {\cal E}^{L_{S}}\hspace{-2pt}\left(\hspace{-2pt}\left.X_{t}\right|\mathscr{Y}_{t}\right)\hspace{-2pt}-\hspace{-2pt}\mathbb{E}\hspace{-2pt}\left\{ \hspace{-2pt}\left.X_{t}\right|\mathscr{Y}_{t}\right\} \right\Vert _{1},
\end{multline}
for all $t\in\mathbb{N}$, where we recall that $\gamma\equiv\max\left\{ \left|a\right|,\left|b\right|\right\} $.
Again, Difference 2 converges to zero as $L_{S}\rightarrow\infty$,
exactly in the same sense as Difference 1 above. Consequently, putting
it altogether, we have shown that
\begin{equation}
\sup_{t\in\mathbb{N}_{T}}\sup_{\omega\in\widehat{\Omega}_{T}}\left\Vert {\cal V}^{L_{S}}\left(\hspace{-2pt}\left.X_{t}\right|\mathscr{Y}_{t}\right)-\mathbb{V}\left\{ \hspace{-2pt}\left.X_{t}\right|\mathscr{Y}_{t}\right\} \right\Vert _{1}^{\mathsf{E}}\underset{L_{S}\rightarrow\infty}{\longrightarrow}0,
\end{equation}
proving asymptotic consistency of the approximate estimator.

Now, in order to show that ${\cal V}^{L_{S}}\left(\hspace{-2pt}\left.X_{t}\right|\mathscr{Y}_{t}\right)$
indeed has the form advocated in Theorem \ref{Covariance}, it suffices
to observe that \eqref{eq:IMPLEMENT} in fact coincides with the $\left(i,j\right)$-th
element of the matrix
\begin{equation}
{\bf X}\text{diag}\hspace{-2pt}\left(\dfrac{\boldsymbol{E}_{t}}{\left\Vert \boldsymbol{E}_{t}\right\Vert _{1}}\right)\hspace{-2pt}{\bf X}^{\boldsymbol{T}}\equiv{\cal E}^{L_{S}}\hspace{-2pt}\left(\hspace{-2pt}\left.X_{t}X_{t}^{\boldsymbol{T}}\right|\mathscr{Y}_{t}\right),
\end{equation}
for all $t\in\mathbb{N}$. The proof is now complete.\hfill{}\ensuremath{\blacksquare}

\section*{Appendix E: Proof of Theorem \ref{NAR_CR}}

\begin{figure*}[!b]
\normalsize
\vspace*{-10pt}
\hrulefill
\vspace*{-10pt}

\begin{flalign}
\hspace{-2pt}\hspace{-2pt}\hspace{-2pt}\hspace{-2pt}\hspace{-2pt}\hspace{-2pt}\hspace{-2pt}\hspace{-2pt}\hspace{-2pt}\underset{y\in\mathbb{R}}{\mathrm{ess}\hspace{0.2em}\mathrm{sup}}\hspace{-1pt}\left|\kappa\hspace{-2pt}\left(\hspace{-1pt}\left.y\right|x\right)\hspace{-2pt}-\hspace{-2pt}\kappa_{t}\hspace{-2pt}\left(\hspace{-1pt}\left.y\right|\hspace{-2pt}\in\hspace{-2pt}{\cal Z}_{L_{S}}\hspace{-2pt}\left(x\right)\right)\right| & \hspace{-2pt}\nonumber \\
 & \hspace{-2pt}\hspace{-2pt}\hspace{-2pt}\hspace{-2pt}\hspace{-2pt}\hspace{-2pt}\hspace{-2pt}\hspace{-2pt}\hspace{-2pt}\hspace{-2pt}\hspace{-2pt}\hspace{-2pt}\hspace{-2pt}\hspace{-2pt}\hspace{-2pt}\hspace{-2pt}\hspace{-2pt}\hspace{-2pt}\hspace{-2pt}\hspace{-2pt}\hspace{-2pt}\hspace{-2pt}\hspace{-2pt}\hspace{-2pt}\hspace{-2pt}\hspace{-2pt}\hspace{-2pt}\hspace{-2pt}\hspace{-2pt}\hspace{-2pt}\hspace{-2pt}\hspace{-2pt}\hspace{-2pt}\hspace{-2pt}\hspace{-2pt}\hspace{-2pt}\hspace{-2pt}\hspace{-2pt}\hspace{-2pt}\hspace{-2pt}\hspace{-2pt}\hspace{-2pt}\hspace{-2pt}\hspace{-2pt}\hspace{-2pt}\hspace{-2pt}\hspace{-2pt}\hspace{-2pt}\hspace{-2pt}\hspace{-2pt}\hspace{-2pt}\hspace{-2pt}\hspace{-2pt}\hspace{-2pt}\hspace{-2pt}\hspace{-2pt}\hspace{-2pt}\hspace{-2pt}\hspace{-2pt}\hspace{-2pt}\hspace{-2pt}\hspace{-2pt}\hspace{-2pt}\hspace{-2pt}\hspace{-2pt}\hspace{-2pt}\hspace{-2pt}\hspace{-2pt}\hspace{-2pt}\hspace{-2pt}\hspace{-2pt}\hspace{-2pt}\hspace{-2pt}\hspace{-2pt}\hspace{-2pt}\equiv\hspace{-2pt}\underset{y\in\mathbb{R}}{\mathrm{ess}\hspace{0.2em}\mathrm{sup}}\hspace{-1pt}\left|\kappa\hspace{-2pt}\left(\hspace{-1pt}\left.y\right|x\right)\hspace{-2pt}-\hspace{-2pt}\dfrac{{\displaystyle \int_{\hspace{-1pt}{\cal Z}_{L_{S}}\left(x\right)}\hspace{-2pt}\hspace{-2pt}\kappa\hspace{-2pt}\left(\hspace{-1pt}\left.y\right|\theta\right){\cal P}_{X_{t-1}}\hspace{-2pt}\left(\text{d}\theta\right)}}{{\cal P}\left(X_{t-1}\in{\cal Z}_{L_{S}}\left(x\right)\right)}\right|\hspace{-2pt}\equiv\hspace{-2pt}\underset{y\in\mathbb{R}}{\mathrm{ess}\hspace{0.2em}\mathrm{sup}}\hspace{-1pt}\left|\dfrac{{\displaystyle \int_{\hspace{-1pt}{\cal Z}_{L_{S}}\left(x\right)}\hspace{-2pt}\hspace{-2pt}\kappa\hspace{-2pt}\left(\hspace{-1pt}\left.y\right|x\right)\hspace{-2pt}-\hspace{-2pt}\kappa\hspace{-2pt}\left(\hspace{-1pt}\left.y\right|\theta\right){\cal P}_{X_{t-1}}\hspace{-2pt}\left(\text{d}\theta\right)}}{{\cal P}\left(X_{t-1}\in{\cal Z}_{L_{S}}\left(x\right)\right)}\right|\nonumber \\
 & \hspace{-2pt}\hspace{-2pt}\hspace{-2pt}\hspace{-2pt}\hspace{-2pt}\hspace{-2pt}\hspace{-2pt}\hspace{-2pt}\hspace{-2pt}\hspace{-2pt}\hspace{-2pt}\hspace{-2pt}\hspace{-2pt}\hspace{-2pt}\hspace{-2pt}\hspace{-2pt}\hspace{-2pt}\hspace{-2pt}\hspace{-2pt}\hspace{-2pt}\hspace{-2pt}\hspace{-2pt}\hspace{-2pt}\hspace{-2pt}\hspace{-2pt}\hspace{-2pt}\hspace{-2pt}\hspace{-2pt}\hspace{-2pt}\hspace{-2pt}\hspace{-2pt}\hspace{-2pt}\hspace{-2pt}\hspace{-2pt}\hspace{-2pt}\hspace{-2pt}\hspace{-2pt}\hspace{-2pt}\hspace{-2pt}\hspace{-2pt}\hspace{-2pt}\hspace{-2pt}\hspace{-2pt}\hspace{-2pt}\hspace{-2pt}\hspace{-2pt}\hspace{-2pt}\hspace{-2pt}\hspace{-2pt}\hspace{-2pt}\hspace{-2pt}\hspace{-2pt}\hspace{-2pt}\hspace{-2pt}\hspace{-2pt}\hspace{-2pt}\hspace{-2pt}\hspace{-2pt}\hspace{-2pt}\hspace{-2pt}\hspace{-2pt}\hspace{-2pt}\hspace{-2pt}\hspace{-2pt}\hspace{-2pt}\hspace{-2pt}\hspace{-2pt}\hspace{-2pt}\hspace{-2pt}\hspace{-2pt}\hspace{-2pt}\hspace{-2pt}\hspace{-2pt}\hspace{-2pt}\hspace{-2pt}\le\underset{y,\theta\in{\cal Z}_{L_{S}}\left(x\right)}{\mathrm{ess}\hspace{0.2em}\mathrm{sup}}\left|\kappa\left(\left.y\right|x\right)-\kappa\left(\left.y\right|\theta\right)\right|\equiv\underset{y,\theta\in{\cal Z}_{L_{S}}\left(x\right)}{\mathrm{ess}\hspace{0.2em}\mathrm{sup}}\left|f_{W}\left(y-h\left(x\right)\right)-f_{W}\left(y-h\left(\theta\right)\right)\right|\nonumber \\
 & \hspace{-2pt}\hspace{-2pt}\hspace{-2pt}\hspace{-2pt}\hspace{-2pt}\hspace{-2pt}\hspace{-2pt}\hspace{-2pt}\hspace{-2pt}\hspace{-2pt}\hspace{-2pt}\hspace{-2pt}\hspace{-2pt}\hspace{-2pt}\hspace{-2pt}\hspace{-2pt}\hspace{-2pt}\hspace{-2pt}\hspace{-2pt}\hspace{-2pt}\hspace{-2pt}\hspace{-2pt}\hspace{-2pt}\hspace{-2pt}\hspace{-2pt}\hspace{-2pt}\hspace{-2pt}\hspace{-2pt}\hspace{-2pt}\hspace{-2pt}\hspace{-2pt}\hspace{-2pt}\hspace{-2pt}\hspace{-2pt}\hspace{-2pt}\hspace{-2pt}\hspace{-2pt}\hspace{-2pt}\hspace{-2pt}\hspace{-2pt}\hspace{-2pt}\hspace{-2pt}\hspace{-2pt}\hspace{-2pt}\hspace{-2pt}\hspace{-2pt}\hspace{-2pt}\hspace{-2pt}\hspace{-2pt}\hspace{-2pt}\hspace{-2pt}\hspace{-2pt}\hspace{-2pt}\hspace{-2pt}\hspace{-2pt}\hspace{-2pt}\hspace{-2pt}\hspace{-2pt}\hspace{-2pt}\hspace{-2pt}\hspace{-2pt}\hspace{-2pt}\hspace{-2pt}\hspace{-2pt}\hspace{-2pt}\hspace{-2pt}\hspace{-2pt}\hspace{-2pt}\hspace{-2pt}\hspace{-2pt}\hspace{-2pt}\hspace{-2pt}\hspace{-2pt}\hspace{-2pt}\hspace{-2pt}\le\underset{y,\theta\in{\cal Z}_{L_{S}}\left(x\right)}{\mathrm{ess}\hspace{0.2em}\mathrm{sup}}\dfrac{\left|\varphi\left(\hspace{-2pt}\dfrac{y-h\left(x\right)}{\sigma}\hspace{-2pt}\right)\mathds{1}_{\left[-\alpha,\alpha\right]}\left(y-h\left(x\right)\right)-\varphi\left(\hspace{-2pt}\dfrac{y-h\left(\theta\right)}{\sigma}\hspace{-2pt}\right)\mathds{1}_{\left[-\alpha,\alpha\right]}\left(y-h\left(\theta\right)\right)\right|}{2\sigma\varPhi\left(\alpha/\sigma\right)-\sigma},\nonumber \\
 & \hspace{-2pt}\hspace{-2pt}\hspace{-2pt}\hspace{-2pt}\hspace{-2pt}\hspace{-2pt}\hspace{-2pt}\hspace{-2pt}\hspace{-2pt}\hspace{-2pt}\hspace{-2pt}\hspace{-2pt}\hspace{-2pt}\hspace{-2pt}\hspace{-2pt}\hspace{-2pt}\hspace{-2pt}\hspace{-2pt}\hspace{-2pt}\hspace{-2pt}\hspace{-2pt}\hspace{-2pt}\hspace{-2pt}\hspace{-2pt}\hspace{-2pt}\hspace{-2pt}\hspace{-2pt}\hspace{-2pt}\hspace{-2pt}\hspace{-2pt}\hspace{-2pt}\hspace{-2pt}\hspace{-2pt}\hspace{-2pt}\hspace{-2pt}\hspace{-2pt}\hspace{-2pt}\hspace{-2pt}\hspace{-2pt}\hspace{-2pt}\hspace{-2pt}\hspace{-2pt}\hspace{-2pt}\hspace{-2pt}\hspace{-2pt}\hspace{-2pt}\hspace{-2pt}\hspace{-2pt}\hspace{-2pt}\hspace{-2pt}\hspace{-2pt}\hspace{-2pt}\hspace{-2pt}\hspace{-2pt}\hspace{-2pt}\hspace{-2pt}\hspace{-2pt}\hspace{-2pt}\hspace{-2pt}\hspace{-2pt}\hspace{-2pt}\hspace{-2pt}\hspace{-2pt}\hspace{-2pt}\hspace{-2pt}\hspace{-2pt}\hspace{-2pt}\hspace{-2pt}\hspace{-2pt}\hspace{-2pt}\hspace{-2pt}\hspace{-2pt}\hspace{-2pt}\hspace{-2pt}\hspace{-2pt}\le\hspace{-2pt}\underset{y,\theta\in{\cal Z}_{L_{S}}\left(x\right)}{\mathrm{ess}\hspace{0.2em}\mathrm{sup}}\hspace{-2pt}\left[\vphantom{\dfrac{\min\left\{ \hspace{-2pt}\varphi\hspace{-2pt}\left(\hspace{-2pt}\dfrac{y\hspace{-2pt}-\hspace{-2pt}h\left(x\right)}{\sigma}\hspace{-2pt}\right)\hspace{-2pt},\varphi\hspace{-2pt}\left(\hspace{-2pt}\dfrac{y\hspace{-2pt}-\hspace{-2pt}h\left(\theta\right)}{\sigma}\hspace{-2pt}\right)\hspace{-2pt}\right\} }{2\sigma\varPhi\left(\alpha/\sigma\right)-\sigma}}\right.\dfrac{\min\left\{ \hspace{-2pt}\varphi\hspace{-2pt}\left(\hspace{-2pt}\dfrac{y\hspace{-2pt}-\hspace{-2pt}h\left(x\right)}{\sigma}\hspace{-2pt}\right)\hspace{-2pt},\varphi\hspace{-2pt}\left(\hspace{-2pt}\dfrac{y\hspace{-2pt}-\hspace{-2pt}h\left(\theta\right)}{\sigma}\hspace{-2pt}\right)\hspace{-2pt}\right\} \hspace{-2pt}\left|\mathds{1}_{\left[-\alpha,\alpha\right]}\hspace{-2pt}\left(y\hspace{-2pt}-\hspace{-2pt}h\left(x\right)\right)\hspace{-2pt}-\hspace{-2pt}\mathds{1}_{\left[-\alpha,\alpha\right]}\hspace{-2pt}\left(y\hspace{-2pt}-\hspace{-2pt}h\left(\theta\right)\right)\right|}{2\sigma\varPhi\left(\alpha/\sigma\right)-\sigma}\nonumber \\
 & \quad\quad\quad\quad\quad\quad\quad\quad\quad\quad\quad+\dfrac{\left|\varphi\hspace{-2pt}\left(\hspace{-2pt}\dfrac{y\hspace{-2pt}-\hspace{-2pt}h\left(x\right)}{\sigma}\hspace{-2pt}\right)\hspace{-2pt}\hspace{-2pt}-\hspace{-2pt}\varphi\hspace{-2pt}\left(\hspace{-2pt}\dfrac{y\hspace{-2pt}-\hspace{-2pt}h\left(\theta\right)}{\sigma}\hspace{-2pt}\right)\right|}{2\sigma\varPhi\left(\alpha/\sigma\right)-\sigma}\left.\vphantom{\dfrac{\min\left\{ \hspace{-2pt}\varphi\hspace{-2pt}\left(\hspace{-2pt}\dfrac{y\hspace{-2pt}-\hspace{-2pt}h\left(x\right)}{\sigma}\hspace{-2pt}\right)\hspace{-2pt},\varphi\hspace{-2pt}\left(\hspace{-2pt}\dfrac{y\hspace{-2pt}-\hspace{-2pt}h\left(\theta\right)}{\sigma}\hspace{-2pt}\right)\hspace{-2pt}\right\} }{2\sigma\varPhi\left(\alpha/\sigma\right)-\sigma}}\right]\label{eq:LOOONG-1}
\end{flalign}
\vspace*{-20pt}
%\hrulefill
\end{figure*}

By Definition \ref{Cond_Reg} and the additive model under consideration,
it is obvious that we are interested in CRT II, which, for the case
of an arbitrary initial measure ${\cal P}_{X_{-1}}$, is equivalent
to the strengthened \textit{global} demand that 
\begin{equation}
\underset{\boldsymbol{y}\in\mathbb{R}^{M\times1}}{\mathrm{ess}\hspace{0.2em}\mathrm{sup}}\left|\kappa\left(\left.\boldsymbol{y}\right|\boldsymbol{x}\right)-\kappa_{t}\left(\left.\boldsymbol{y}\right|\hspace{-2pt}\in\hspace{-2pt}{\cal Z}_{L_{S}}\left(\boldsymbol{x}\right)\right)\right|\le\delta_{L_{S},t}^{II}\left(\boldsymbol{x}\right),
\end{equation}
being true ${\cal P}_{X_{t}}\hspace{-2pt}\hspace{-2pt}-\hspace{-1pt}a.e.$,
for some ${\cal P}_{X_{t}}\hspace{-2pt}\hspace{-2pt}-UI$, nonnegative
sequence $\left\{ \delta_{n,t}^{II}\left(\cdot\right)\right\} _{n\in\mathbb{N}^{+}}$,
with $\delta_{n,t}^{II}\left(\cdot\right)\underset{n\rightarrow\infty}{\longrightarrow}0$,
${\cal P}_{X_{t}}\hspace{-2pt}\hspace{-2pt}-\hspace{-2pt}a.e$, \textbf{\textit{for
all}} $t\in\left\{ -1\right\} \cup\mathbb{N}_{T}$, for some desired
$T\in\left[0,\infty\right]$. Of course, $\kappa_{t}\left(\left.\cdot\right|\hspace{-2pt}\in\hspace{-2pt}{\cal Z}_{L_{S}}\left(\cdot\right)\right)$
is defined exactly as in \eqref{eq:CUM_DENSITY}, but with an explicit
subscript ``$t$'', indicating possible temporal variability.

Then, in regard to the additive NAR under consideration and using
the respective definitions, it is true that (see \eqref{eq:LOOONG-1})
\begin{flalign}
\underset{y\in\mathbb{R}}{\mathrm{ess}\hspace{0.2em}\mathrm{sup}}\hspace{-1pt}\left|\kappa\hspace{-2pt}\left(\hspace{-1pt}\left.y\right|x\right)\hspace{-2pt}-\hspace{-2pt}\kappa_{t}\hspace{-2pt}\left(\hspace{-1pt}\left.y\right|\hspace{-2pt}\in\hspace{-2pt}{\cal Z}_{L_{S}}\hspace{-2pt}\left(x\right)\right)\right| & \le\underset{y\in\mathbb{R},\theta\in{\cal Z}_{L_{S}}\left(x\right)}{\mathrm{ess}\hspace{0.2em}\mathrm{sup}}\dfrac{\varphi\left(\dfrac{\alpha}{\sigma}\right)+\left|\varphi\hspace{-2pt}\left(\hspace{-2pt}\dfrac{y\hspace{-2pt}-\hspace{-2pt}h\left(x\right)}{\sigma}\hspace{-2pt}\right)\hspace{-2pt}\hspace{-2pt}-\hspace{-2pt}\varphi\hspace{-2pt}\left(\hspace{-2pt}\dfrac{y\hspace{-2pt}-\hspace{-2pt}h\left(\theta\right)}{\sigma}\hspace{-2pt}\right)\right|}{2\sigma\varPhi\left(\alpha/\sigma\right)-\sigma}\nonumber \\
 & \equiv f_{W}\left(\alpha\right)+\underset{y\in\mathbb{R},\theta\in{\cal Z}_{L_{S}}\left(x\right)}{\mathrm{ess}\hspace{0.2em}\mathrm{sup}}\dfrac{\left|\varphi\hspace{-2pt}\left(\hspace{-2pt}\dfrac{y\hspace{-2pt}-\hspace{-2pt}h\left(x\right)}{\sigma}\hspace{-2pt}\right)\hspace{-2pt}\hspace{-2pt}-\hspace{-2pt}\varphi\hspace{-2pt}\left(\hspace{-2pt}\dfrac{y\hspace{-2pt}-\hspace{-2pt}h\left(\theta\right)}{\sigma}\hspace{-2pt}\right)\right|}{2\sigma\varPhi\left(\alpha/\sigma\right)-\sigma}\nonumber \\
 & \le f_{W}\left(\alpha\right)+\dfrac{{\displaystyle \sup_{\theta\in{\cal Z}_{L_{S}}\left(x\right)}\left|h\left(x\right)-h\left(\theta\right)\right|}}{\left(2\sigma^{2}\varPhi\left(\alpha/\sigma\right)-\sigma^{2}\right)\sqrt{2e\pi}},\:{\cal P}_{X_{t}}\hspace{-2pt}\hspace{-2pt}-\hspace{-1pt}a.e.,\label{eq:PRE_Final}
\end{flalign}
for all $t\in\left\{ -1\right\} \cup\mathbb{N}_{T}$. From \eqref{eq:PRE_Final},
it is almost obvious that $\sup_{\theta\in{\cal Z}_{L_{S}}\left(x\right)}\left|h\left(x\right)-h\left(\theta\right)\right|$
vanishes as $L_{S}\rightarrow\infty$. Indeed, for each fixed $x$,
by definition of ${\cal Z}_{L_{S}}\left(x\right)$, it follows that
\begin{flalign}
{\displaystyle \sup_{\theta\in{\cal Z}_{L_{S}}\left(x\right)}\hspace{-2pt}\left|h\left(x\right)\hspace{-2pt}-\hspace{-2pt}h\left(\theta\right)\right|} & \equiv{\displaystyle \sup_{\left|\theta-{\cal Q}_{L_{S}}\left(x\right)\right|\le\frac{B+\alpha}{L_{S}}}\left|h\left(x\right)\hspace{-2pt}-\hspace{-2pt}h\left(\theta\right)\right|}\nonumber \\
 & \equiv\left|h\left(x\right)\hspace{-2pt}-\hspace{-2pt}h\left(\theta_{L_{S}}^{*}\left(x\right)\right)\right|,
\end{flalign}
where $\theta_{L_{S}}^{*}\left(x\right)\underset{L_{S}\rightarrow\infty}{\longrightarrow}x$,
${\cal P}_{X_{t}}\hspace{-2pt}\hspace{-2pt}-\hspace{-1pt}a.e.$. Thus,
due to the continuity of $h\left(\cdot\right)$, $\sup_{\theta\in{\cal Z}_{L_{S}}\left(x\right)}\left|h\left(x\right)-h\left(\theta\right)\right|\underset{L_{S}\rightarrow\infty}{\longrightarrow}0$,
${\cal P}_{X_{t}}\hspace{-2pt}\hspace{-2pt}-\hspace{-1pt}a.e.$, for
all $t\in\left\{ -1\right\} \cup\mathbb{N}_{T}$. Now, note that $\sup_{\theta\in{\cal Z}_{L_{S}}\left(x\right)}\left|h\left(x\right)-h\left(\theta\right)\right|\le2B$,
set
\begin{equation}
\delta_{L_{S}}^{II}\left(x\right)\triangleq\dfrac{{\displaystyle \sup_{\theta\in{\cal Z}_{L_{S}}\left(x\right)}\left|h\left(x\right)-h\left(\theta\right)\right|}}{\left(2\sigma^{2}\varPhi\left(\alpha/\sigma\right)-\sigma^{2}\right)\sqrt{2e\pi}}
\end{equation}
and choose $T\equiv\infty$. The proof is complete.\hfill{}\ensuremath{\blacksquare}

\section*{Appendix F: Proof of Lemma \ref{P_Equivalence}}

This is a technical proof and requires a deeper appeal to the theoretics
of change of probability measures. Until now, we have made use of
the so called \textbf{\textit{reverse}} \cite{Elliott1994Hidden}
change of measure formula
\begin{equation}
\mathbb{E}_{{\cal P}}\left\{ \left.X_{t}\right|\mathscr{Y}_{t}\right\} \equiv\dfrac{\mathbb{E}_{\widetilde{{\cal P}}}\left\{ \left.X_{t}\Lambda_{t}\right|\mathscr{Y}_{t}\right\} }{\mathbb{E}_{\widetilde{{\cal P}}}\left\{ \left.\Lambda_{t}\right|\mathscr{Y}_{t}\right\} },\quad\forall t\in\mathbb{N}.\label{eq:CoM_basic}
\end{equation}
Formula \eqref{eq:CoM_basic} is characterized as reverse, simply
because it provides a representation for the conditional expectation
of $X_{t}$ under the original base measure ${\cal P}$ via operations
performed exclusively under another auxiliary, hypothetical base measure
$\widetilde{{\cal P}}$. In full generality, the likelihood ratio
process $\Lambda_{t}$ on the RHS of \eqref{eq:CoM_basic} may be
expressed as \cite{KalPetNonlinear_2015} 
\begin{flalign}
\Lambda_{t} & \equiv\dfrac{{\displaystyle \prod_{i\in\mathbb{N}_{t}}}\exp\left(\dfrac{1}{2}\left\Vert {\bf y}_{i}\right\Vert _{2}^{2}-\dfrac{1}{2}\left({\bf y}_{i}-\boldsymbol{\mu}_{i}\left(X_{i}\right)\right)^{\boldsymbol{T}}\left(\boldsymbol{\Sigma}_{i}\left(X_{i}\right)+\sigma_{\xi}^{2}{\bf I}_{N\times N}\right)^{-1}\left({\bf y}_{i}-\boldsymbol{\mu}_{i}\left(X_{i}\right)\right)\right)}{{\displaystyle \prod_{i\in\mathbb{N}_{t}}}\sqrt{\det\left(\boldsymbol{\Sigma}_{i}\left(X_{i}\right)+\sigma_{\xi}^{2}{\bf I}_{N\times N}\right)}}\nonumber \\
 & \equiv{\displaystyle \prod_{i\in\mathbb{N}_{t}}}\dfrac{\sqrt{\left(2\pi\right)^{N}}}{\exp\left(-\dfrac{1}{2}\left\Vert {\bf y}_{i}\right\Vert _{2}^{2}\right)}\dfrac{\exp\left(-\dfrac{1}{2}\left({\bf y}_{i}-\boldsymbol{\mu}_{i}\left(X_{i}\right)\right)^{\boldsymbol{T}}\left(\boldsymbol{\Sigma}_{i}\left(X_{i}\right)+\sigma_{\xi}^{2}{\bf I}_{N\times N}\right)^{-1}\left({\bf y}_{i}-\boldsymbol{\mu}_{i}\left(X_{i}\right)\right)\right)}{\sqrt{\left(2\pi\right)^{N}}\sqrt{\det\left(\boldsymbol{\Sigma}_{i}\left(X_{i}\right)+\sigma_{\xi}^{2}{\bf I}_{N\times N}\right)}}\nonumber \\
 & \equiv{\displaystyle \prod_{i\in\mathbb{N}_{t}}\dfrac{{\cal N}\left({\bf y}_{i};\boldsymbol{\mu}_{i}\left(X_{i}\right),{\bf C}_{i}\left(X_{i}\right)\right)}{{\cal N}\left({\bf y}_{i};{\bf 0},{\bf I}\right)}}\nonumber \\
 & \triangleq{\displaystyle \prod_{i\in\mathbb{N}_{t}}}\mathsf{L}_{i}\left(X_{i},{\bf y}_{i}\right)\in\mathbb{R}_{++},\label{eq:RD}
\end{flalign}
for all $t\in\mathbb{N}$. Also, $\Lambda_{-1}\equiv1$. Note that
we have slightly overloaded the definition of the $\Lambda_{i}$'s
and $\mathsf{L}_{i}$'s, compared to \eqref{eq:Changed_Expectation}.
But this is fine, since the term $\exp\left(-\left\Vert {\bf y}_{t}\right\Vert _{2}^{2}/2\right)$
is $\left\{ \mathscr{Y}_{t}\right\} $-adapted. Here, $\Lambda_{t}$,
\textit{as defined in \eqref{eq:RD}}, is interpreted precisely as
the restriction of the Radon-Nikodym derivative $\text{d}{\cal P}/\text{d}\widetilde{{\cal P}}$
on the filtration $\left\{ \mathscr{H}_{t}\right\} _{t\in\mathbb{N}}$,
generated by both $X_{t}$ (including $X_{-1}$ in $\mathscr{H}_{0}$)
and ${\bf y}_{t}$. That is,
\begin{flalign}
\left.\dfrac{\text{d}{\cal P}}{\text{d}\widetilde{{\cal P}}}\right|_{\mathscr{H}_{t}} & \equiv\Lambda_{t},\quad\forall t\in\mathbb{N}\cup\left\{ -1\right\} \quad\text{with}\\
1 & \equiv\Lambda_{-1}.
\end{flalign}
Observe that, for at $t\in\mathbb{N}$, $\mathscr{Y}_{t}\subset\mathscr{H}_{t}$
and, thus, \eqref{eq:CoM_basic} is a valid expression. In other words,
the Radon-Nikodym Theorem is applied accordingly on the measurable
space $\left(\Omega,\mathscr{H}_{t}\right)$, for each $t\in\mathbb{N}$.

However, because the base measures ${\cal P}$ and $\widetilde{{\cal P}}$
are \textit{equivalent on} $\mathscr{H}_{t}$ (that is, the one is
absolutely continuous with respect to the other), it is possible,
in exactly the same fashion as above, to ``start'' under ${\cal P}$
and express conditional expectations under $\widetilde{{\cal P}}$
via a \textbf{\textit{forward}} change of measure formula. In particular,
it is true that
\begin{equation}
\mathbb{E}_{\widetilde{{\cal P}}}\left\{ \left.X_{t}\right|\mathscr{Y}_{t}\right\} \equiv\dfrac{\mathbb{E}_{{\cal P}}\left\{ \left.X_{t}\Lambda_{t}^{-1}\right|\mathscr{Y}_{t}\right\} }{\mathbb{E}_{{\cal P}}\left\{ \left.\Lambda_{t}^{-1}\right|\mathscr{Y}_{t}\right\} },\quad\forall t\in\mathbb{N}\label{eq:CoM_Forward}
\end{equation}
where, as it is natural, this time we have
\begin{flalign}
\left.\dfrac{\text{d}\widetilde{{\cal P}}}{\text{d}{\cal P}}\right|_{\mathscr{H}_{t}} & \equiv\Lambda_{t}^{-1},\quad\forall t\in\mathbb{N}\cup\left\{ -1\right\} \quad\text{with}\\
1 & \equiv\Lambda_{-1}^{-1}.
\end{flalign}
From the above, one may realize that the ``mechanics'' of the change
of measure procedures (forward and reverse), at least in discrete
time, are very well structured and much simpler than they may initially
seem to be at a first glance. In more generality, it is true that
if $\mathscr{C}_{t}$ is a sub $\sigma$-algebra of $\mathscr{H}_{t}$
and for a $\left\{ \mathscr{H}_{t}\right\} $-adapted process $H_{t}$
\cite{Elliott1994Hidden,Elliott1994Exact,ElliottAggoun2004Measure},
\begin{equation}
\mathbb{E}_{\widetilde{{\cal P}}}\left\{ \left.H_{t}\right|\mathscr{C}_{t}\right\} \equiv\dfrac{\mathbb{E}_{{\cal P}}\left\{ \left.H_{t}\Lambda_{t}^{-1}\right|\mathscr{C}_{t}\right\} }{\mathbb{E}_{{\cal P}}\left\{ \left.\Lambda_{t}^{-1}\right|\mathscr{C}_{t}\right\} },\quad\forall t\in\mathbb{N}.
\end{equation}
And, of course, we can even evaluate (conditional) probabilities under
$\widetilde{{\cal P}}$ as
\begin{equation}
\widetilde{{\cal P}}\left(\left.H_{t}\in{\cal A}\right|\mathscr{C}_{t}\right)\equiv\mathbb{E}_{\widetilde{{\cal P}}}\left\{ \left.\mathds{1}_{\left\{ H_{t}\in{\cal A}\right\} }\right|\mathscr{C}_{t}\right\} \equiv\dfrac{\mathbb{E}_{{\cal P}}\left\{ \left.\mathds{1}_{\left\{ H_{t}\in{\cal A}\right\} }\Lambda_{t}^{-1}\right|\mathscr{C}_{t}\right\} }{\mathbb{E}_{{\cal P}}\left\{ \left.\Lambda_{t}^{-1}\right|\mathscr{C}_{t}\right\} },\quad\forall t\in\mathbb{N},
\end{equation}
for any Borel set ${\cal A}$.

Now, consider the process $X_{t}\equiv f\left(X_{t-1},W_{t}\right),t\in\mathbb{N}$.
As assumed throughout the paper, $X_{t}$ is Markov under ${\cal P}$,
with $W_{t}$ being a white noise (i.i.d.) innovations process. Also,
under $\widetilde{{\cal P}}$, $X_{t}$ is again Markov with exactly
the same dynamics, \textit{but independent of} ${\bf y}_{t}$. However,
at this point nothing is known regarding the nature of $W_{t}$ (distribution,
whiteness) and how it is related to $X_{-1}$ and ${\bf y}_{t}$.
The proof of the remarkable fact that, without any other modification,
$\widetilde{{\cal P}}$ may be chosen such that $W_{t}$ indeed satisfies
the aforementioned properties under question, follows.

\textit{Without changing the respective Radon-Nikodym derivatives
for either the forward or reverse change of measure formulas presented
above}, let us \textit{enlarge} the measurable space for which the
change of measure procedure is valid, by defining $\left\{ \mathscr{H}_{t}\right\} _{t\in\mathbb{N}}$
to be the joint filtration generated by, ${\bf y}_{t}$, the initial
condition $X_{-1}$ and the innovations process $W_{t}$ (Why enlarged?).
Our goal in the following will be to show the following, regarding
the base measure $\widetilde{{\cal P}}$, defined, for each $t\in\mathbb{N}$,
on the enlarged measurable space $\left(\Omega,\mathscr{H}_{t}\right)$:
\begin{enumerate}
\item First, we will show that, under $\widetilde{{\cal P}}$, the observations
process ${\bf y}_{t}$ is mutually independent of both $X_{-1}$ and
$W_{t}$ and therefore also independent of the state $X_{t}$.
\item Second, we will show that, under $\widetilde{{\cal P}}$, $W_{t}$
is white and identically distributed as as under ${\cal P}$ (in addition
to it being independent of ${\bf y}_{t}$ from (1)).
\item Third, we will show that, under $\widetilde{{\cal P}}$, $X_{t}$
is Markov with the same dynamics as under ${\cal P}$ (in addition
to it being independent of ${\bf y}_{t}$ from (1)). 
\end{enumerate}
In order to embark on the rigorous proof of the above, define, for
each $t\in\mathbb{N}$, the auxiliary $\sigma$-algebra $\mathscr{H}_{t}^{-}$,
generated by $\left\{ {\bf y}_{i}\right\} _{i\in\mathbb{N}_{t-1}},$
$X_{-1}$ and $\left\{ W_{i}\right\} _{i\in\mathbb{N}_{t}}$.

\noindent \textbf{1}. For any $\alpha\in\mathbb{R}^{N\times1}$, it
is true that (the ``$\le$'' operator is interpreted in the elementwise
sense)
\begin{flalign}
\widetilde{{\cal P}}\left({\bf y}_{t}\le\alpha\left|\mathscr{H}_{t}^{-}\right.\right) & \equiv\mathbb{E}_{\widetilde{{\cal P}}}\left\{ \left.\mathds{1}_{\left\{ {\bf y}_{t}\le\alpha\right\} }\right|\mathscr{H}_{t}^{-}\right\} \nonumber \\
 & \equiv\dfrac{\mathbb{E}_{{\cal P}}\left\{ \left.\mathds{1}_{\left\{ {\bf y}_{t}\le\alpha\right\} }\Lambda_{t}^{-1}\right|\mathscr{H}_{t}^{-}\right\} }{\mathbb{E}_{{\cal P}}\left\{ \left.\Lambda_{t}^{-1}\right|\mathscr{H}_{t}^{-}\right\} }\nonumber \\
 & =\dfrac{\mathbb{E}_{{\cal P}}\left\{ \left.\mathds{1}_{\left\{ {\bf y}_{t}\le\alpha\right\} }\mathsf{L}_{t}^{-1}\right|\mathscr{H}_{t}^{-}\right\} }{\mathbb{E}_{{\cal P}}\left\{ \left.\mathsf{L}_{t}^{-1}\right|\mathscr{H}_{t}^{-}\right\} },\quad\forall t\in\mathbb{N}.
\end{flalign}
Let us consider the denominator $\mathbb{E}_{{\cal P}}\left\{ \left.\mathsf{L}_{t}^{-1}\right|\mathscr{H}_{t}^{-}\right\} $.
We have
\begin{flalign}
\mathbb{E}_{{\cal P}}\left\{ \left.\mathsf{L}_{t}^{-1}\right|\mathscr{H}_{t}^{-}\right\}  & \equiv\mathbb{E}_{{\cal P}}\left\{ \left.\dfrac{{\cal N}\left({\bf y}_{t};{\bf 0},{\bf I}\right)}{{\cal N}\left({\bf y}_{t};\boldsymbol{\mu}_{t}\left(X_{t}\right),{\bf C}_{t}\left(X_{t}\right)\right)}\right|\mathscr{H}_{t}^{-}\right\} \nonumber \\
 & =\mathbb{E}_{{\cal P}}\left\{ \left.\dfrac{{\cal N}\left(\boldsymbol{\mu}_{t}\left(X_{t}\right)+\sqrt{{\bf C}_{t}\left(X_{t}\right)}\boldsymbol{u}_{t};{\bf 0},{\bf I}\right)}{{\cal N}\left(\boldsymbol{\mu}_{t}\left(X_{t}\right)+\sqrt{{\bf C}_{t}\left(X_{t}\right)}\boldsymbol{u}_{t};\boldsymbol{\mu}_{t}\left(X_{t}\right),{\bf C}_{t}\left(X_{t}\right)\right)}\right|\mathscr{H}_{t}^{-}\right\} ,
\end{flalign}
and given the facts that knowledge of $X_{-1}$ and $\left\{ W_{i}\right\} _{i\in\mathbb{N}_{t}}$
completely determines $\left\{ X_{i}\right\} _{i\in\mathbb{N}_{t}}$
and that the observations are conditionally independent given the
states $\left\{ X_{i}\right\} _{i\in\mathbb{N}_{t}}$, we get
\begin{equation}
\mathbb{E}_{{\cal P}}\left\{ \left.\mathsf{L}_{t}^{-1}\right|\mathscr{H}_{t}^{-}\right\} =\int\dfrac{{\cal N}\left({\bf y}_{t};{\bf 0},{\bf I}\right)}{{\cal N}\left({\bf y}_{t};\boldsymbol{\mu}_{t}\left(X_{t}\right),{\bf C}_{t}\left(X_{t}\right)\right)}{\cal N}\left({\bf y}_{t};\boldsymbol{\mu}_{t}\left(X_{t}\right),{\bf C}_{t}\left(X_{t}\right)\right)\text{d}{\bf y}_{t}\equiv1,\quad\forall t\in\mathbb{N}.
\end{equation}
Likewise, concerning the numerator $\mathbb{E}_{{\cal P}}\left\{ \left.\mathds{1}_{\left\{ {\bf y}_{t}\le\alpha\right\} }\mathsf{L}_{t}^{-1}\right|\mathscr{H}_{t}^{-}\right\} $,
it is true that
\begin{flalign}
\mathbb{E}_{{\cal P}}\left\{ \left.\mathds{1}_{\left\{ {\bf y}_{t}\le\alpha\right\} }\mathsf{L}_{t}^{-1}\right|\mathscr{H}_{t}^{-}\right\}  & \equiv\mathbb{E}_{{\cal P}}\left\{ \left.\mathds{1}_{\left\{ {\bf y}_{t}\le\alpha\right\} }\dfrac{{\cal N}\left({\bf y}_{t};{\bf 0},{\bf I}\right)}{{\cal N}\left({\bf y}_{t};\boldsymbol{\mu}_{t}\left(X_{t}\right),{\bf C}_{t}\left(X_{t}\right)\right)}\right|\mathscr{H}_{t}^{-}\right\} \nonumber \\
 & =\int\dfrac{\mathds{1}_{\left\{ {\bf y}_{t}\le\alpha\right\} }{\cal N}\left({\bf y}_{t};{\bf 0},{\bf I}\right)}{{\cal N}\left({\bf y}_{t};\boldsymbol{\mu}_{t}\left(X_{t}\right),{\bf C}_{t}\left(X_{t}\right)\right)}{\cal N}\left({\bf y}_{t};\boldsymbol{\mu}_{t}\left(X_{t}\right),{\bf C}_{t}\left(X_{t}\right)\right)\text{d}{\bf y}_{t}\nonumber \\
 & \equiv\int\mathds{1}_{\left\{ {\bf y}_{t}\le\alpha\right\} }{\cal N}\left({\bf y}_{t};{\bf 0},{\bf I}\right)\text{d}{\bf y}_{t},
\end{flalign}
or, equivalently,
\begin{equation}
\widetilde{{\cal P}}\left({\bf y}_{t}\le\alpha\left|\mathscr{H}_{t}^{-}\right.\right)\equiv\widetilde{{\cal P}}\left({\bf y}_{t}\le\alpha\right),\quad\forall t\in\mathbb{N}
\end{equation}
and for any $\alpha\in\mathbb{R}^{N\times1}$. Therefore, ${\bf y}_{t}$
is white standard normal under $\widetilde{{\cal P}}$ and, additionally,
mutually independent of $X_{-1}$ and $W_{t}$ and, therefore, mutually
independent of $X_{t}$, too.

\noindent \textbf{2}. Similarly, concerning the innovations process
$W_{t}$, for any $\alpha\in\mathbb{R}^{M_{W}\times1}$, it is true
that
\begin{align}
\widetilde{{\cal P}}\left(W_{t}\le\alpha\left|\mathscr{H}_{t-1}\right.\right) & \equiv\dfrac{\mathbb{E}_{{\cal P}}\left\{ \left.\mathds{1}_{\left\{ W_{t}\le\alpha\right\} }\mathsf{L}_{t}^{-1}\right|\mathscr{H}_{t-1}\right\} }{\mathbb{E}_{{\cal P}}\left\{ \left.\mathsf{L}_{t}^{-1}\right|\mathscr{H}_{t-1}\right\} },\quad\forall t\in\mathbb{N}.
\end{align}
In this case, for the denominator, we again have
\begin{align}
\mathbb{E}_{{\cal P}}\left\{ \left.\mathsf{L}_{t}^{-1}\right|\mathscr{H}_{t-1}\right\}  & \equiv\mathbb{E}_{{\cal P}}\left\{ \left.\dfrac{{\cal N}\left(\boldsymbol{\mu}_{t}\left(X_{t}\right)+\sqrt{{\bf C}_{t}\left(X_{t}\right)}\boldsymbol{u}_{t};{\bf 0},{\bf I}\right)}{{\cal N}\left(\boldsymbol{\mu}_{t}\left(X_{t}\right)+\sqrt{{\bf C}_{t}\left(X_{t}\right)}\boldsymbol{u}_{t};\boldsymbol{\mu}_{t}\left(X_{t}\right),{\bf C}_{t}\left(X_{t}\right)\right)}\right|\mathscr{H}_{t-1}\right\} ,
\end{align}
but because $X_{t}\equiv f\left(X_{t-1},W_{t}\right)$, knowledge
of $X_{-1}$ and $\left\{ W_{i}\right\} _{i\in\mathbb{N}_{t-1}}$
completely determines $\left\{ X_{i}\right\} _{i\in\mathbb{N}_{t-1}}$,
the processes $W_{t}$ and $\boldsymbol{u}_{t}$ are mutually independent
and since the random variable $W_{t}$ is independent of $\left\{ {\bf y}_{i}\right\} _{i\in\mathbb{N}_{t-1}^{+}}$,
we get
\begin{flalign}
\mathbb{E}_{{\cal P}}\left\{ \left.\mathsf{L}_{t}^{-1}\right|\mathscr{H}_{t-1}\right\}  & =\int_{W_{t}}\int_{\boldsymbol{u}_{t}}\dfrac{{\cal N}\left(\boldsymbol{\mu}_{t}\left(X_{t}\right)+\sqrt{{\bf C}_{t}\left(X_{t}\right)}\boldsymbol{u}_{t};{\bf 0},{\bf I}\right)}{{\cal N}\left(\boldsymbol{\mu}_{t}\left(X_{t}\right)+\sqrt{{\bf C}_{t}\left(X_{t}\right)}\boldsymbol{u}_{t};\boldsymbol{\mu}_{t}\left(X_{t}\right),{\bf C}_{t}\left(X_{t}\right)\right)}{\cal N}\left(\boldsymbol{u}_{t};{\bf 0},{\bf I}\right)\text{d}\boldsymbol{u}_{t}{\cal P}_{W_{t}}\left(\text{d}W_{t}\right)\nonumber \\
 & =\int_{W_{t}}\int_{\boldsymbol{u}_{t}}\sqrt{\det\left({\bf C}_{t}\left(X_{t}\right)\right)}{\cal N}\left(\boldsymbol{\mu}_{t}\left(X_{t}\right)+\sqrt{{\bf C}_{t}\left(X_{t}\right)}\boldsymbol{u}_{t};{\bf 0},{\bf I}\right)\text{d}\boldsymbol{u}_{t}{\cal P}_{W_{t}}\left(\text{d}W_{t}\right)\nonumber \\
 & \equiv\int_{W_{t}}\int_{\boldsymbol{u}_{t}}\det\left(\sqrt{{\bf C}_{t}\left(X_{t}\right)}\right){\cal N}\left(\boldsymbol{\mu}_{t}\left(X_{t}\right)+\sqrt{{\bf C}_{t}\left(X_{t}\right)}\boldsymbol{u}_{t};{\bf 0},{\bf I}\right)\text{d}\boldsymbol{u}_{t}{\cal P}_{W_{t}}\left(\text{d}W_{t}\right)\nonumber \\
 & =\int_{W_{t}}\int_{\boldsymbol{u}_{t}}{\cal N}\left(\boldsymbol{\mu}_{t}\left(X_{t}\right)+\sqrt{{\bf C}_{t}\left(X_{t}\right)}\boldsymbol{u}_{t};{\bf 0},{\bf I}\right)\text{d}\left[\boldsymbol{\mu}_{t}\left(X_{t}\right)+\sqrt{{\bf C}_{t}\left(X_{t}\right)}\boldsymbol{u}_{t}\right]{\cal P}_{W_{t}}\left(\text{d}W_{t}\right)\nonumber \\
 & \equiv\int_{W_{t}}{\cal P}_{W_{t}}\left(\text{d}W_{t}\right)\nonumber \\
 & \equiv1.
\end{flalign}
Likewise, the numerator can be expanded as
\begin{flalign}
 & \hspace{-2pt}\hspace{-2pt}\hspace{-2pt}\hspace{-2pt}\hspace{-2pt}\hspace{-2pt}\hspace{-2pt}\hspace{-2pt}\hspace{-2pt}\hspace{-2pt}\hspace{-2pt}\hspace{-2pt}\hspace{-2pt}\hspace{-2pt}\hspace{-2pt}\mathbb{E}_{{\cal P}}\left\{ \left.\mathds{1}_{\left\{ W_{t}\le\alpha\right\} }\mathsf{L}_{t}^{-1}\right|\mathscr{H}_{t-1}\right\} \nonumber \\
 & \equiv\int_{W_{t}}\mathds{1}_{\left\{ W_{t}\le\alpha\right\} }\int_{\boldsymbol{u}_{t}}\dfrac{{\cal N}\left(\boldsymbol{\mu}_{t}\left(X_{t}\right)+\sqrt{{\bf C}_{t}\left(X_{t}\right)}\boldsymbol{u}_{t};{\bf 0},{\bf I}\right){\cal N}\left(\boldsymbol{u}_{t};{\bf 0},{\bf I}\right)}{{\cal N}\left(\boldsymbol{\mu}_{t}\left(X_{t}\right)+\sqrt{{\bf C}_{t}\left(X_{t}\right)}\boldsymbol{u}_{t};\boldsymbol{\mu}_{t}\left(X_{t}\right),{\bf C}_{t}\left(X_{t}\right)\right)}\text{d}\boldsymbol{u}_{t}{\cal P}_{W_{t}}\left(\text{d}W_{t}\right)\nonumber \\
 & \equiv\int_{W_{t}}\mathds{1}_{\left\{ W_{t}\le\alpha\right\} }{\cal P}_{W_{t}}\left(\text{d}W_{t}\right),
\end{flalign}
or, equivalently,
\begin{equation}
\widetilde{{\cal P}}\left(W_{t}\le\alpha\left|\mathscr{H}_{t-1}\right.\right)\equiv\widetilde{{\cal P}}\left(W_{t}\le\alpha\right)\equiv{\cal P}\left(W_{t}\le\alpha\right),\quad\forall t\in\mathbb{N}
\end{equation}
and for any $\alpha\in\mathbb{R}^{N\times1}$. Therefore, $W_{t}$
is white under $\widetilde{{\cal P}}$, in addition to it being independent
of ${\bf y}_{t}$ and with the same distribution as under ${\cal P}$.

\noindent \textbf{3}. It suffices to show that, under $\widetilde{{\cal P}}$,
the initial condition $X_{-1}$ has the same distribution as under
${\cal P}$. If this is true, then, given all the above facts, under
$\widetilde{{\cal P}}$, the process $X_{t}\equiv f\left(X_{t-1},W_{t}\right),t\in\mathbb{N}$
is Markov with the same dynamics as under ${\cal P}$. Indeed, for
any $\alpha\in\mathbb{R}^{M\times1}$, it is trivially true that
\begin{flalign}
\widetilde{{\cal P}}\left(X_{-1}\le\alpha\right) & \equiv\widetilde{{\cal P}}\left(X_{-1}\le\alpha\left|\left\{ \emptyset,\Omega\right\} \right.\right)\nonumber \\
 & =\dfrac{\mathbb{E}_{{\cal P}}\left\{ \mathds{1}_{\left\{ X_{-1}\le\alpha\right\} }\Lambda_{t}^{-1}\right\} }{\mathbb{E}_{{\cal P}}\left\{ \Lambda_{t}^{-1}\right\} },\quad\forall t\in\mathbb{N}\cup\left\{ -1\right\} .
\end{flalign}
Simply, choose $t\equiv-1$. QED.\hfill{}\ensuremath{\blacksquare}

\bibliographystyle{ieeetr}
\bibliography{IEEEabrv}

\end{document}